\documentclass[11pt]{rspublic}
\usepackage{graphicx}
\usepackage{amsmath,amssymb}
\usepackage{amsthm}
\usepackage{bm}
\usepackage[numbers]{natbib}
\usepackage{natbib}
\usepackage{color}
\newcommand{\comments}[1]{}

\newtheorem{innercustomthm}{Theorem}
\newenvironment{customthm}[1]
  {\renewcommand\theinnercustomthm{#1}\innercustomthm}
  {\endinnercustomthm}

  \newtheorem{lem}{Lemma}

\newtheorem{defin}{Definition}[section]

\allowdisplaybreaks

\begin{document}

\title[Linear theory for filtering nonlinear multiscale systems with model error]{Linear theory for filtering nonlinear multiscale systems with model error}

\author[Berry and Harlim]{Tyrus Berry$^1$ and John Harlim$^{1,2}$}
\affiliation{$^1$Department of Mathematics and $^2$Department of Meteorology, The Pennsylvania State University, University Park, PA, 16802., U.S.A.}
\label{firstpage}

\maketitle

\begin{abstract} {data assimilation; filtering; multi-scale systems; covariance inflation; stochastic parameterization; uncertainty quantification; model error; averaging; parameter estimation} 

In this paper, we study filtering of multiscale dynamical systems with model error arising from limitations in resolving the smaller scale processes. In particular, the analysis assumes the availability of continuous-time noisy observations of all components of the slow variables. Mathematically, this paper presents new results on higher-order asymptotic expansion of the first two moments of a conditional measure. In particular, we are interested in the application of filtering multiscale problems in which the conditional distribution is defined over the slow variables, given noisy observation of the slow variables alone.

From the mathematical analysis, we learn that for a continuous time linear model with Gaussian noise, there exists a unique choice of parameters in a linear reduced model for the slow variables which gives the optimal filtering when only the slow variables are observed.  Moreover, these parameters simultaneously give the optimal equilibrium statistical estimates of the underlying system, and as a consequence they can be estimated offline from the equilibrium statistics of the true signal.  By examining a nonlinear test model, we show that the linear theory extends in this non-Gaussian, nonlinear configuration as long as we know the optimal stochastic parameterization and the correct observation model.  However, when the stochastic parameterization model is inappropriate, parameters chosen for good filter performance may give poor equilibrium statistical estimates and vice versa; this finding is based on analytical and numerical results on our nonlinear test model and the two-layer Lorenz-96 model. Finally, even when the correct stochastic ansatz is given, it is imperative to estimate the parameters simultaneously and to account for the nonlinear feedback of the stochastic parameters into the reduced filter estimates. In numerical experiments on the two-layer Lorenz-96 model, we find that the parameters estimated \emph{online}, as part of a filtering procedure, simultaneously produce accurate filtering and equilibrium statistical prediction.  In contrast, an offline estimation technique based on a linear regression, which fits the parameters to a training data set without using the filter, yields filter estimates which are worse than the observations or even divergent when the slow variables are not fully observed. This finding does not imply that all offline methods are inherently inferior to the online method for nonlinear estimation problems, it only suggests that an ideal estimation technique should estimate all parameters simultaneously whether it is online or offline. 
\end{abstract}

\comments{ 
We study filtering of multiscale dynamical systems with model error arising from unresolved smaller scale processes. The analysis assumes continuous-time noisy observations of all components of the slow variables alone. For a linear model with Gaussian noise, we prove existence of a unique choice of parameters in a linear reduced model for the slow variables.  The linear theory extends to to a non-Gaussian, nonlinear test problem, where we assume we know the optimal stochastic parameterization and the correct observation model.  We show that when the parameterization is inappropriate, parameters chosen for good filter performance may give poor equilibrium statistical estimates and vice versa. Given the correct parameterization, it is imperative to estimate the parameters simultaneously and to account for the nonlinear feedback of the stochastic parameters into the reduced filter estimates. In numerical experiments on the two-layer Lorenz-96 model, we find that parameters estimated online, as part of a filtering procedure, produce accurate filtering and equilibrium statistical prediction.  In contrast, a linear regression based offline method, which fits the parameters to a given training data set without using the filter, yields filter estimates which are worse than the observations or even divergent when the slow variables are not fully observed.
}

\maketitle

\section{Introduction}
Model error is a fundamental barrier to state estimation (or filtering). This problem is attributed to incomplete understanding of the underlying physics and our lack of computational resources to resolve physical processes at various time and length scales. While many numerical approaches have been developed to cope with state estimation in the presence of model errors, most of these methods were designed to estimate only one of the model error statistics, either the mean or covariance, while imposing various assumptions on the other statistics which are not estimated. For example, classical approaches proposed in \cite{dds:98,baek:06} estimate mean model error (which is also known as the forecast bias), assuming that the model error covariance (or the random part) is proportional to the prior error covariance from the imperfect model. Popular approaches are to inflate the prior error covariance estimate with an empirically chosen inflation factor \cite{aa:99,ott:04,scyangetal:06,kalnayetal:07,whwst:08} or with an adaptive inflation factor \cite{mehra:70,mehra:72,belanger:74,dcdg:85,anderson:07,lkm:09,miyoshi:11,mh:13,bs:13,hmm:14}. All of these covariance inflation methods assume unbiased forecast error (meaning that there is no mean model error). Recently, reduced stochastic filtering approaches to mitigate model errors in multiscale complex turbulent systems were introduced in \cite{hm:08a,gm:08a,gm:10,ghm:10a,ghm:10b}; see also \cite{mhg:10,mh:12} for a complete treatment of filtering complex turbulent systems. While many of these computationally cheap methods produce relatively accurate mean estimates, the offline based methods such as the mean stochastic model (MSM) \cite{mgy:10,mh:12} tend to underestimate the error covariance statistics that characterizes the uncertainty of the mean estimate in the nonlinear setting. Similar conclusions were also reported in a comparison study of various approximate filtering methods \cite{ls:12}. There are only handful of numerical results, which suggest that an appropriate stochastic parameterization can improve the filtered covariance estimates at short time \cite{mg:12}. Many studies also show that when the stochastic parameters in the filter are obtained by online fitting as part of the data assimilation scheme \cite{gm:10,ghm:10a,ghm:10b,kh:12,mh:13,bs:13,hmm:14}, both the filter mean and covariance estimates become more accurate. These results suggest that one should treat model error as a stochastic process, rather than estimating model error statistics (the bias term and the random component) separately, as is done in many of the empirical approaches mentioned above. 

Independent from the data assimilation context, there is a vast literature in modeling unresolved scale processes with stochastic parameterizations \cite{mtv:01,fev:04,wilks:05,cev:08,mfc:09,amp:13,mh:13,hmm:14}. In principle, these approaches were designed to address the predictability of the equilibrium statistics, with climate modeling as a natural application. We should point out that not only are the forms of the stochastic parameterizations of these methods different, their stochastic parameters are determined by various offline/online data fitting methods. In particular, the approach in \cite{mh:13,hmm:14} determines the stochastic parameters by fitting the data online with a data assimilation scheme. In \cite{hmm:14}, it was shown that it is necessary to use a stochastic parameterization model with at least a one-lag memory to obtain reasonably accurate equilibrium statistical prediction of a highly skewed, non-Gaussian distributed dynamical system. When a memory-less stochastic parameterization is used, the equilibrium statistical prediction for the skewness is constrained to zero even when the true equilibrium distribution is highly skewed. However, the trajectory of the filtered state estimates for the observed variables are comparable and they are relatively accurate, regardless of whether the stochastic parameterization with no-lag or one-lag memory is used. This result suggests that a good reduced stochastic model for filtering may not necessarily be a good model for predicting equilibrium statistics. Here, we will show that the converse is also true when the form of the stochastic parameterization is not chosen appropriately.  

In this paper, we examine the role of the form of the stochastic parameterization and the method of parameter estimation. This issue is closely tied to the above hypothesis which suggests treating model error as a stochastic process in a filtering problem rather than estimating the bias and random components separately, as is typically done in practice. In particular, we want to address the following questions:
\begin{enumerate}
\item Is it possible to have a stochastic parameterization that will produce, simultaneously, optimal filtering and equilibrium statistical prediction in the presence of model error? If so, when can we expect this hypothesis to prevail? 
\item Why is it difficult to find such a stochastic parameterization in practical applications? In particular, what could happen when the appropriate stochastic parameterization ansatz is not available to us?
\item If we have an appropriate stochastic parameterization ansatz, how should we fit the parameters? We will compare the filtering and equilibrium statistical predictive skills of an online parameter estimation scheme with those of a standard linear regression based offline parameter estimation method. By online, we mean parameters are estimated as part of the filtering procedure and by offline, we mean independent of the filter.
\end{enumerate}
To answer the first question, we develop a linear theory for optimal filtering of multiscale dynamical systems with model error arising from limitations in resolving the smaller scale processes. By optimality, we mean the expected state estimate and the error covariance matrix are as accurate as the true posterior estimates obtained with the perfect model. Ideally, we would like to have accurate estimates of all higher-order moments, but due to practical considerations we only discuss the accuracy of the first two moments which are already difficult to obtain beyond the linear and Gaussian setting. Note that this optimality condition is only a minimum requirement for accurate uncertainty quantification. In order to make a rigorous investigation of state estimation in the presence of model error, we consider the following prototype continuous-time filtering problem,
\begin{align}\label{genFullModel}
dx &= f_1(x,y;\theta) dt + \sigma_x(x,y;\theta) \,dW_x, \nonumber \\ 
dy &= \frac{1}{\epsilon}f_2(x,y;\theta) dt + \frac{\sigma_y(x,y;\theta)}{\sqrt{\epsilon}} \,dW_y, \\
dz &= x\,dt + \sqrt{R} dV, \quad R>0.\nonumber
\end{align}
Intuitively, the variable $x$ represents the slow component of the state which we wish to estimate and predict, while the variable $y$ which represents the fast component (characterized by small $\epsilon$) is either unknown or impractical to estimate.  In \eqref{genFullModel}, $W_x, W_y,$ and $V$ are i.i.d. Wiener processes and $\theta$ denotes the true model parameters, which may be partially unknown in real applications. The mathematical analysis in this paper assumes:
\begin{description}
\item{I)} Full observations of only the resolved variables $x$, contaminated by unbiased noise with a positive definite covariance matrix, $R$. For general observation models that involve both the $x$ and $y$ variables, such as those considered in \cite{hm:08a,gm:08a,gm:10}, we recommend that the reader consult the information criteria for optimality of the filtered solutions \cite{bm:14}. While their strategy is more general, our analysis (at least in this simpler context) provides convergence estimates for both the mean and covariance statistics.
\item{II)} The model for the fast unresolved scales in \eqref{genFullModel} are known in order to find the reduced model analytically. In the linear case, we will also discuss how to obtain the reduced model when the fast dynamics in \eqref{genFullModel} are unknown. To make the analysis tractable, our results assume the filtered solutions based on the full model are stable.
\end{description}
While there are many results concerning the convergence of \eqref{genFullModel} as $\epsilon \to 0$ to an averaged reduced filter for $x$ (such as \cite{imkeller:13}, which also developed a nonlinear theory), we are interested in the case where $\epsilon$ may be $\mathcal{O}(10^{-1})$ or even $\mathcal{O}(1)$ and we want to understand the structure of the averaged operators $F(X;\Theta)$ and $\sigma_X(X;\Theta)$ corresponding to the reduced filtering problem,
\begin{align}\label{genReducedModel}
dX &= F(X;\Theta)\, dt + \sigma_{X}(X;\Theta) \,dW_X,\\
dz &= X\,dt + \sqrt{R} dV, \quad R>0.\nonumber
\end{align}
Ultimately, we would like to find $\Theta$ such that the mean and covariance estimates of the reduced filtering problem in \eqref{genReducedModel} are close to the mean and covariance estimates of the true filtering problem with the perfect model in \eqref{genFullModel}. In this reduced filtering problem, the observations $z$ in \eqref{genReducedModel} are noisy observations of the solutions of the true model in \eqref{genFullModel}. We assume that there are no errors in the observation model of the reduced filtering problem, which will allow direct comparison of the filtered estimates from \eqref{genFullModel} and \eqref{genReducedModel}. The parameters $\Theta$ will depend on the scale gap $\epsilon$ and the unknown true dynamics, including the true parameters $\theta$. 

In Section~2, a linear theory is developed in a linear and Gaussian setting under the assumptions I) and II) above. This linear theory will address question~1 above. The results in this section introduce a notion of {\it consistency} as a necessary (but not sufficient) condition for filtering with model error. By consistency condition, we mean the error covariance estimate agrees with the actual error covariance; this motivates us to introduce a weak measure to check whether the filter covariance estimate is under- or over-estimating the actual error covariance when optimal filtering is not available. In Section~3, we study a simple, yet challenging nonlinear problem, where the optimal filter is not available as in practical applications. The ultimate goal is to address the second part of question~1 and question~2. In Section 4, we will compare numerical results of filtering the two-layer Lorenz-96 models with a one-layer Lorenz-96 model combined with various stochastic parameterization methods. The numerical results in this section confirm the theoretical findings in Sections 2 and 3, even for larger discrete observation time intervals and sparsely observed slow variables. Furthermore, these results will suggest a promising method to address question~3. We conclude the paper with a short summary and discussion in Section~5. We also accompany this article with an electronic supplementary material that provides the detailed proofs of the analytical results and a detailed description of the online parameter estimation method.

\section{Linear Theory}\label{lintheory}

The goal in this section is to develop a linear theory for filtering multiscale dynamical systems with model errors.  In the presence of model error, even for a linear system, we must carefully differentiate between the {\it actual error covariance} of the filtered mean estimate and the {\it error covariance estimate} produced by the filtering scheme.  The actual error covariance is simply the expected mean squared error of the state estimate produced by the filter, on the other hand, the linear Kalman-Bucy filter \cite{kalman:61} produces an estimate of error covariance which solves a Riccati equation. In the perfect model scenario, the Kalman-Bucy solutions are optimal and these two error covariances are identical. When the error covariances agree, we say the filter estimate is {\it consistent}. However, when the model used by the filter is not the true model, finding a consistent filter estimate is nontrivial since the covariance solutions of the Riccati equation will typically differ from the actual error of the state estimate. 

In the discussion below, we will first show that there are infinitely many choices of parameters, $\Theta$, for the reduced model in \eqref{genReducedModel}, such that the filter covariance estimate matches the optimal covariance estimate of the true filter in \eqref{genFullModel}. However, most of these parameters will not give accurate estimates of the mean and therefore the covariance estimate will be inconsistent with the actual error covariance. In the context of predictability, information theoretic criteria were advocated to ensure consistent covariance estimates \cite{bm:12}. While in the context of filtering, information theoretic criteria were also suggested for optimizing the filtering skill \cite{bm:14}. In the mathematical analysis below, we will enforce a different criteria which is based on orthogonal projection on Hilbert subspaces (see Theorem 6.1.2 in \cite{oksendal:03}) to find the unique set of reduced filter parameters that ensures not only consistent but also optimal filtering in the sense of least squares. While this is a useful mathematical tool to understand the structure of the stochastic correction in the linear setting, in general, we do not advocate this criteria as a practical tool for parameter estimation. { Moreover, we will show that the same optimal parameters can also be found by matching the equilibrium covariance statistics and posterior covariance estimates or by matching two equilibrium statistics alone.}

Consider a linear model where $f_1 = a_{11}x + a_{12}y$ and $f_2 = a_{21}x + a_{22}y$ with a linear observation which involves only the slow variable, $x$.  For this particular case the full filtering problem in \eqref{genFullModel} becomes
\begin{align}\label{fullmodel}
d x &= (a_{11}x + a_{12}y)\,dt + \sigma_x d W_x, \nonumber \\
d y &= \frac{1}{\epsilon}({a_{21}x} + {a_{22}y})\,dt + \frac{\sigma_y}{\sqrt{\epsilon}}d W_y,\\
d z & = x \, dt + \sqrt{R} \,dV = H(x,y)^\top\,dt + \sqrt{R} \,dV,\nonumber
\end{align} 
where we define observation operator $H = (1,0)$ for convenience.
We assume that the matrix $A=(a_{ij})$ is negative definite and $\sigma_x, \sigma_y>0$ are constants of $\mathcal{O}(1)$. We also assume that $\tilde{a}= a_{11}-a_{12}a_{22}^{-1}a_{21}<0$, which guarantees the existence of the averaged dynamics in \eqref{genReducedModel} for $\epsilon\rightarrow 0$; in this case, $F(X)=\tilde{a}X$ and $\sigma_{X} = \sigma_x$ (see e.g., \cite{gh:13} for detailed derivation).  

\subsection{Expansion of the Optimal Filter}

For the continuous time linear filtering problem in \eqref{fullmodel}, the optimal filter estimates (in the sense of minimum variance estimator), are the first and second order statistics of a Gaussian posterior distribution that can be completely characterized by the Kalman-Bucy solutions \cite{kalman:61}. For this linear and Gaussian filtering problem, the covariance solutions of the filter will converge to a steady state covariance matrix $\hat{S}=\{\hat{s}_{ij}\}_{i,j=1,2}$, which solves the following algebraic Riccati equation, 
\begin{align}\label{riccati}
A_\epsilon \hat S + \hat SA_\epsilon^\top - \hat SH^\top R^{-1}H\hat S +Q_\epsilon = 0,
\end{align}
where,
\begin{align}
A _{\epsilon} = \left( \begin{array}{cc} a_{11} & a_{12} \\ a_{21}/\epsilon & a_{22}/\epsilon \end{array}\right), \hspace{20pt}\textup{ }\hspace{20pt} Q_{\epsilon} = \left(\begin{array}{cc} \sigma_x^2 & 0\\ 0 &\sigma_y^2/\epsilon \end{array}\right).\nonumber
\end{align}

We can rewrite the first diagonal component of the algebraic Riccati equation \eqref{riccati} for $\hat{s}_{11}:= \mathbb{E}((x-\hat{x})^2)$ as follows (see Appendix A in the electronic supplementary material):
\begin{align}\label{s11}
 -\frac{\hat{s}_{11}^2}{R} 
+ 2\tilde a \left(1 - \epsilon\hat a \right) \hat{s}_{11} 
+  \sigma_x^2(1-2\epsilon\hat a) + \epsilon \sigma_y^2 \frac{a_{12}^2}{a_{22}^2}   = \mathcal{O}(\epsilon^2) 
\end{align}
where $\tilde a = a_{11} - \frac{a_{12}a_{21}}{a_{22}}$ and $\hat a = \frac{a_{12}a_{21}}{a_{22}^2}$. 

Our goal is to find a one-dimensional model for the slow variable, $x$, which still gives the optimal state estimate.  Motivated by the results in \cite{gh:13}, and the fact that \eqref{s11} has the form of a one-dimensional Riccati equation, we consider the following one-dimensional linear filtering problem,
\begin{align}
dX &=  aX\,dt+\sigma_X \,dW_X, \label{reducedmodel}\\
dz &= X\,dt + \sqrt{R}\, dV.\nonumber\end{align}
The corresponding steady state covariance solution for the reduced filter in \eqref{reducedmodel} satisfies the following algebraic Riccati equation, 
\begin{align}\label{riccati1d1} -\frac{\tilde{s}^2}{R} + 2a\tilde{s} + \sigma_X^2 = 0.\end{align}
Substracting equation \eqref{s11} from \eqref{riccati1d1}, we have the following result (see the detailed proof in Appendix A in the electronic supplementary material),

\begin{theorem}\label{thm1}
Let $\hat{s}_{11}$ be the first diagonal component of the algebraic Riccati equation in \eqref{riccati} and let $\tilde{s}$ be the solution of \eqref{riccati1d1}. Then $\lim_{\epsilon\to 0} \frac{\tilde{s}-\hat{s}_{11}}{\epsilon} = 0$ if and only if 
\begin{align} \sigma_X^2 = -2(a-\tilde a(1-\epsilon\hat a))\hat{s}_{11} + \sigma_x^2(1-2\epsilon\hat a) + \epsilon\sigma_y^2\frac{a_{12}^2}{a_{22}^2}  + \mathcal{O}(\epsilon^2).\label{manifold} \end{align}
\end{theorem}
Theorem \ref{thm1} says that there is a manifold of parameters $\Theta=\{a, \sigma_X\}$ for which the steady-state filter covariance estimate $\tilde{s}$ produced by the reduced model agrees with the steady-state covariance estimate of the optimal filter $\hat{s}_{11}$, obtained with perfect model. So, for any parameters on the manifold \eqref{manifold}, the reduced filter mean estimate solves,
\begin{align}
d\tilde{x} = a\tilde{x}\,dt + \frac{\tilde s}{R}(dz-\tilde{x}\,dt),\label{reducedestimate}
\end{align}
while the true filter mean estimate for $x$-variable solves,
\begin{align}
d \hat{x} = H A_\epsilon (\hat{x},\hat{y})^\top\,dt + \frac{\hat{s}_{11}}{R}(dz-\hat{x}\,dt).\label{truestimate}
\end{align}
While the true filter estimate in \eqref{truestimate} is consistent, meaning that $\hat{s}_{11}=\mathbb{E}[(x-\hat{x})^2]$, as shown in the derivation of the Kalman-Bucy equations \cite{kalman:61}, the reduced filter estimate $\tilde{x}$ from \eqref{reducedestimate} is not always consistent in the presence of model error.  Notice that the actual steady state error covariance, $E_{11}=\lim_{t\rightarrow\infty}\mathbb{E}[e(t)^2]$, where $e(t)=x(t)-\tilde x(t)$, is not necessarily equal to the steady state filter covariance estimate $\tilde{s} = \hat{s}_{11} +\mathcal{O}(\epsilon^2)$. In fact, most choices of parameters on the manifold in \eqref{manifold} lead to poor filter performance, despite the optimality of $\tilde{s}$ (in the sense of minimum variance), due to the inconsistency of the reduced filter.

\begin{figure}
\centering
\includegraphics[width=.45\textwidth]{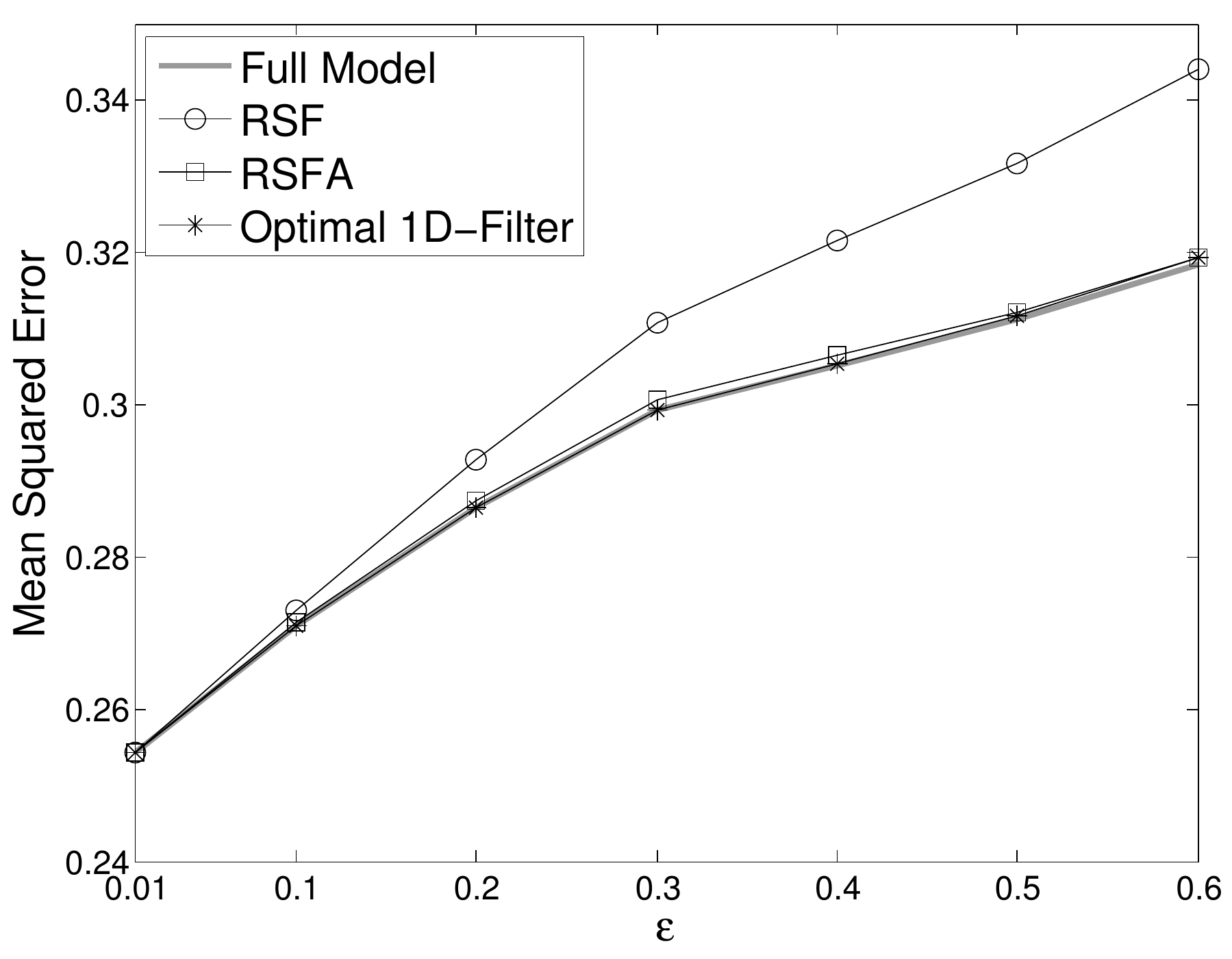}
\includegraphics[width=0.45\textwidth]{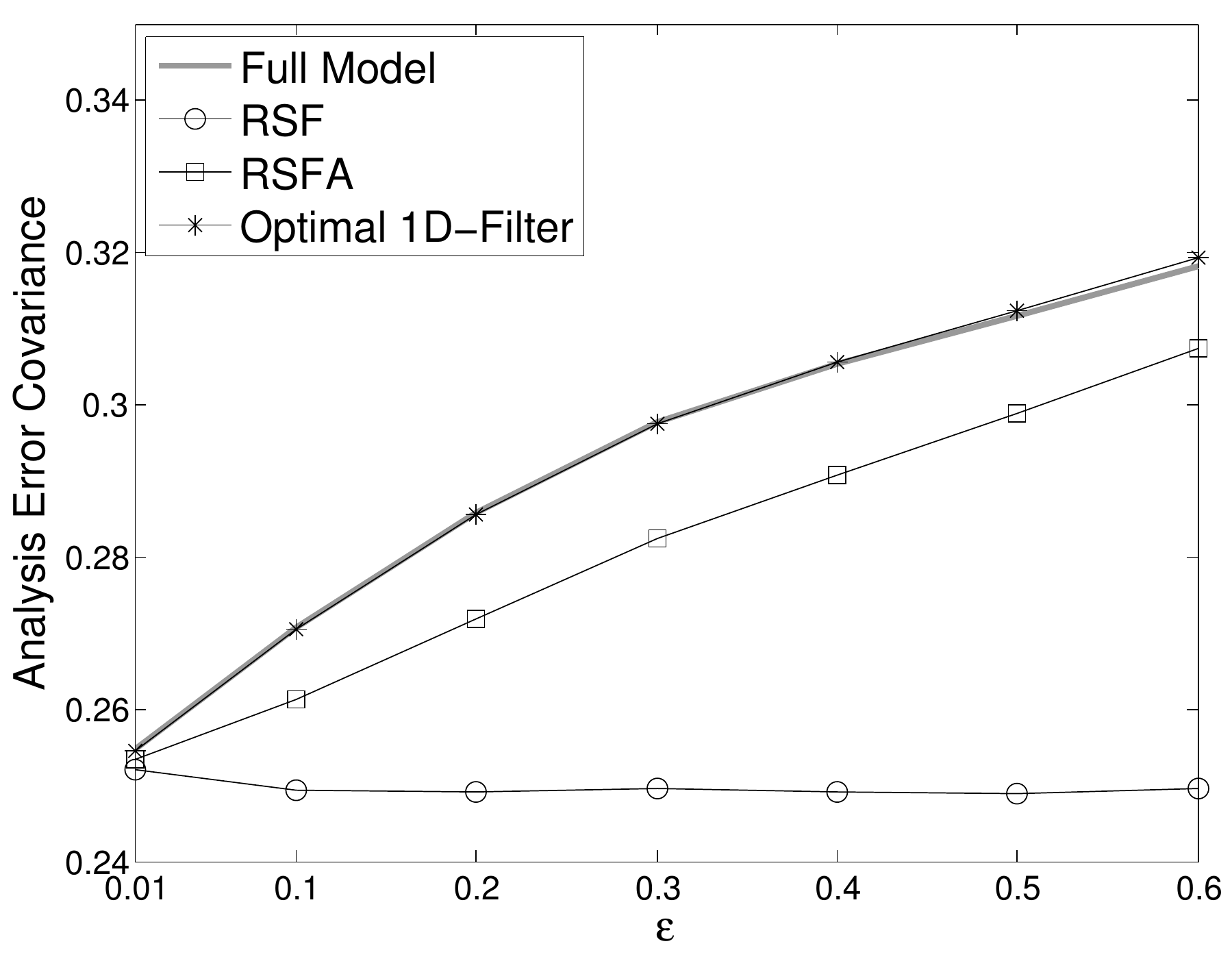}
\caption{\label{linearEx} Mean squared error of the filtered mean (left) and covariance estimates (right) for an observation series produced by \eqref{fullmodel}.  The filter uses either the full model \eqref{fullmodel} or various reduced models given by parameter choices in \eqref{reducedmodel}.  RSF: $a=\tilde a$ and $\sigma_X = \sigma_x$; RSFA: $a=\tilde a$ and $\sigma_X^2 = \sigma_x^2 + \epsilon\sigma_y^2 a_{12}^2/a_{22}^2$; optimal one-dimensional filter: $a = \tilde a(1-\epsilon \hat a)$ and $\sigma_X^2 = \sigma_x^2(1-2\epsilon\hat a) + \epsilon\sigma_y^2 a_{12}^2/a_{22}^2$.  The observation noise covariance is $R=0.5$ and observations are at time interval $\Delta t = 1$. Results are averaged over 100,000 assimilation cycles.}
\end{figure}

Our goal is to specify the parameters such that the filtered solutions are consistent, $E_{11} = \tilde{s} + \mathcal{O}(\epsilon^2)$. Unfortunately, this consistency condition is too weak and only specifies the choice of parameters up to order-$\epsilon$. From the general linear theory of Hilbert spaces, the optimal filter mean estimate in the sense of least squares is given by the orthogonal projection onto the subspace spanned by its innovations (see Theorem 6.1.2 and the discussion in Section 6.2 in \cite{oksendal:03}). This condition implies that the actual error, $e=x-\tilde{x}$ is orthogonal to the estimate $\tilde{x}$ under the joint probability distribution for $(W_X,V)$, that is, $\mathbb{E}(e\,\tilde{x})=0$.

By requiring the reduced filter estimates to satisfy $\mathbb{E}(e\,\tilde{x})=0$, we find a unique choice of parameters $\Theta=\{a,\sigma_X\}$ on the manifold in \eqref{manifold} which produces optimal filter solutions  (see Appendix B in the electronic supplementary material for the detailed proof of Theorem \ref{thm2}). 
To obtain these parameters, we apply the following procedure: { We write the Lyapunov equation for an augmented state variable, $(x,y,\tilde{x})^T$ and find the steady state solution for $\mathbb{E}(e\,\tilde{x})$ up to order-$\epsilon^2$.
Then we enforce the condition, $\mathbb{E}(e\,\tilde{x})=0$, which yields a unique choice of parameters on the manifold in \eqref{manifold}.  
Furthermore, we can also use the steady solutions of the same Lyapunov equation to verify that these parameters guarantee consistent filtered solutions, $E_{11}=\tilde{s}+\mathcal{O}(\epsilon^2)$. In fact, the same parameters can be obtained by requiring the variance of the reduced model in \eqref{reducedmodel} to match the equilibrium variance of the underlying system in \eqref{fullmodel} for variable $x$, in additional to the manifold in \eqref{manifold}.} These results are summarized in the following theorem.

\begin{theorem}\label{thm2} 
There exists a unique choice of parameters given by $a = \tilde a(1-\epsilon \hat a)$ and $\sigma_X^2$ according to Theorem \ref{thm1}, such that the steady state reduced filter \eqref{reducedmodel} is both consistent and optimal up to order-$\epsilon^2$.  This means that $\tilde s$, the steady state covariance estimate of the reduced filter, is consistent with the steady state actual error covariance $E_{11} = \lim_{t\to\infty}\mathbb{E}[(x(t)-\tilde x(t))^2]$ so that $\tilde s = E_{11} + \mathcal{O}(\epsilon^2)$, and also $\tilde s$ agrees with the steady state covariance $\hat s_{11}$ from the optimal filter $\tilde s = \hat s_{11} + \mathcal{O}(\epsilon^2)$. { The unique optimal parameters can also be determined by requiring the covariance of the reduced model to match that of the slow variable from the full model up to order-$\epsilon^2$.}
\end{theorem}
{

 We remark that a result of \cite{fz:11} shows that for $a=\tilde a + \mathcal{O}(\epsilon)$ and $\sigma_X^2 = \sigma_x^2 +\mathcal{O}(\epsilon)$, the reduced filter mean and covariance estimates are uniformly optimal for all time in the following sense: Given identical initial statistics, $\hat{x}(0)=\tilde{x}(0), \hat{s}_{11}(0)=\tilde{s}(0)>0$, there are time-independent constants $C$, such that, $\mathbb{E}\Big(|\hat{x}(t)-\tilde{x}(t)|^2\Big) \leq C\epsilon^2$.
 In fact, we conjecture that the pathwise convergence should be, 
 \[ \mathbb{E}\Big(|\hat{x}(t)-\tilde{x}(t)|^2\Big) \leq C\epsilon^4, \]
 for the unique parameters from Theorem \ref{thm2} and we confirm this conjecture numerically in Appendix B.  However, the proof of this would require solving the Lyapunov equation of the five-dimensional joint evolution of the full model, full filter, and reduced filter.  Since this Lyapunov equation is an algebraic system of 15 equations of 15 variables it is not illuminating to verify our conjecture analytically.}


Comparing this result to the reduced stochastic filter with an additive noise correction (RSFA) computed in \cite{gh:13}, Theorem~\ref{thm2} imposes additional order-$\epsilon$ corrections in the form of linear damping, $-\epsilon \tilde a\hat a \tilde{x}$, and additive stochastic forcing, $-2\epsilon\sigma_x^2\hat{a}\,dW_x$. 
This additive noise correction term was also found in the formal asymptotic derivation of \cite{gh:13} (they denoted the covariance estimate associated with this additive noise correction by $Q_2$), but the absence of the order-$\epsilon$ linear damping correction term in their calculation makes it impossible to match the posterior statistics of the full model to the same level of accuracy. They dropped this additional additive noise term and, subsequently, underestimated the true error covariance (as shown in Figure \ref{linearEx}). We now verify the accuracy of the filter covariance estimate suggested by Theorem~\ref{thm2} in the numerical simulation described below. 

In Figure \ref{linearEx}, we show numerical results comparing the true filter using the perfect model with approximate filter solutions based on three different one-dimensional reduced models of the form \eqref{reducedmodel}. Here, the model parameters are $a_{11}=a_{21}=a_{22}=-1, a_{12}=1, \sigma_x^2=\sigma_y^2=2$. The numerical experiments are for discrete time observations at $\Delta t = 1$ with observation noise covariance $R = 0.5$ and the dynamics are solved analytically between observations.  The three reduced models include: (1) the simple averaging model (RSF) where $a = \tilde a$ and $\sigma_X^2 = \sigma_x^2$; (2) the order-$\epsilon$ reduced model (RSFA) introduced in \cite{gh:13} with $a = \tilde a$ and $\sigma_X^2 = \sigma_x^2 + \epsilon \sigma_y^2 a_{12}^2/a_{22}^2$; and (3) the order-$\epsilon^2$ optimal reduced filter described in Theorem \ref{thm2}.  Notice that only the order-$\epsilon^2$ optimal reduced filter produces mean and covariance estimates that match the true filter solutions. Furthermore, the resulting covariance estimate is consistent, that is, the mean square error, $\tilde{E}_{11}:=\langle (x-\tilde x)^2 \rangle$, where $\langle\cdot\rangle$ denotes temporal average (which equals $E_{11}$ for ergodic posterior distribution) matches the asymptotic covariance estimates $\tilde s$ (compare the starred data points in the left and the right panels in Figure~\ref{linearEx}).  

{ In this linear and Gaussian example, we found the optimal stochastic reduced model either by applying an asymptotic expansion to the Kalman-Bucy solutions alone or by applying asymptotic expansion to both the model equilibrium covariance and the filter posterior covariance solutions. In fact, we will show in the next section that the same reduced model can be obtained by applying an asymptotic expansion to the equilibrium statistical solutions of the model alone. We note that the higher-order expansion of the filter solution does not require a pathwise expansion of the prior model.}


\subsection{An optimal stochastic parameter estimation method for filtering linear problems}

In practical applications, one may have no access to the true dynamics in \eqref{fullmodel} and in this case it is necessary to estimate the parameters in the reduced model in \eqref{reducedmodel} to obtain the optimal filtered solutions. Ideally, we would like to be able find the optimal parameters using some limited information about the marginal statistics of the slow variable, $x$.  For the linear SDE in \eqref{reducedmodel} (which is also known as the Ornstein-Uhlenbeck process \cite{gardiner:97}), the two parameters, namely the linear damping coefficient, $a$, and the noise amplitude $\sigma_X$, can be characterized by two equilibrium statistics, variance and correlation time. Theorem~\ref{thm2} guarantees that in the linear and Gaussian setting, one can obtain an optimal filtering by specifying the model parameters from these two equilibrium statistics. This parameter estimation strategy was introduced as the \emph{Mean Stochastic Model (MSM)} in \cite{mgy:10} (see also \cite{mh:12} for different stochastic parameterization strategies for the linear SDE in \eqref{reducedmodel}). Formally, we have:

\begin{corollary}\label{thm4} Given the equilibrium variance and the correlation time statistics of the true signal, $x(t)$, that evolves based on the linear SDE in \eqref{fullmodel}, the reduced Mean Stochastic Model (MSM) filter is an optimal filter, in the sense that the posterior mean and covariance estimates differ by an order of $\epsilon^2$ from the true filter estimates obtained with the perfect model.
\end{corollary}

\begin{proof}{
The slow variable, $x$, from the full model in \eqref{fullmodel} has correlation time $T_x = \frac{-1}{\tilde a(1-\epsilon \hat a)} + \mathcal{O}(\epsilon^2)$ which was found in Lemma 2 of Appendix B.  Furthermore, the equilibrium variance of the slow variable, $x$, from \eqref{fullmodel} is,
\[ \mathbb{E}[x(t)^2]  = \frac{\sigma_x^2(1-2\epsilon\hat a) + \epsilon\frac{a_{12}^2}{a_{22}^2}\sigma_y^2}{ -2\tilde a(1-\epsilon\hat a)} + \mathcal{O}(\epsilon^2), \]
as shown in Lemma 1 of Appendix B.
%

The Mean Stochastic Model (MSM) for \eqref{reducedmodel} specifies its parameters with the analytical formula for the variance statistics, $\mathbb{E}[x^2] = \mathbb{E}[X^2]=-\sigma_X^2/(2a)$, and correlation times, $T_x = T_{X}=a^{-1}$ \cite{mgy:10,mh:12}, and from these equations, we obtain $a = \tilde a(1-\epsilon\hat a), \sigma_X^2 = \sigma_x^2(1-2\epsilon\hat a) + \epsilon\frac{a_{12}^2}{a_{22}^2}\sigma_y^2$, which are the parameters in Theorem~\ref{thm2} that give the optimal filter solutions up to order-$\epsilon^2$. }
\end{proof}

This result suggests that in the linear and Gaussian setting, it is possible to find the parameters for optimal filtering without using the filter, by using the Mean Stochastic Model \cite{mgy:10,mh:12}. Furthermore, these parameters also give the optimal equilibrium statistics, up to order-$\epsilon^2$. In the nonlinear setting, however, filtering with the Mean Stochastic Model can produce an accurate mean estimate but typically underestimates the covariance statistics \cite{ls:12}. In the next section, we will explain how this issue arises.

\section{Extension of the Linear Theory to a Nonlinear Test Model}\label{nonlinear theory}

In this section, we consider a simple nonlinear continuous-time filtering  problem,
\begin{align}\label{SPEKF}
d u &= [-(\gamma + \lambda_u)u + b] \, dt+ \sigma_u dW_u, \nonumber \\
d b &= -\frac{\lambda_b}{\epsilon} b \,dt + \frac{\sigma_b}{\sqrt{\epsilon}}dW_b, \\
d\gamma &= -\frac{\lambda_{\gamma}}{\epsilon} \gamma \,dt + \frac{ \sigma_{\gamma}}{\sqrt{\epsilon}} dW_{\gamma}, \nonumber\\
dz &= h(u)\,dt +\sqrt{R}\,dV = u\,dt +\sqrt{R}\,dV.\label{spekfobs}
\end{align}
The discrete observation-time analog of this nonlinear filtering problem was introduced as SPEKF, which stands for ``Stochastic Parameterized Extended Kalman Filter" in \cite{ghm:10b,ghm:10a}, in which filter estimates for SPEKF are obtained by applying a Kalman update to the exactly solvable prior statistical solutions of the full model in \eqref{SPEKF}.  
The nonlinear system in \eqref{SPEKF} has several attractive features as a test model. First, it has exactly solvable statistical solutions which are non-Gaussian.  This fact has allowed evaluation of non-Gaussian prior statistics conditional to the Gaussian posterior statistical solutions of a Kalman filter, which verified certain uncertainty quantification methods \cite{mb:12,bm:13}. Second, the results in \cite{bgm:12} suggest that the system in \eqref{SPEKF} can reproduce signals in various turbulent regimes such as intermittent instabilities in a turbulent energy transfer range, a dissipative range, and for laminar dynamics. Third, the system in \eqref{SPEKF} was also used as a test bed for investigating the consistency of the statistical solutions for various imperfect models in the context of long-term predictability \cite{bm:12}. 
Our goal here is to verify the existence of an ``accurate" reduced filter for this simple test model and to determine whether the corresponding reduced filter model produces accurate long term statistical prediction. Then we will close this section with a simple example which shows what can go wrong when an insufficient reduced stochastic model is used.

In contrast to the linear filtering problem in Section~2, the optimal solution to a nonlinear filtering problem is not available in practice, since it requires solving an infinite-dimensional stochastic system. In particular, the true posterior distribution, $p(\vec u,t) = P(\vec u,t \, | \, z(\tau), 0\leq \tau \leq t)$ for $\vec u = (u,b,\gamma)$, solves the Kushner equation \cite{kushner:64},
\begin{eqnarray} 
dp = \mathcal{L}^*p\, dt + p(h-\mathbb{E}[h])^\top R^{-1}dw_{\hat u}, \label{kushner}
\end{eqnarray}
where $\mathcal{L}^*$ is the Fokker-Planck operator for the state variables $\vec u$.  The term $dw_{\hat u} = dz-\mathbb{E}[h]dt$ is called the innovation process, and it represents the difference between the actual observation $z$ and the expected observation $\mathbb{E}[h]$ with respect to $p$.  As in the linear example above we will assume that $h(\vec u) = u$ so that only the slow variable is observed, thus allowing fair comparison with a reduced model for the slow variable. The Kushner equation is a stochastic partial differential equation (SPDE) which is easily solved when both the dynamics and observation process are linear, in which case one recovers the Kalman-Bucy equations. Since most practical methods that are being used for assimilating high-dimensional nonlinear problems are linear (or Kalman) based methods, we restrict our study to the Gaussian approximation of first two moments, $\hat u = \int u p\, d\vec u$ and $\hat S = \int (u-\hat u)^2p\, d\vec u$, for the slow variable, $u$, of the conditional density which solves \eqref{kushner}. 

In particular, substituting the Kushner expression for $dp$ in $d\hat u = \int u dp d\vec u$ and applying integration by parts with the assumption that $p$ has fast decay at infinity, we find that, 
\[ d{\hat{u}} = \left(-\lambda_{u}\hat{u} - \overline{u\gamma} + \overline{b}\, \right) dt+ \hat{S}R^{-1}dw_{\hat{u}}, \]
where $\overline{u\gamma} = \int u \gamma p\, d\vec u$ and $\overline{b} = \int b p\, d\vec u$.  By differentiating these terms and applying the expansion $p = p_0 + \epsilon p_1$ we can explicitly approximate these terms up to order-$\epsilon^2$.  The full details of this expansion are found in Appendix C (see the electronic supplementary material), where we find that evolution of $\hat u$ and $\hat S$ reduces to
\begin{align}\label{fullSPEKFsolution}
d{\hat{u}} &= -\left(\lambda_{u} -  \frac{\epsilon\sigma_{\gamma}^2}{2\lambda_{\gamma}(\lambda_u\epsilon +\lambda_{\gamma})} \right)\hat{u}\, dt + \hat{S}R^{-1}dw_{\hat{u}} + \mathcal{O}(\epsilon^2), \nonumber \\
d\hat{S} &= \left[-2\left(\lambda_u - \frac{\epsilon\sigma_{\gamma}^2}{\lambda_{\gamma}(\lambda_u\epsilon +\lambda_{\gamma})} \right) \hat{S} + \frac{\epsilon\sigma_{\gamma}^2}{\lambda_{\gamma}(\lambda_u\epsilon +\lambda_{\gamma})}\hat{u}^2 + \sigma_u^2 + \frac{\epsilon\sigma_b^2}{\lambda_b(\lambda_b + \lambda_u\epsilon)}  - \hat{S}R^{-1}\hat{S} \right]dt, \nonumber \\
&\hspace{10pt}+ \left[\int (u-\hat u)^3 p  d\vec{u} \right] R^{-1} dw_{\hat{u}}+  \mathcal{O}(\epsilon^2).
\end{align}
These equations give the exact solutions for the evolution of the first two statistics of the \emph{posterior distribution}, $p$, up to order-$\epsilon^2$, however they are not closed since the skewness $\int (u-\hat u)^3 p\,  d\vec{u}$ appears in the evolution of the covariance $\hat S$.  We close these equations by assuming that the posterior distribution is Gaussian, or effectively, $\int (u-\hat u)^3 p  d\vec{u} = 0$. While the equilibrium statistics of the dynamics in \eqref{SPEKF} have zero skewness, this is not necessarily the case for the posterior distribution given a noisy observation sequence \cite{bm:13}. Note that this closure is  different from the Gaussian Closure Filter (GCF) introduced in \cite{bgm:12}, which applies a Gaussian closure on the prior dynamics before using a Kalman update to obtain posterior solutions.  

Since we are interested in finding a one-dimensional reduced model for the slow variable $u$ we only derive the moment estimates for $u$ which are given by,
\begin{align}\label{twomomentSPEKF}
d{\hat{u}} &= -\left(\lambda_{u} -  \frac{\epsilon\sigma_{\gamma}^2}{2\lambda_{\gamma}(\lambda_u\epsilon +\lambda_{\gamma})} \right)\hat{u}\, dt + \hat{S}R^{-1}dw_{\hat{u}} + \mathcal{O}(\epsilon^2), \nonumber \\
\frac{d\hat{S}}{dt} &= -2\left(\lambda_u - \frac{\epsilon\sigma_{\gamma}^2}{\lambda_{\gamma}(\lambda_u\epsilon +\lambda_{\gamma})} \right) \hat{S} + \frac{\epsilon\sigma_{\gamma}^2}{\lambda_{\gamma}(\lambda_u\epsilon +\lambda_{\gamma})}\hat{u}^2 + \sigma_u^2 + \frac{\epsilon\sigma_b^2}{\lambda_b(\lambda_b + \lambda_u\epsilon)}  - \hat{S}R^{-1}\hat{S} +  \mathcal{O}(\epsilon^2).
\end{align}
We refer to these statistical estimates as the {\bf continuous-time SPEKF solutions} for the variable $u$. To obtain the full continuous-time SPEKF solution, one can compute the mean and covariance matrix of the full state $\vec u$, with similar computations via the It\^o calculus. In this sense, the original SPEKF that was introduced in \cite{ghm:10a,ghm:10b} is a discrete-time analog of the continuous-time SPEKF since it implicitly truncates the higher-order moments of the posterior statistics through a discrete-time Kalman update.


Motivated by the results in \cite{mb:12,gh:13,bm:13}, we now propose the following reduced filter model to approximate the filtering problem in \eqref{SPEKF},
\begin{align}
\label{reducedSPEKF} dU &= -\alpha U dt + \beta U \circ dW_\gamma + \sigma_1 dW_u + \sigma_2 dW_b,\nonumber \\ &= -\left(\alpha-\frac{\beta^2}{2}\right)U dt + \beta U dW_\gamma + \sigma_1 dW_u + \sigma_2 dW_b, \\
dz &= h(U)\,dt + \sqrt{R}\,dV = U\,dt + \sqrt{R}\,dV.\nonumber
\end{align}
The evolution of the first two moments of \eqref{reducedSPEKF}, $\tilde u = \int U\pi dU$ and $\tilde S = \int (U - \tilde u)^2\pi\,dU$, where $\pi$ is the posterior distribution governed by the Kushner equation for \eqref{reducedSPEKF}, 
are given by,
\begin{align}\label{twomomentReducedSPEKF}
d{\tilde{u}} &= -\left(\alpha-\frac{\beta^2}{2}\right)\tilde{u}\,dt + \tilde SR^{-1}dw_{\tilde{u}}, \nonumber \\
\frac{d}{dt}\tilde S &= -2\left(\alpha-\beta^2\right) \tilde S + \beta^2  \tilde{u}^2 + \sigma_1^2+\sigma_2^2 - \tilde SR^{-1}\tilde S,
\end{align}
where $dw_{\tilde{u}} = dz - \tilde{u}\,dt$ denotes the innovation process and Gaussian closure is imposed by setting the skewness to zero (see Appendix C for detailed derivation).  We can specify the parameters in \eqref{reducedSPEKF} by matching coefficients in the equations governing the evolution of the mean and covariance in the filters \eqref{twomomentSPEKF} and \eqref{twomomentReducedSPEKF} which yields 
\begin{align}\label{params}
\alpha &= \lambda_u, \quad \sigma_1^2 = \sigma_u^2,\nonumber \\
\sigma_2^2 &= \frac{\epsilon\sigma_b^2}{\lambda_b(\lambda_b + \epsilon\lambda_u)}, \quad\beta^2 =  \frac{\epsilon\sigma_{\gamma}^2}{\lambda_{\gamma}(\lambda_u\epsilon +\lambda_{\gamma})}.
\end{align}
We refer to the solutions of \eqref{twomomentReducedSPEKF} with parameters in \eqref{params} as the {\bf continuous-time reduced SPEKF solutions} of the filtering problem \eqref{reducedSPEKF}.  With this choice of coefficients we have the following result (see Appendix C for detailed proof).
\begin{theorem}\label{thm3}
Let $\lambda_u > 0$, and let $z$ be noisy observations of the state variable $u$ which solves the full model in \eqref{SPEKF}.  Given identical initial statistics, $\tilde{u}(0)=\hat{u}(0)$ and $\tilde{S}(0) = \hat S(0)>0$, the mean and covariance estimates of a stable continuous-time reduced SPEKF in \eqref{reducedSPEKF} with parameters \eqref{params} agree with mean and covariance of a stable continuous-time SPEKF for variable $u$ in the following sense. There exist time-independent constants, $C_1, C_2$, such that,
\begin{align}
|\hat{S}(t)-\tilde{S}(t)| &\leq C_1\epsilon,\nonumber\\
\mathbb{E}\left[|\hat{u}(t)-\tilde{u}(t)|^2 \right]&\leq C_2 \epsilon^2.\nonumber
\end{align}
Furthermore, the reduced filtered solutions are also consistent, up to order-$\epsilon$.
\end{theorem}
Theorem \ref{thm3} shows that the continuous-time reduced SPEKF solutions in \eqref{twomomentReducedSPEKF} are consistent up to order-$\epsilon$, and match the first two moments of the continuous-time SPEKF solutions for the slow variable $u$ up to order-$\epsilon$. Moreover, Theorem \ref{thm3} implies that in the context of Gaussian closure on the posterior distribution, accounting for truncation of fast time scales in a nonlinear model with only additive noise requires a multiplicative noise correction term in the reduced model.  

We note that the term $\epsilon \lambda_u$ appearing in the denominator of the parameters $\sigma_2$ and $\beta$ in \eqref{params} is technically an order-$\epsilon^2$ adjustment, however, this term arises naturally in the derivation of the continuous-time SPEKF solutions for \eqref{SPEKF} in Appendix C and is important as we will discuss below. 
We should point out that these extra order-$\epsilon^2$ correction terms were not found in the white noise limit approximation \cite{mb:12,gh:13,bm:13}.


\subsection{Numerical Experiments: assessing the mean and covariance filter estimates}

In the numerical experiments below, we show results for two regimes (as defined in \cite{bgm:12}) for \eqref{SPEKF}. Regime I corresponds to the turbulent energy transfer range, in which $\gamma$ decays faster than $u$. The parameters for this regime are: $\lambda_u= 1.2-1.78i, \lambda_b = 0.5-i, \lambda_{\gamma} = 20, \sigma_u = 0.5, \sigma_b = 0.5, \sigma_{\gamma} = 20, \epsilon=1$. Regime II, as defined in \cite{bgm:12}, is an extremely difficult regime corresponding to the dissipative range, where the dynamics of $u(t)$ exhibits intermittent burst of transient instabilities, followed by quiescent phases. The parameters are: $\lambda_u= 0.55-1.78i, \lambda_b = 0.4-i, \lambda_{\gamma} = 0.5, \sigma_u = 0.1, \sigma_b = 0.4, \sigma_{\gamma} = 0.5, \epsilon=1$. In this regime, the decaying time scales for $u$ and $\gamma$ are comparable. 
Note that the accuracy of the closure in \eqref{reducedSPEKF} is up to order-$\epsilon$ when the parameters in the full model in \eqref{SPEKF} are all order-one. Since the parameters in Regime I are defined without $\epsilon$ in \cite{bgm:12} and not all of them are order-one, by taking the ratio of the damping coefficients $Re(\lambda_u)=1.2$ and $\lambda_\gamma=20$, the implicit time scale separation is approximately $\epsilon\approx \lambda_u/\lambda_\gamma=0.05$. In regime II, the implicit time scale separation is approximately $\epsilon\approx \lambda_u/\lambda_\gamma=1.1$. 
In these numerical experiments, we apply all the filters with discrete time observations at time $\Delta t=0.5$ and noise covariance $R = 50\%Var(u)$. Here, we numerically compare the full SPEKF solutions with: 
\begin{itemize}
\item The reduced stochastic filter (RSF) which assumes $\alpha = \lambda_u, \sigma_1^2= \sigma_u^2, \sigma_2 =0$ and $\beta = 0$. \item The reduced stochastic filter with additive correction (RSFA) which assumes $\alpha = \lambda_u, \sigma_1^2= \sigma_u^2$, $\sigma_2^2 = \epsilon\sigma_b^2/\lambda_b^2$, $\beta=0$.
\item The reduced SPEKF solutions with white-noise limit parameters \cite{mb:12,gh:13,bm:13}, $\alpha = \lambda_u, \sigma_1^2 = \sigma_u^2, \sigma_2^2 = \epsilon\sigma_b^2/\lambda_b^2, \beta^2 =  \epsilon\sigma_{\gamma}^2/\lambda_{\gamma}^2$. We'll denote this by RSFC, following the notation in \cite{gh:13}. \item The RSPEKF solutions with parameters in \eqref{params}. 
\end{itemize}

In Table~\ref{table1}, we show the average RMS errors, averaged over 20,000 assimilation cycles. In regime I, the accuracy of the filtered mean estimates of SPEKF, RSFC, and RSPEKF are roughly similar.  On the other hand, RSF and RSFA are less accurate, in particular, the average RMS error of RSF is larger than the observation noise error, $\sqrt{R}$. In regime II, RSPEKF has the smallest error, even smaller than SPEKF, followed by RSFC. The linear filters without multiplicative noise, RSF and RSFA, are not accurate at all, their errors are roughly twice the observation noise error.  We do not show the pathwise filtered solutions compared to the true signals since they look very similar to those of Figures~7 and 8 of \cite{gh:13}. Instead, we examine the filter covariance estimates (see Figure~\ref{nonlinearfilter}). Notice that in both regimes, the covariance estimates of both RFSC and RSPEKF are larger than that of SPEKF. The differences between RSFC and SPEKF in Regime II are even more pronounced. The differences between RSPEKF and SPEKF are of order-$\epsilon$ in both regimes, where $\epsilon=0.05$ for regime I and $\epsilon=1.1$ in regime II. The covariance estimates of the other two linear filters, RSF and RSFA, converge to constant solutions, as expected; RSF underestimates the covariance, while RSFA covariance estimates are closer to RSPEKF. 

From these covariance estimates, we cannot conclude which of them over- or under-estimate 
the actual error covariance since we have no access to the optimal filter solutions; even SPEKF solutions are sub-optimal since they are the Gaussian approximation of the first two-moments of \eqref{kushner}. Motivated by the result in Theorem~\ref{thm2}, where the optimal filter guarantees a consistency condition in the sense that the filter covariance estimate matches the actual filter error, we propose the following metric as an empirical measure to determine whether the filter covariance estimates are consistent. 

\begin{defin}{Consistency (of Covariance).}
Let $\tilde x(t)\in\mathbb{R}^n$ and $\tilde S(t)\in\mathbb{R}^{n\times n}$ be a realization of the solution to a filtering problem for which the true signal of the realization is $x(t)\in\mathbb{R}^n$.  The \emph{consistency} of the realization is defined as
\begin{align}\label{consistency} \mathcal{C}(x,\tilde x,\tilde S) =  \Big\langle \frac{1}{n} (x(t)-\tilde{x}(t))^\top\tilde{S}(t)^{-1}(x(t)-\tilde{x}(t))\Big\rangle,\end{align}
where $\langle\cdot\rangle$ denotes temporal average. We say that a filter is \emph{consistent} if $\mathcal{C} = 1$ almost surely (independent of the realization). The filter covariance under(over)estimates the actual error covariance when $\mathcal{C}>1$ ($\mathcal{C}<1$). 
\end{defin}

This metric is simply the signal part of the relative entropy measure of two Gaussian distributions \cite{bm:14}. With this definition, it is obvious that an optimal filter is always consistent. However it is not the only consistent filter and not every consistent filter is accurate (see Appendix D in the electronic supplementary material for trivial examples). 
In parallel to the consistency condition in \eqref{thm2}, this consistency measure is only a necessary (or weak) condition for the covariance to be meaningful. It should be used together with the mean squared error measure. However, this measure has the following useful property: a consistent filter which produces posterior mean estimates close to the true posterior mean estimates also has a covariance close to the true posterior covariance (see Appendix D in the electronic supplementary material for detail). We should point out that although this measure is much weaker than the pattern correlation measure advocated in \cite{bm:14}, we shall see that many suboptimal filters are not even consistent in the sense of Definition~3.1.

\comments{
\begin{figure}
\centering
\includegraphics[width=.45\textwidth]{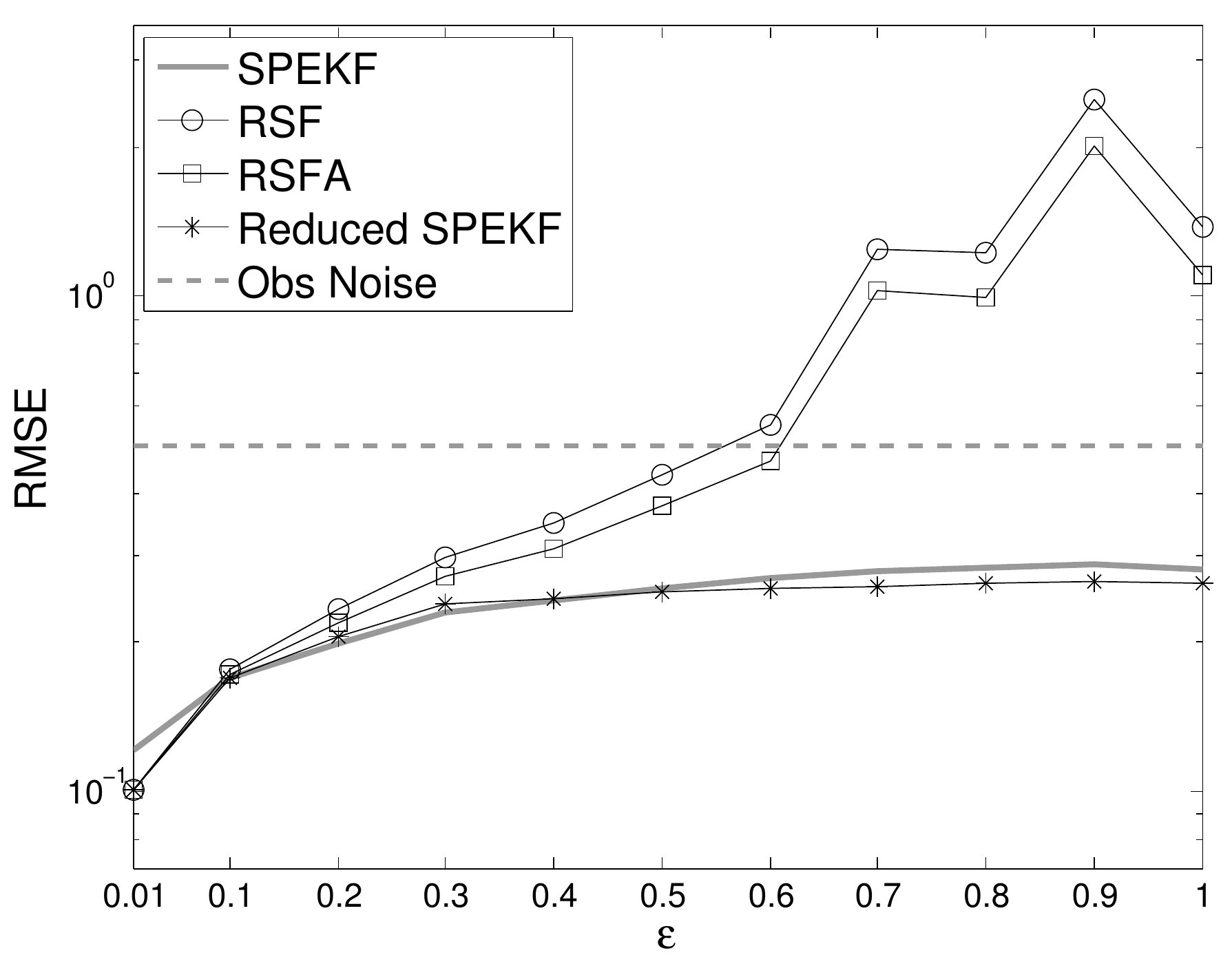}
\includegraphics[width=0.45\textwidth]{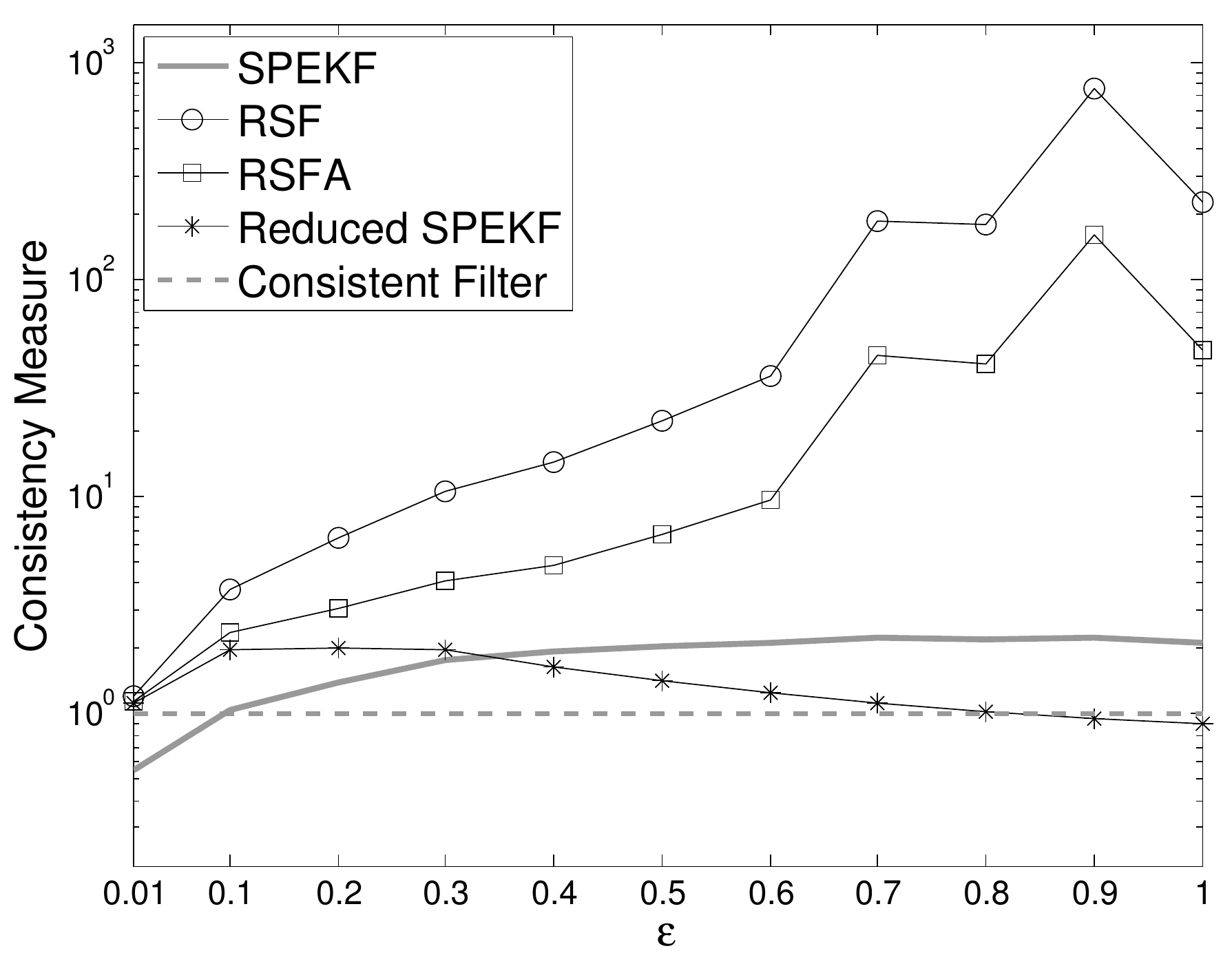}
\caption{\label{nonlinearEx} Filter performance measured in terms of root mean squared error (RMSE, left) and consistency measure (right) for various time-scale separations $\epsilon$ with an integration time step of $dt = 0.05$ and observations at time intervals $\Delta t = 0.5$ with observation noise $R=0.25$.  Results are overaged over $20,000$ observations in Regime II from \cite{gh:13} where $\lambda_u= 0.55, \lambda_b = 0.4, \lambda_{\gamma} = 0.5, \sigma_u = 0.1, \sigma_b = 0.4, \sigma_{\gamma} = 0.5$.  The \emph{SPEKF} solution use the full model from \eqref{SPEKF}, the \emph{Reduced SPEKF} uses the model from \eqref{reducedSPEKF} with the parameters \eqref{params} whereas \emph{RSFA} sets $\beta = 0$ and \emph{RSF} sets $\sigma_2=\beta=0$.}
\end{figure}}

\begin{table}
\caption{Average RMSE (and empirical consistency measure) for various filtered mean estimates in Regimes I and II over 20,000 assimilation cycles for observation time $\Delta t=0.5$ and $R=0.5Var(u)$.}
\begin{center}
\begin{tabular}{|c|c|c|c|c|}
\hline
Scheme & Regime I & Regime II \\ \hline
SPEKF & 0.47 (1.27) & 2.29 (11.50)\\
RSF & 0.84 (9.42) & 10.53 ($1.22\times 10^4$)\\
RSFA & 0.54 (1.52)& 9.54 (106.76)\\
RSFC & 0.47 (0.90)& 3.00 (0.60)\\
RSPEKF & 0.47 (1.10)& 2.02 (3.37)\\ 
$\sqrt{R}$ & 0.5866 & 5.2592 \\ \hline
\end{tabular}
\end{center}
\label{table1}
\end{table}%

\begin{figure}
\includegraphics[width=.45\textwidth]{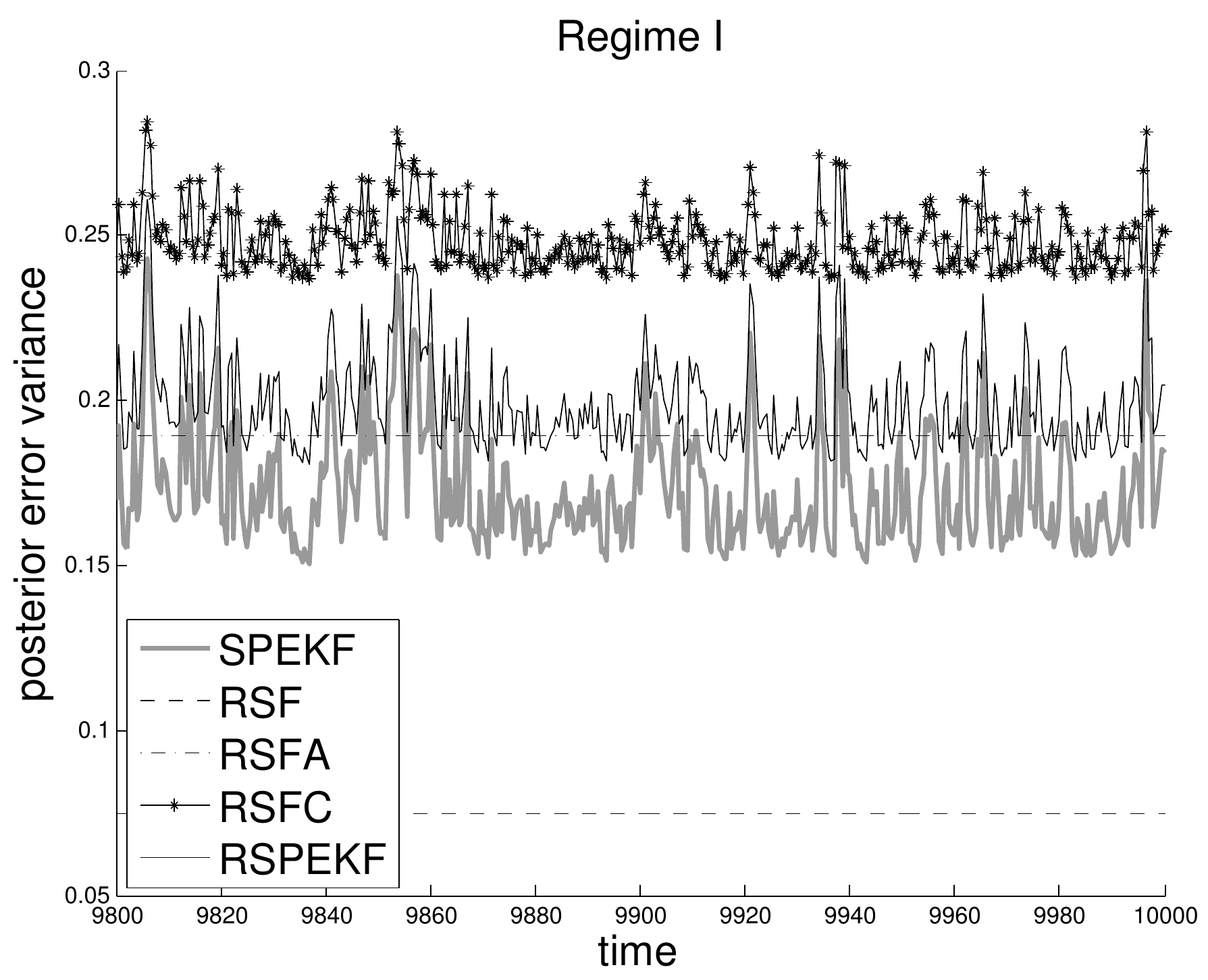}
\includegraphics[width=.45\textwidth]{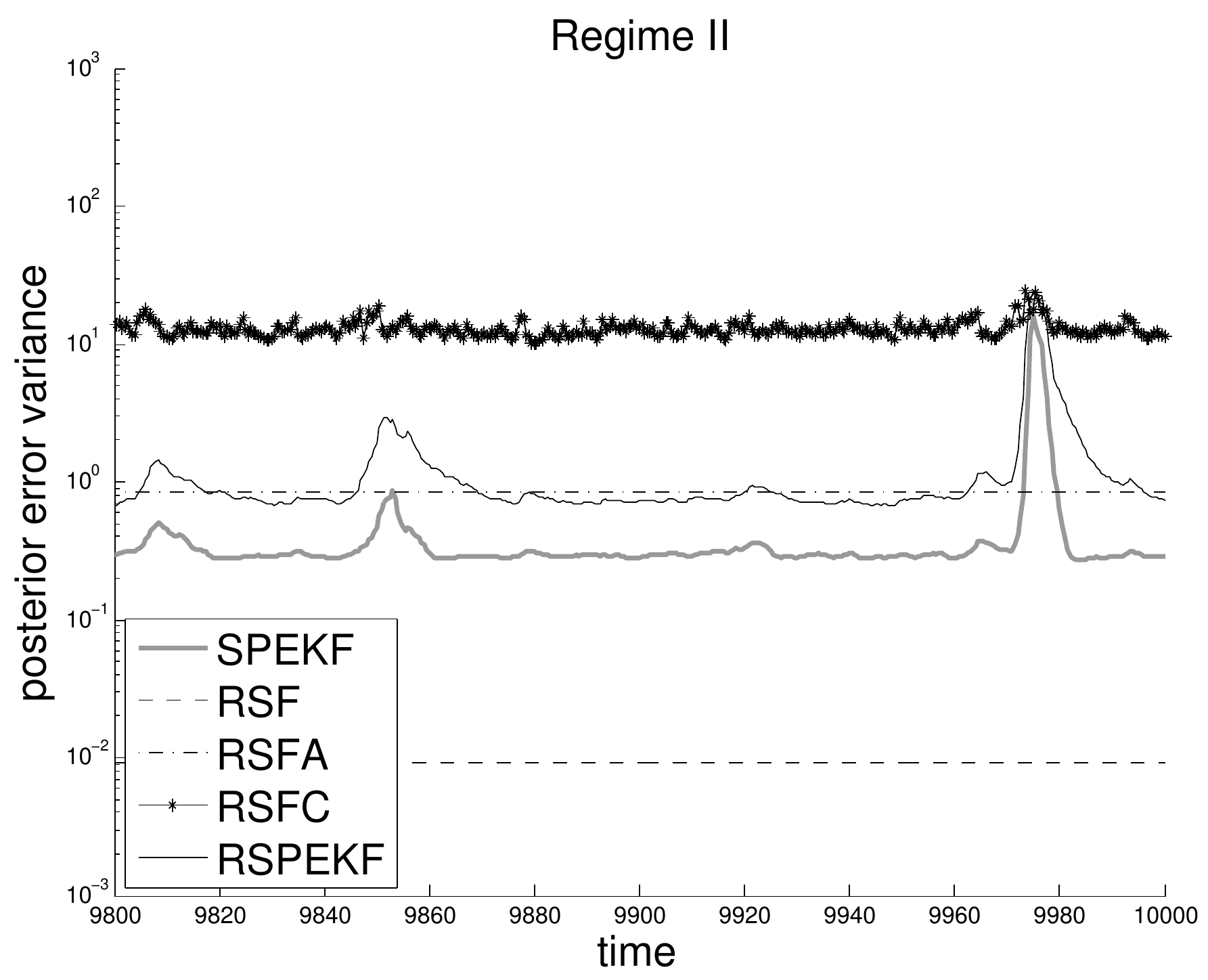}
\caption{\label{nonlinearfilter}Posterior error covariance estimates, corresponding to the mean solutions in Table~1, as functions of time. In this numerical simulations, we consider observation time interval $\Delta t=0.5$ and $R=0.5Var(u)$.}
\end{figure}
 
In Table~\ref{table1}, we record the numerically computed empirical consistency for the corresponding filtering experiments. The consistency results show that in regime I, almost all filtering methods except RSFC are underestimating the actual error covariance. In this regime, both RSFC and RSPEKF produce the most consistent covariance estimates. In Regime II, both linear filters (RSF and RSFA) significantly underestimate the actual error covariance. RSFA improves both mean and covariance estimate (it reduces the consistency measure from $\mathcal{C}\approx 10^{4}$ to $10^2$) by an additive covariance inflation factor $\sigma_2^2=\epsilon\sigma_b^2/\lambda_b^2$. In this regime, SPEKF, which produces reasonably accurate filtered solutions, also underestimates the actual error covariance ($\mathcal{C}\approx 11.5$). Even RSPEKF underestimates the covariance ($\mathcal{C}\approx 3.37$). We suspect that the underestimation of these covariances in SPEKF and RSPEKF are due to the Gaussian closure approximation. Additionally, the underestimation of the actual error covariance in the full SPEKF solutions can be attributed to a combination of the following issues: (1) the full SPEKF has sparse observations of only $u$; (2) the prior statistical solutions of the full SPEKF involve quadrature approximation of various integral terms.

\subsection{Numerical Experiments: Assessing the predictive skill of the covariance estimates}
In this section we compare the evolution of the covariance of the stochastic variable $u$ in the nonlinear model in \eqref{SPEKF} with
the evolution of the covariance of $U$ in the approximate model in \eqref{reducedSPEKF} for the following parameter sets:
\begin{itemize}
\item The RSFC or white-noise limit parameters \cite{mb:12,gh:13,bm:13}.
\item The RSPEKF parameters obtained in \eqref{params},
\end{itemize}
In this numerical experiment, we solve the evolution of the true covariance of $u$ and the two reduced models analytically as in \cite{mh:12} and in the Appendix of \cite{gh:13}, respectively.  We assume an independent Gaussian initial condition, $p(u,b,\gamma) = p_G(u)p_G(b)p_G(\gamma)$, where,
\[ p_G(u) = \mathcal{N}(0,\sigma_u^2/(2Re(\lambda_u))) \hspace{10pt}\textup{and}\hspace{10pt} p_G(b) = \mathcal{N}(0,\sigma_b^2/(2Re(\lambda_b))) \hspace{10pt}\textup{and}\hspace{10pt} p_G(\gamma) = \mathcal{N}(0,\sigma_\gamma^2/(2\lambda_\gamma)). \] 
Each model is then used to evolve these initial conditions forward in time and the resulting covariances are shown in Figure~\ref{nonlinearprior}. Notice that in both regimes, the covariance estimates from the parameters of RSFC are not accurate at all. In regime I, the absolute error of the final covariance estimates shown in Figure~\ref{nonlinearprior} is about $|Var(u)-Var(U)| = .7171\approx 14\epsilon$ for $\epsilon=0.05$. In regime~II, the covariance estimate of the RSFC is unstable since the stability condition, 
\begin{align}
\Xi_2 = -2\lambda_u + \epsilon \frac{2\sigma_\gamma^2}{\lambda_\gamma^2}<0,
\end{align}
is not satisfied. The order-$\epsilon^2$ correction terms in \eqref{params} yield significant improvement in both regimes. In regime I, the absolute error of the covariance estimate from RSPEKF, $|Var(u)-Var(U)| = .3 = 6\epsilon$, is much smaller than that of RSFC. In regime II, the RSPEKF correction terms ensure the stability of the prior covariance estimate, that is, it provides the following stability condition, 
\begin{align}
\tilde{\Xi}_2 = -2\lambda_u + \epsilon \frac{2\sigma_\gamma^2}{\lambda_\gamma(\lambda_\gamma + \epsilon \lambda_u)}<0.
\end{align}
Moreover, the corresponding absolute error in this regime is $|Var(u)-Var(U)| = .5979 = .5\epsilon$, where $\epsilon=1.1$.

\begin{figure}
\includegraphics[width=.45\textwidth]{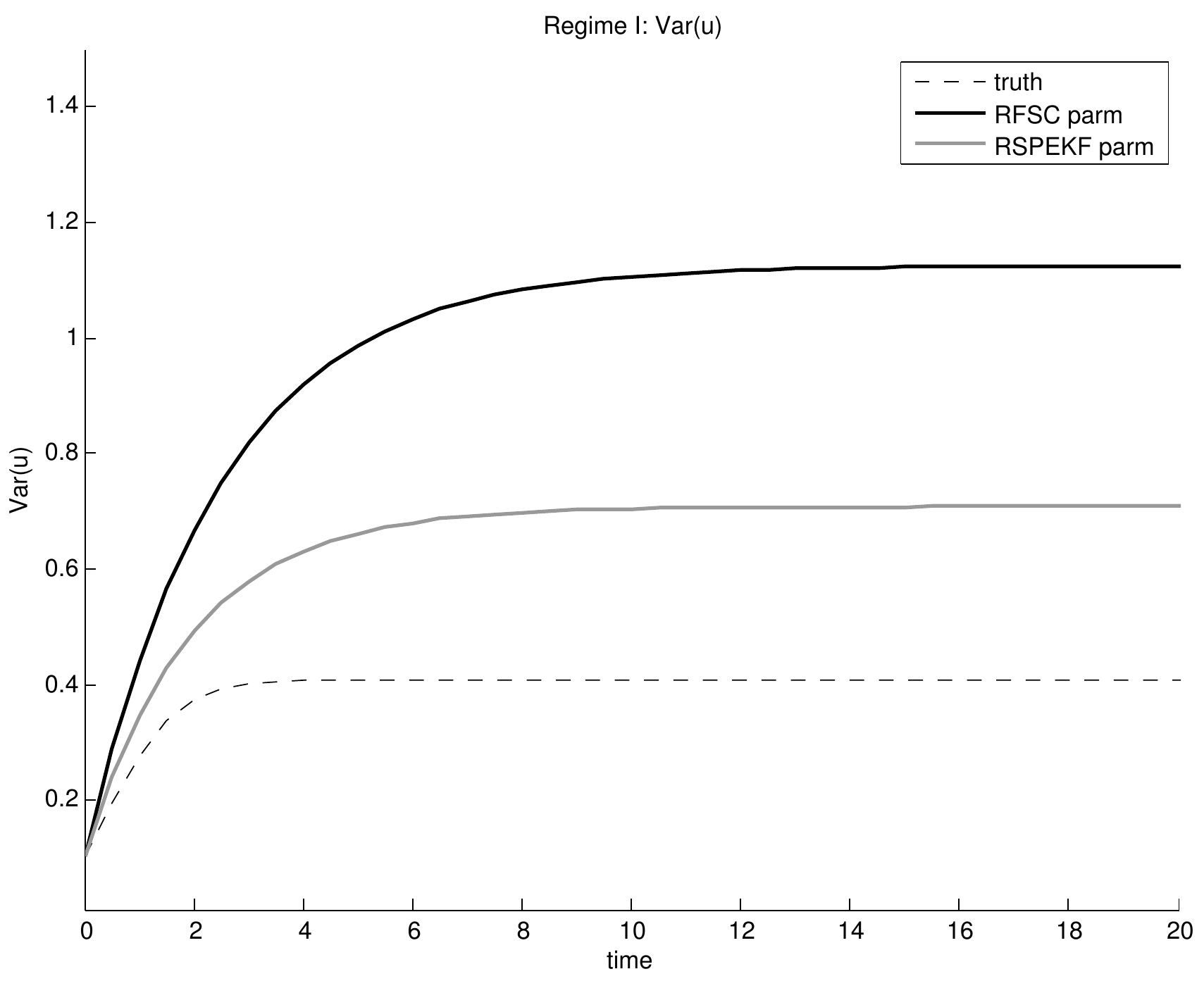}
\includegraphics[width=.45\textwidth]{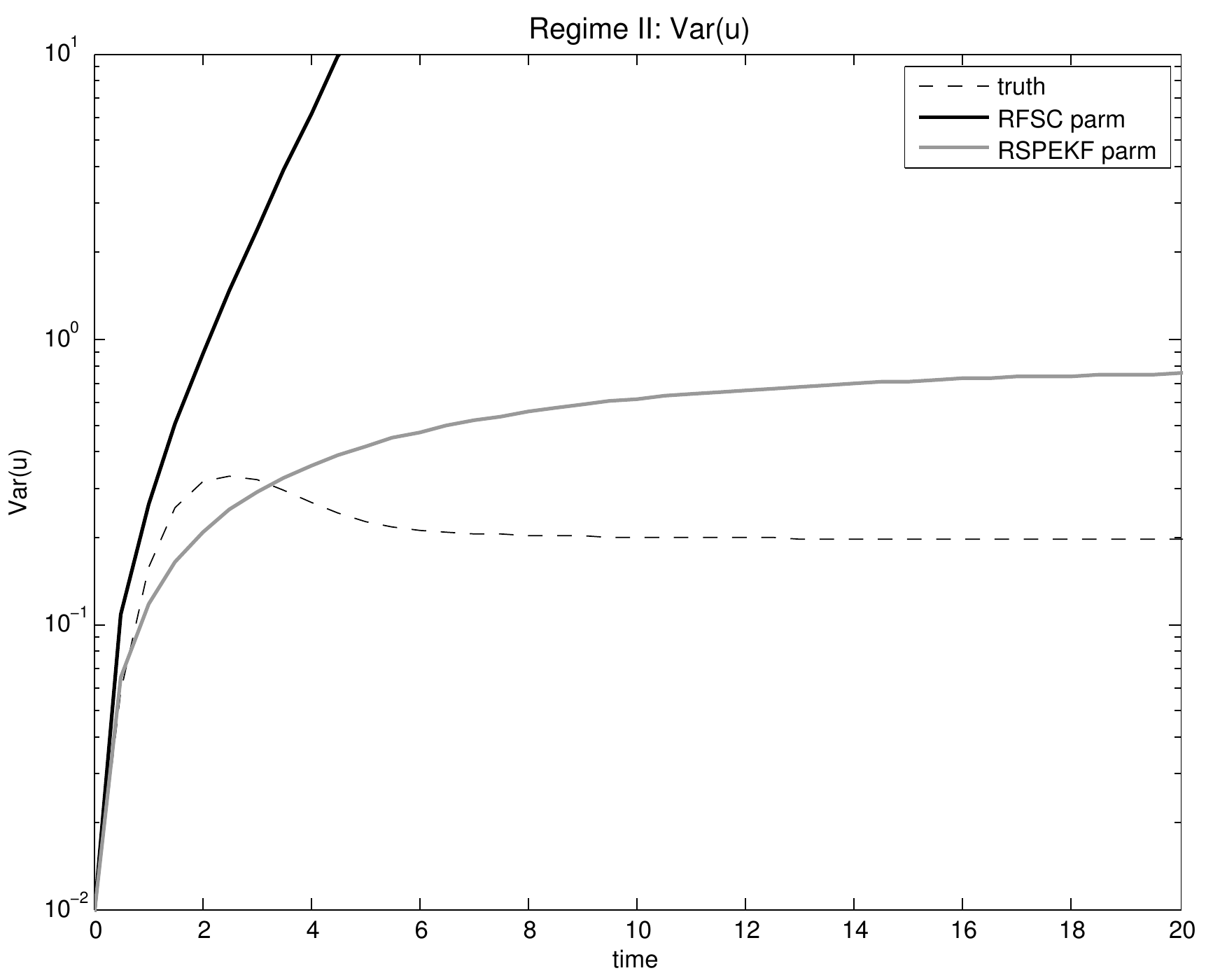}
\caption{\label{nonlinearprior}Covariance solutions of $u$ for the true model in \eqref{SPEKF}, the reduced model in \eqref{reducedSPEKF} with parameters specified as in RSFC and RSPEKF.}
\end{figure}

The RSPEKF example shows that there exists a parameter set, given by \eqref{params}, that produces reasonably accurate and consistent filtered estimates as well as relatively accurate covariance solutions, up to order-$\epsilon$.  
Interestingly, the stochastic noise terms in the moment equations of \eqref{twomomentSPEKF} and \eqref{twomomentReducedSPEKF} do not effect the determination of the parameters in \eqref{params}. In fact, we can obtain the same stochastic parameters by applying the asymptotic expansion to the first two-moments of the marginal distribution of $u$ which solves the deterministic part of SPDE in \eqref{kushner} (which is simply the Fokker-Planck equation).  We also note that by a straightforward (but tedious and lengthy) calculation, one can verify that, for this particular example, the parameters \eqref{params} also match the third moments of SPEKF and RSPEKF, when one does not apply the Gaussian closure in \eqref{fullSPEKFsolution}. The coefficients in the higher-order moment equations still satisfy the same constraints which yield \eqref{params}. Thus, as long as no model error is committed in the observation operator $h$ in the reduced filter of \eqref{reducedSPEKF}, the stochastic term in \eqref{kushner} will not produce extra constraints.

In real applications, however, it is typically difficult to find an appropriate ansatz which can give both accurate filter solutions and accurate long term statistical prediction with the same parameters. In this case, it is possible to have more constraints than the number parameters in the reduced model. To illustrate this point, suppose we choose the following stochastic ansatz for the reduced filter model,
\begin{align}
dU = -\alpha U\,dt + \sigma_1 dW_u + \sigma_2 dW_b,\label{badansatz}
\end{align} 
ignoring the multiplicative noise component in \eqref{reducedSPEKF}. By comparing the Kalman-Bucy solutions of \eqref{badansatz} to the Gaussian closure moments in \eqref{twomomentSPEKF}, it is clear that the accuracy of the filtered mean and covariance estimates are not within order-$\epsilon^2$. With this stochastic ansatz, we can fit the equilibrium variance and correlation time (with the MSM method as mentioned in Corollary~\ref{thm3}) to obtain $\alpha=0.7683-0.9971i$ and $\sigma_1^2+\sigma_2^2=2.1147$ (shown with the `+' sign in Figure~\ref{fig4} for Regime I). In Figure~\ref{fig4}, we compare the average RMS errors and consistency measure of filtering noisy observations of \eqref{SPEKF}, with the stochastic ansatz in \eqref{badansatz} as the filter model, for a wide range of parameters. We set the frequency parameter, $Im\{\alpha\}=-0.9971i$ to be the exactly the value determined by MSM. In these numerical simulations, the filtered performance is quantified over 20,000 assimilation cycles for observation noise $R=0.5Var(u)$ and time $\Delta t=0.5$. Notice that the MSM parameters do not produce the best filtered solutions; they yield an average RMSE close to 0.85 and consistency measure close to 1.4. Moreover, in this parameter range, the average RMS error is much larger than 0.54 which was produced by RSFA(see Table~\ref{table1}), which is also using the ansatz in \eqref{badansatz} with parameters in \eqref{params}, except for $\beta=0$. Conversely, the parameters associated with RSFA in Table~1 produce inaccurate equilibrium statistics; the correlation time and the equilibrium variance are significantly underestimated by 52\% and 83\%, respectively. This example illustrates the importance of having an appropriate ansatz.  Moreover, when the ansatz is not appropriate, parameters which are chosen to give good equilibrium statistics may give poor filter performance and vice versa.

\begin{figure}
\includegraphics[width=.45\textwidth]{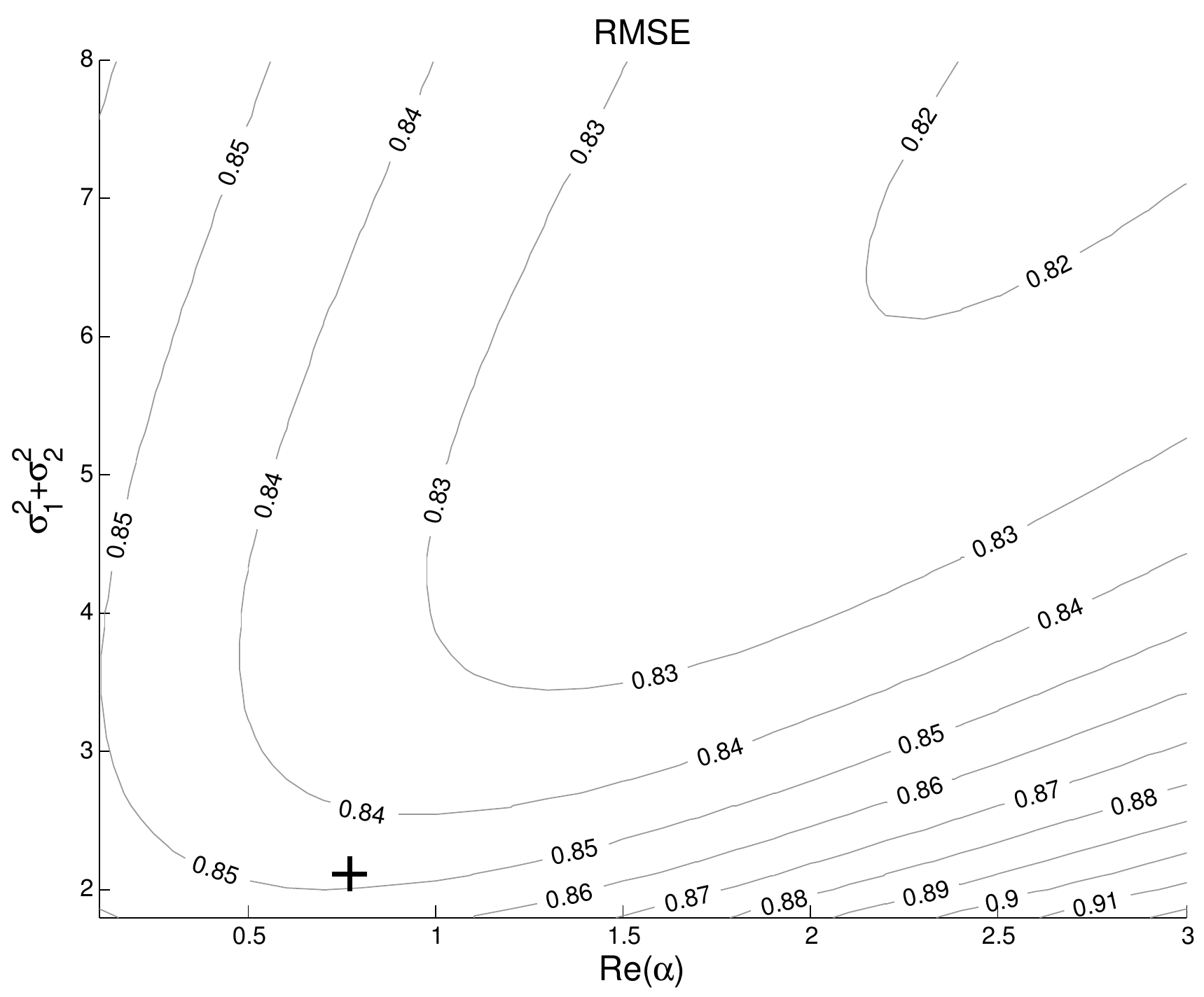}
\includegraphics[width=.45\textwidth]{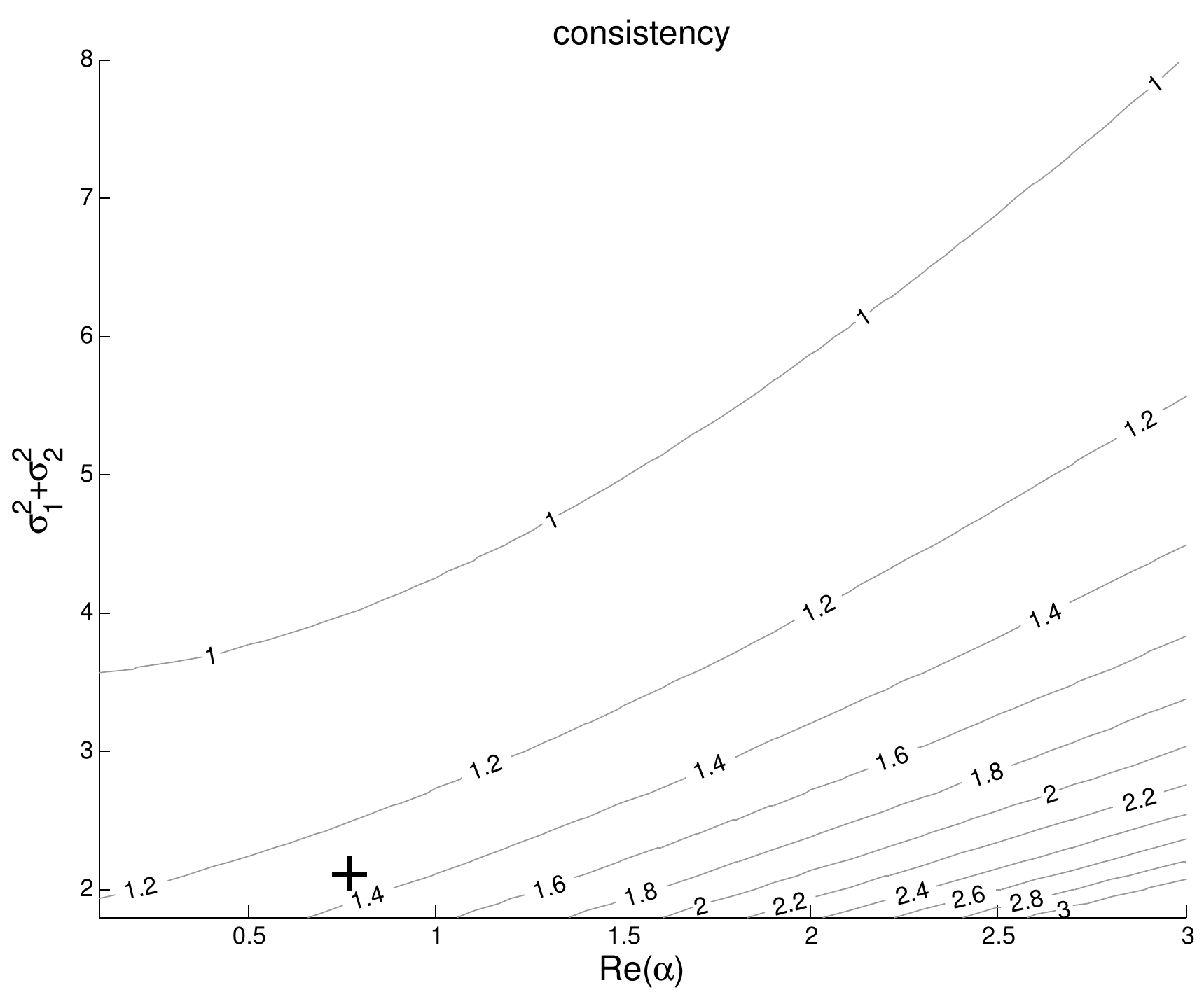}
\caption{\label{fig4} Contour plot of the average RMS error and the weak consistency measure for filtering noisy observations in \eqref{SPEKF} with the reduced filter model in \eqref{badansatz} for various parameters $Re(\alpha)$ and $\sigma_1^2+\sigma_2^2$ and fixed frequency $Im(\alpha)$ determined by the equilibrium statistics through MSM method (see Corollary~\ref{thm3}) in Regime I. The `+' sign denotes the parameters determined by the equilibrium statistics through MSM method.}
\end{figure}

\section{Stochastic parameterizations for the two layer Lorenz-96 model}\label{L96}

In the previous sections we were able to obtain an optimal stochastic parameterization ansatz which compensates for the unresolved scales, because the full dynamics are known and the two test problems are quite simple.  Our results showed that it is critical to use the correct stochastic ansatz in order to simultaneously obtain accurate filtering and accurate equilibrium statistical estimates. In practical applications, it is rather difficult to derive the correct parametric form for the reduced model, especially when the dynamics of the unresolved variables are not completely known. Motivated by our results in Sections~2 and 3, as well as the normal forms for reduced climate models deduced in \cite{mtv:01,mfc:09}, we propose a stochastic parameterization ansatz which includes a linear damping term and a combined, additive and multiplicative, stochastic forcing to account for the unresolved scales in filtering nonlinear multiscale problems. 

In this section, we present numerical results from filtering the two-layer Lorenz-96 model \cite{lorenz:96}, which has been widely used for testing stochastic parameterization methods \cite{,fev:04,wilks:05,cev:08,kh:12}. The two-layer Lorenz-96 model is an $N(J+1)$-dimensional ODE given by,
\begin{align}\label{lor96}\frac{dx_i}{dt} &= x_{i-1}(x_{i+1}-x_{i-2}) - x_i + F + h_x\sum_{j=(i-1)J+1}^{iJ} y_{j}, \nonumber \\
 \frac{dy_{j}}{dt} &= \frac{1}{\epsilon}\big(a y_{j+1}(y_{j-1}-y_{j+2}) - y_{j} + h_y x_{\textup{ceil}(i/J)} \big),
\end{align}
where $\vec{x}=(x_i)$ and $\vec{y}=(y_j)$ are vectors in $\mathbb{R}^{N}$ and $\mathbb{R}^{NJ}$ respectively and the subscript $i$ is taken modulo $N$ and $j$ is taken modulo $NJ$. To generate the observations, we integrate this model using the Runge-Kutta method (RK4) with a time step $\delta t$ and take noisy observations $\vec{z}_k$ at discrete times $t_k$ with spacing $\Delta t=t_{k+1}-t_k$ given by,
\begin{align}
\vec{z}_k = h(\vec{x}(t_k)) + \eta_k = \vec{x}(t_k) + \eta_k, \quad \eta_k\sim\mathcal{N}(0,R),\label{L96obs}
\end{align}
where $\vec{z}_k\in\mathbb{R}^M$ and $R$ is a symmetric positive definite $M\times M$ matrix.
The main goal of this section is to show that when the ``correct" ansatz is known, a natural way to estimate the parameters is \emph{online} (as part of the filtering procedure). To achieve this, we separate the discussion into three subsections. In the first subsection, we will provide a short review of the online parameterization method that was introduced in \cite{bs:13} (we accompany this section with a more detail discussion in the Appendix E in the electronic supplementary material). Second, we compare the filter performance of various choices of ansatz with the online parameter estimation scheme. The main point here is to empirically find the ``correct" ansatz since the complexity of the full model makes it difficult to analytically derive such an  ansatz. In the third part of this section, we compare the filter and the equilibrium statistical estimates of the online parameterization method with an offline stochastic parameter estimation method proposed in \cite{wilks:05,amp:13}. To make a fair comparison, we will use the same stochastic parametric form (ansatz). 

\subsection{Review of the Online Parameter Estimation Method}

We consider the following reduced stochastic model to approximate the filtering problem in \eqref{lor96}, \eqref{L96obs}, 
\begin{align}\label{lor96reduced}
\frac{dx_i}{dt} &= x_{i-1}(x_{i+1}-x_{i-2}) - a x_i + F + \Big(- \alpha x_i + \sum_{j=1}^N \sigma_{ij} \dot{W}_j + \sum_{j=1}^N \beta_{ij}x_j \circ \dot{V}_j \Big). 
\end{align}

As we pointed out earlier, such a stochastic parameterization was motivated by our results in Sections~2 and 3, as well as by earlier work that suggested a normal form for reduced climate modeling \cite{mtv:01,mfc:09}. The filter model in \eqref{lor96reduced} is simply the one-layer Lorenz-96 model augmented with an additional linear damping term $\alpha$, an additive noise term which amplitude is given by the matrix $\sigma = (\sigma_{ij})$ and a multiplicative noise term with coefficient $\beta=(\beta_{ij})$. The notations $\dot{W}, \dot{V}$ denote standard i.i.d. white noise. In the remainder of this paper, we will set $\beta=0$ since the multiplicative noise seems to play little role based on the study in \cite{amp:13} and we suspect that an online estimation method for the multiplicative noise amplitude may involve a more expensive MCMC algorithm, which is beyond the scope of this paper.

The core of the online parameter estimation scheme considered in this paper is the EnKF algorithm with an adaptive noise estimation scheme proposed in \cite{bs:13}. For the application in this paper, we will combine this scheme with the classical online state augmentation method \cite{friedland:69} to obtain the deterministic parameters ($\alpha$ in our case). Generally speaking, the algorithm consists of two steps: The first step is to apply the standard EnKF method to estimate the augmented state-parameter variables, $(\vec{x},\alpha)$, assuming that the parameter dynamics are the persistent model, $d{\alpha}/dt=0$, as in \cite{friedland:69}. The second step is to use the zero-lag and one-lag covariances of the resulting innovation vectors to obtain an estimate for $Q=\sigma\sigma^\top$ and the observation noise covariance $R$. This second step was originally proposed in \cite{mehra:72} for linear problems and extended to EnKF framework in \cite{bs:13}. See Appendix E in the electronic supplementary material for implementation detail of this online noise estimation method.  By only using the zero-lag and one-lag covariances, the method of \cite{bs:13} can estimate at most $M^2$ parameters, where $M$ is the dimension of the observation (this is only a necessary condition and further observability obstructions are possible).  When $M=N$ this means that the entire matrix $Q$ can usually be estimated, and this version of the algorithm is used in Section~4(b).  However, when the observations are sparse we must parameterize the $Q$ matrix.  In Section~4(c) we consider a sparse observation where only half of the slow variables are observed, and because of the spatial homogeneity of the problem we introduce a cyclic parameterization of $Q$.  The idea of the cyclic parameterization is that the covariance in the model error should only depend on the distance between the slow variables, so for example $Q_{12}=Q_{23}=\cdots=Q_{N1}$.  The cyclic parameterization reduces the number of parameters in $Q$ from $N^2$ to $\textup{ceil}(N/2)$ and can be used whenever $\textup{ceil}(N/2) \leq M^2$; the full details are described in detail in Appendix E.
We should point out that this observability issue can also be mitigated with an alternative algorithm in the EnKF framework, which uses more than one-lag covariances of the innovation vectors to estimate $Q$ and $R$ \cite{hmm:14}.

\subsection{The role of damping and stochastic forcing in the reduced Lorenz-96 model}\label{l96params}

In this section, we compare multiple choices of ansatz which are reduced forms of \eqref{lor96reduced} across a wide range of time-scale separations, $2^{-7}\leq \epsilon\leq 1$. The goal of this example is to compare the filtering skill of various stochastic parameterization ansatz when the unresolved scales dynamics in \eqref{lor96} is ignored.  We generate the truth data from \eqref{lor96} and observation data from \eqref{L96obs} using a short observation time $\Delta t = \min\{0.01,\epsilon/10\}$ and an integration time step $\delta t = \frac{\Delta t}{10}$.  The $\epsilon$ dependence in the time step is necessary due to the stiffness of the problem as $\epsilon\to 0$. We use the parameters from Regime 2 of \cite{kh:12} where $N=9$, $J=8$, $a = 1$, $F=10$, $R = 0.1 \times {I}_{N\times N}$, $h_x = -0.1$ and $h_y = 1$.  Note that there are $81$ total variables, only $9$ of which are observed. For diagnostic purpose, we first consider the idealized case where the full model \eqref{lor96} and all the parameters are known exactly.  We apply an ensemble Kalman filter based on the full ensemble transform \cite{bishop:01} with $162$ ensemble members (double the total state variables, $2N(J+1)$), each of which are integrated 10 times between observations.  We will refer to this scheme as the \emph{Full Model}; see Figure \ref{2layerL96} for the average RMSE and consistency measure of the filtered solutions based on this full model.  

Numerically, the reduced filter models are all integrated with $\delta t = \Delta t$, thus using 10 times fewer integration steps than the \emph{Full Model}, since the numerical stiffness disappears when the fast processes are removed. Moreover, since the reduced models have significantly fewer variables, $N=9$, we consider an ensemble of $2N =18$ members (or 20 members, when $\alpha$ is estimated), which is much fewer than the $162$ ensemble members used for the \emph{Full Model}. In this section all $N=9$ slow variables are observed which allows us to estimate the full $9\times 9$ matrix $Q$, however this requires a long time series. Thus, we use a series of 80,000 observations and each filter uses the first 20,000 observations to estimate their parameters so that only the last 60,000 observations are used in the computation of the averages in the RMSE and the Consistency (shown in Figure \ref{2layerL96}).

To evaluate the effectiveness of the additional damping and additive noise in the reduced model we consider four separate cases.  First, we set $\alpha = 0$ and $\sigma = 0$, which we call the \emph{reduced deterministic filter (RDF)} since the slow variables are unchanged and the fast variables have simply been truncated.  As shown in Figure \ref{2layerL96}, the RDF has very poor performance for all but extremely small values of $\epsilon$.  In fact for $\epsilon \geq 0.125$ the truncated model's filtered estimate is actually worse than the observation.  Next we consider the \emph{reduced deterministic filter with an additional damping correction (RDFD)} where $\sigma = 0$ and the \emph{reduced stochastic filter with an additive noise correction (RSFA)} where $\alpha = 0$.  As shown in Figure \ref{2layerL96} the damping improves the filter accuracy for small $\epsilon$ whereas the additive noise stochastic forcing improves the filter accuracy for large $\epsilon$.  Finally we combine both damping and additive stochastic forcing in RSFAD, which shows the improvement that is achievable with this simple stochastic parameterization of model error compared to simply neglecting unresolved scales.   

Of course, estimating the accuracy of the posterior mean filter solutions is only part of filter performance, the filter also quantifies the uncertainty in the mean state estimate $\tilde x(t)$ via the estimated covariance matrix $\tilde S(t)$.  We would like to know if the filter is doing a good job of determining $\tilde S(t)$, however judging the accuracy of $\tilde S(t)$ is difficult since we do not have access to the optimal filter (even our \emph{Full Model} simply uses a Gaussian update).  Thus we compute the empirical measure of filter consistency introduced in Section \ref{nonlinear theory}.  As shown in Section \ref{nonlinear theory} and Appendix D, the consistency quantifies the degree to which the actual error covariance of the suboptimal estimate $\tilde x(t)$ agrees with the filtered covariance estimate, $\tilde S(t)$.  Moreover, if $\tilde x(t)$ is a good estimate of the true posterior mean, consistency close to one implies that $\tilde S(t)$ is close to the true posterior covariance.  
In Figure \ref{2layerL96} we show that consistency is significantly improved by the additive noise term characterized by the parameters $\sigma_{ij}$.  When these stochastic parameters are included, the reduced model is consistent with $\mathcal{C}\approx 1$, compared to the order $10^4$ underestimation of the actual error covariance without this stochastic parameterization.

\begin{figure}
\centering
\includegraphics[width=.45\textwidth]{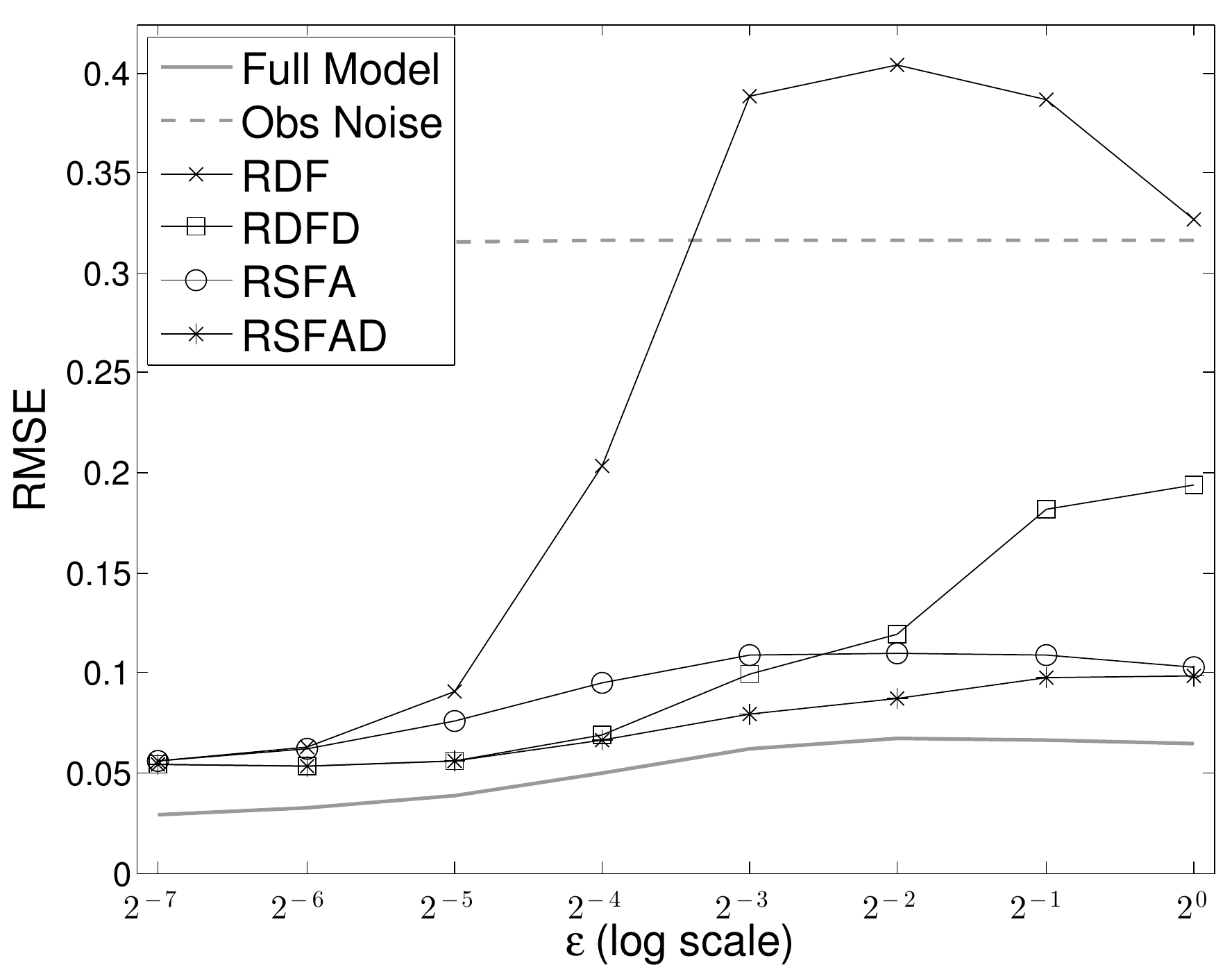}
\includegraphics[width=0.45\textwidth]{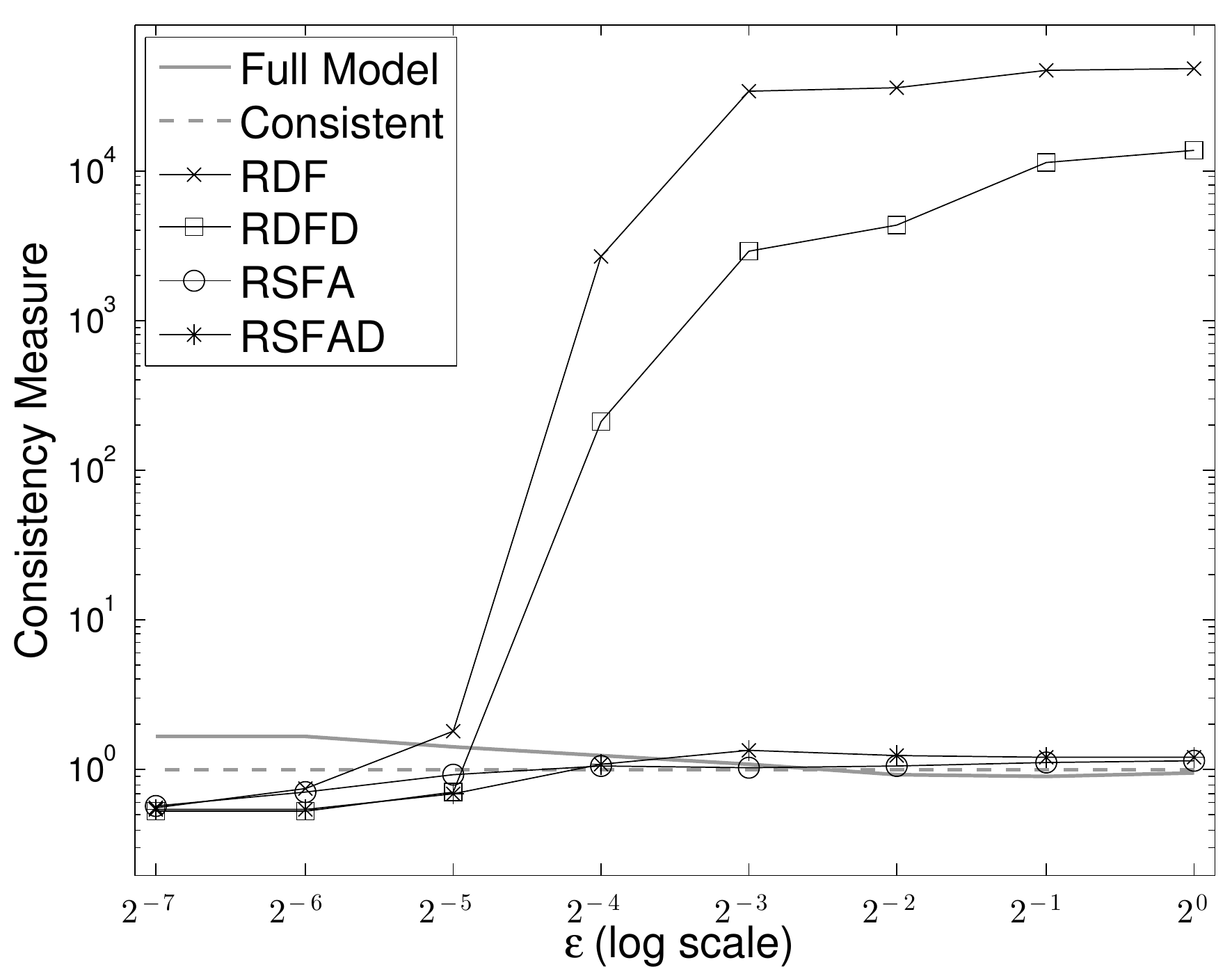}
\caption{\label{2layerL96} Filter performance measured in terms of root mean squared error (RMSE, left) and consistency measure (right) for various time-scale separations $\epsilon$ with an integration time step of $\delta t = \frac{\Delta t}{10}$ and observations at time intervals $\Delta t = \min\{0.01,\epsilon/10\}$ with observation noise $R=0.1$.  Results are overaged over $60,000$ assimilation cycles in Regime 2 from \cite{kh:12} where $N=9$, $J=8$, $a = 1$, $F=10$, $R = 0.1$, $h_x = -0.8$ and $h_y = 1$.  All filters use the ensemble transform Kalman filter with dynamics integrated with the Runge-Kutta (RK4) method.  The \emph{Full Model} filter uses \eqref{lor96}, the same model used to generate the data.  The remaining models use \eqref{lor96reduced} where \emph{RDF} sets $\sigma_{ij} = \alpha  = 0$, \emph{RDFD} sets $\sigma_{ij}=0$, \emph{RSFA} sets $\alpha = 0$ and \emph{RSFAD} fits both parameters simultaneously.}
\end{figure}

\subsection{Comparison of online and offline stochastic parameterization methods}\label{l96compare}

In this section we compare the online parameter estimation scheme in Section~4(a) with the linear regression based offline scheme from \cite{amp:13}. We will consider their parameter regime, which can be written as \eqref{lor96} with parameters $N=8$, $J=32$, $\epsilon=0.25$, $F=20$, $a=10$, $h_x=-0.4$, and $h_y=0.1$. Here, we only consider the regime of \cite{amp:13} with the smaller time-scale separation ($c= \frac{1}{\epsilon}=4$ in their parameters).  In this section we will vary the observation time $\Delta t$ and we use an integration time step of $\delta t= 0.001$ for the full model (to avoid numerical instability) and all the reduced models will use an integration step of $\delta t = 0.005$.  We consider a more challenging regime for the stochastic parameterization with sparse observations of every other grid point of the slow variables so that $M=4$. Each observed variable is perturbed by a mean zero Gaussian noise with variance $0.1$ so that the covariance matrix of the observation noise is $R=0.1\times I_{M\times M}$.  The filter performance is assessed with the average RMSE and the consistency measure, computed over all $N=8$ slow variables and 60,000 assimilation cycles.

We also consider one of the stochastic parameterization ansatz that was studied in \cite{amp:13}, which we refer to as \emph{Cubic+AR(1)},
\begin{align}\label{ArnoldReduced}\frac{dx_i}{dt} &= x_{i-1}(x_{i+1}-x_{i-2}) -  x_i + F - \Big(b_0 + b_1 x_i + b_2 x_i^2 + b_3 x_i^3 + e_i(t) \Big), \nonumber \\
e_i(t) &= \phi e_i(t-\delta t) + \hat{\sigma}(1-\phi^2) z_i(t), \end{align}
where $b_j, \phi, \hat\sigma$ are scalars to be determined from the data and $z_i$ denotes standard i.i.d. white noise.
This model fits the bias component of the model error to a cubic polynomial and the random component of the model error to an AR(1) model. Following \cite{amp:13}, we integrate this model with the stochastic integration method described in \cite{hp:06} with integration step $\delta t = 0.005$ and hold the stochastic term, $e_i(t)$, constant over each integration step. 

The parameters for \eqref{ArnoldReduced} are found \emph{offline} in \cite{amp:13} from a noiseless time series of the slow variables $\{x_i\}_{i=1}^8$. In particular, given a training time series of $x_i(t)$, the offline estimation first constructs
\begin{align} 
U(x_i,t) =  x_{i-1}(t)(x_{i+1}(t)-x_{i-2}(t)) - x_i(t) + F - \left( \frac{x_i(t+\delta t)-x_i(t)}{\delta t} \right), \nonumber \end{align}
which represents the model error from using the truncated one-layer Lorenz-96 model. The errors $U$ are then fit to the cubic polynomial, $U(x_i,t) \approx b_0 + b_1 x_i(t) + b_2 x_i(t)^2 + b_3 x_i(t)^3$ using a standard least squares method. Finally, the residual, $e_i(t) = U(x_i,t) - (b_0 + b_1 x_i(t) + b_2 x_i(t)^2 + b_3 x_i(t)^3)$, is fit with an AR(1) stochastic model to find the parameters $\phi$ and $\hat \sigma$. Since the model is spatially homogeneous, we fit a single set of parameters for all $i$.  
The parameters of \cite{amp:13} are $b_0 = -0.198$, $b_1=0.575$, $b_2=-0.0055$, $b_3=-0.000223$, $\phi = .993$, $\hat \sigma = 2.12$, which we verified using the above procedure, and we use these parameters for the \emph{Cubic+AR(1)} reduced model. We found that the filtered solutions with this model diverge catastrophically (the average RMSE goes to numerical infinity) for the sparse observation. For the fully observed slow dynamics case, we found that the average RMSE of this stochastic parameterization is slightly below the observation error but they were no where close to that of the \emph{Full Model} or the other reduced models considered.

To make a fair comparison between the online and offline parameter estimation schemes, we consider a simplified version of the parametric form in \eqref{ArnoldReduced} with $b_0=b_2=b_3=\phi=0$. This is equivalent to using \eqref{lor96reduced} with a diagonal diffusion matrix, $\sigma_{ij} = \hat{\sigma}\delta(i-j)$. In this sense, both resulting filter models will have only two parameters in their stochastic parameterization, namely $\alpha=b_1$ and $\hat{\sigma}$. For the \emph{Offline Fit}, we obtain the parameters with the same linear regression based offline estimation technique described above as in \cite{wilks:05,amp:13} by fitting to a large data set of $x_i(t)$ ($2\times10^5$ time steps); the resulting parameters are $\alpha = b_1 = 0.481$ and $\hat \sigma = 2.19$. In order to produce online estimates of the reduced model parameters, we ran the adaptive EnKF described in Appendix E using the RSFAD reduced model of \eqref{lor96reduced} on $10^4$ noisy observation of $M=4$ of the $N=8$ slow variables.  The RSFAD scheme estimates the parameters $\alpha$ and $\sigma_{ij}$ on-the-fly, as described in Appendix E, using the cyclic parameterization of $\sigma_{ij}$.  We define the \emph{Online Fit} reduced model by taking $\alpha$ and $\hat{\sigma} = \frac{1}{8}\sum_{i=1}^8 \sigma_{ii}$ from the RSFAD scheme. The parameters from \emph{Online Fit} and \emph{Offline Fit} are shown in Figure \ref{2layerL96comp}. We should emphasize that the parameter estimation scheme of \cite{amp:13} takes place offline, using a noiseless data set of, $x_i(t)$, and fits the deterministic and stochastic parts of the model error separately.  In contrast, the online data assimilation scheme of \cite{bs:13} with RSFAD uses a much shorter and spatially sparse time series of noisy observations without knowing the noise error covariance, $R$, and simultaneously updates the mean and covariance parameters of the reduced stochastic model \eqref{lor96reduced}. 

\begin{figure}
\centering
\includegraphics[width=.32\textwidth]{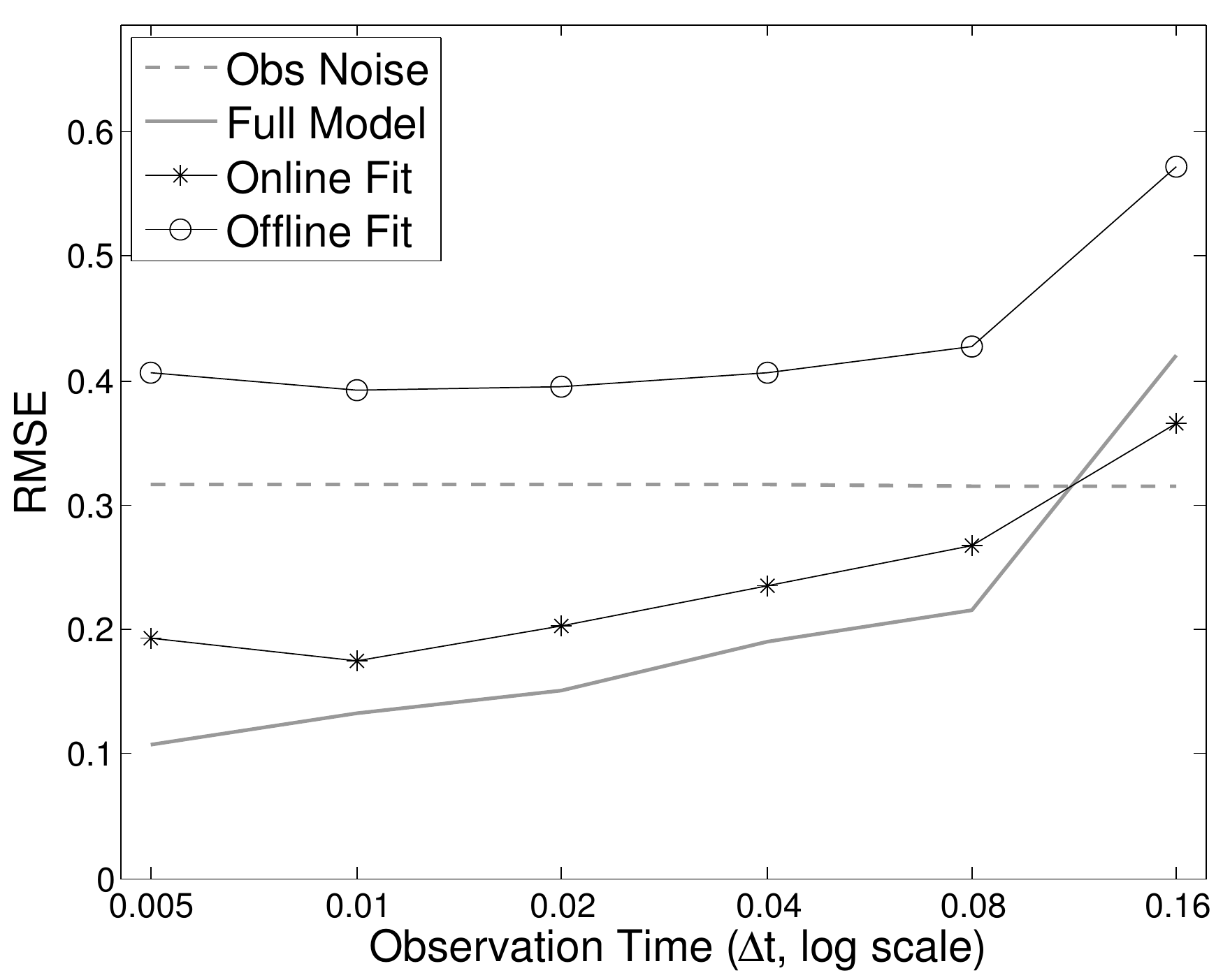}
\includegraphics[width=0.32\textwidth]{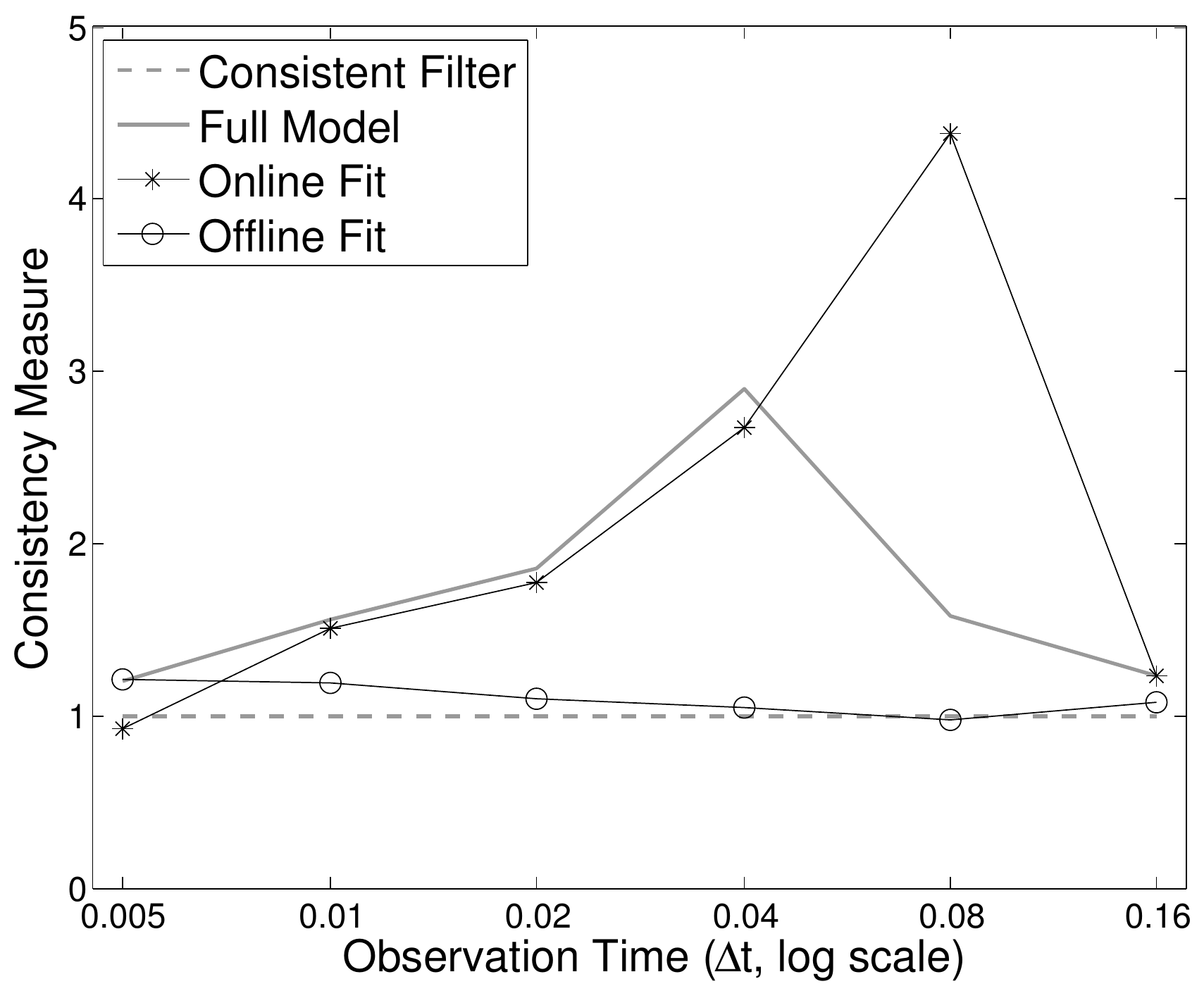}
\includegraphics[width=0.32\textwidth]{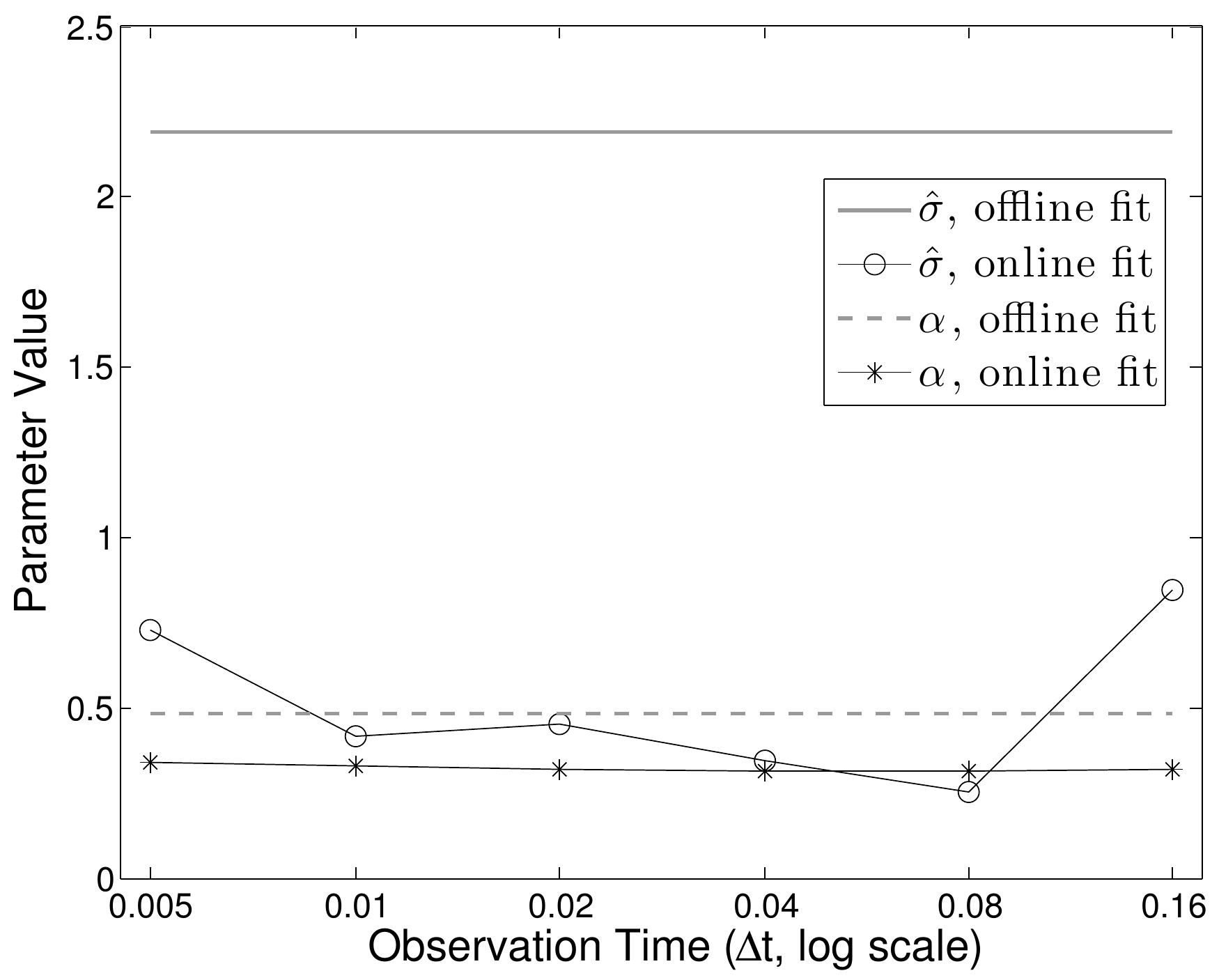}
\caption{\label{2layerL96comp} Filter performance measured in terms of root mean squared error (RMSE, left) and consistency (middle) for sparse observation of the slow variables.  Results are averaged over $7,000$ assimilation cycles.  All filters use the ensemble transform Kalman filter with dynamics integrated with the Runge-Kutta (RK4) method.  The \emph{Full Model} filter uses \eqref{lor96}, the same model used to generate the data.  The \emph{RSFAD} curves use \eqref{lor96reduced} with parameters $\alpha$ and $\sigma_{ij}$ estimated as in Appendix E with the cyclic parameterization of $\sigma_{ij}$.  The \emph{Cubic+AR(1)} model using \eqref{ArnoldReduced} with the parameters of \cite{amp:13} is not shown since the filtered diverged.  The \emph{Online Fit} curves use \eqref{ArnoldReduced} with $b_1=\alpha$ and $\hat \sigma = \frac{1}{8}\sum_{i=1}^8 \sigma_{ii}$ taken from the estimates produced from the RSFAD and the remaining parameters are $b_0=b_2=b_3=\phi=0$.  In the figure on the far right we compare the $\alpha$ and $\hat \sigma$ parameters from the online and offline estimation techniques.  The \emph{Offline Fit} curves use parameters $b_1 = 0.481$ and $\hat \sigma = 2.19$ estimated using the technique of \cite{amp:13}.}
\end{figure}

\begin{figure}
\centering
\includegraphics[width=.45\textwidth]{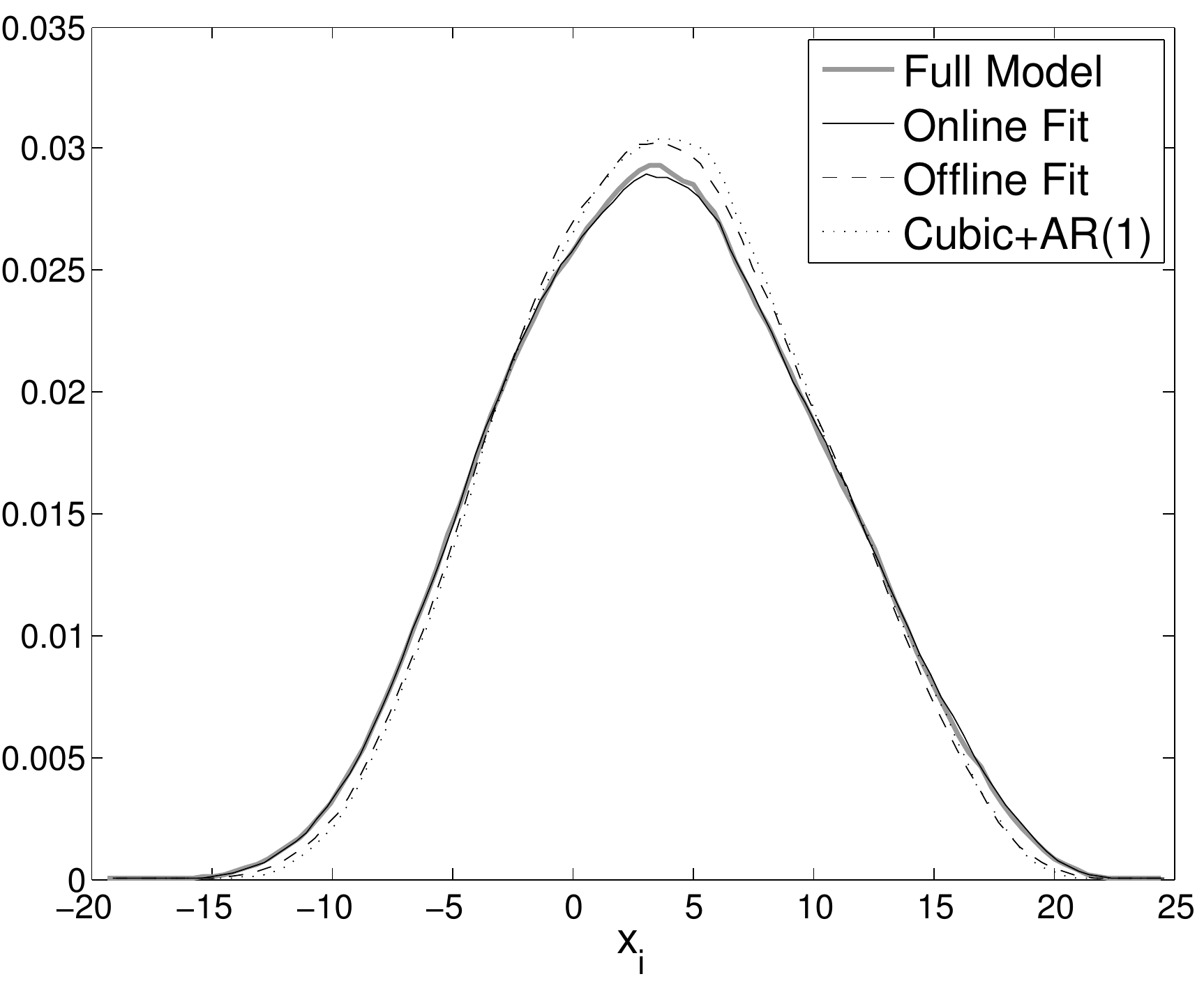}
\includegraphics[width=0.435\textwidth]{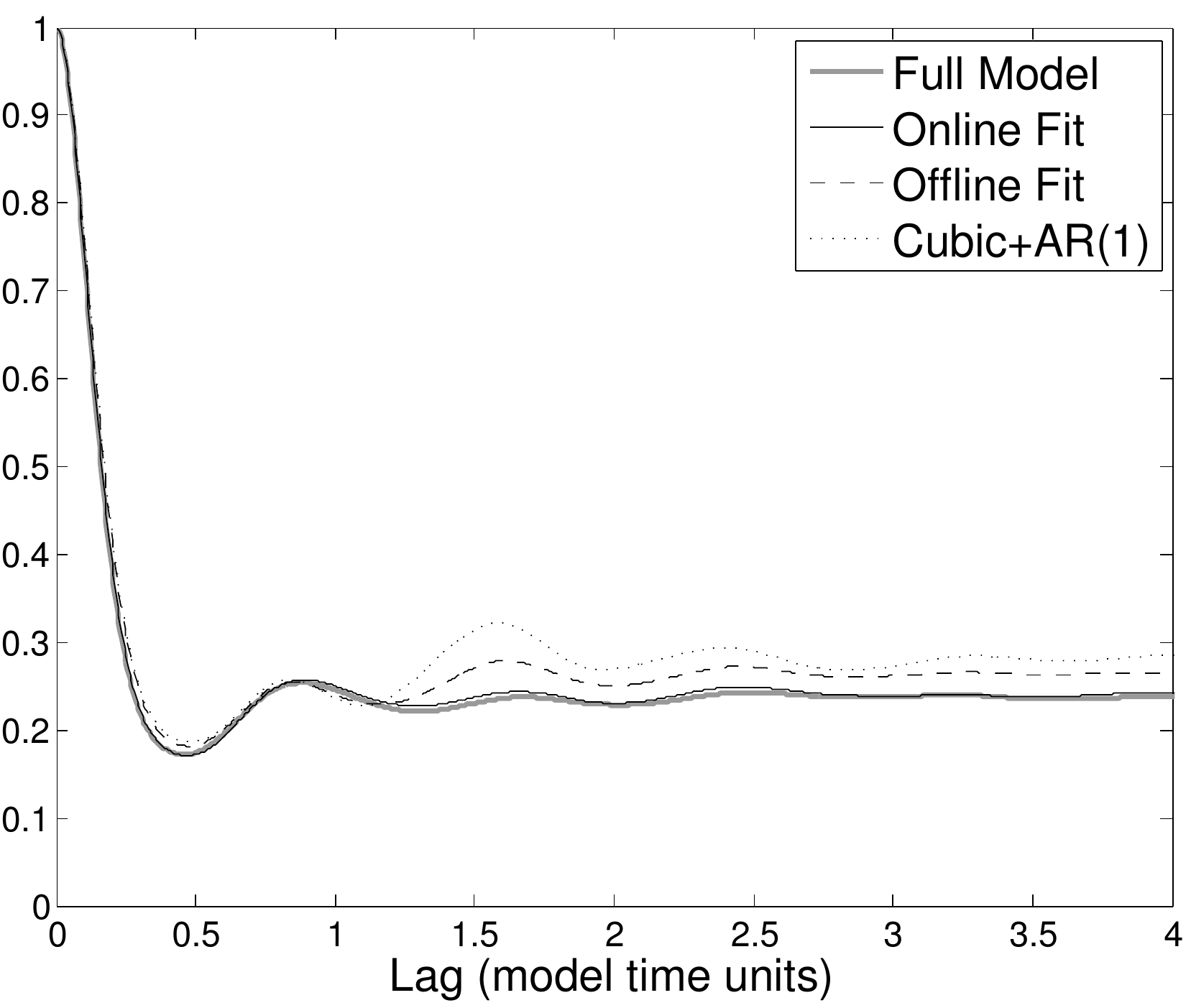}
\caption{\label{pdfcorr} Climatological forecast performance is compared in terms of the invariant measure of the slow variables shown as a probability density function (left) and the autocorrelation as a function of lag steps of length $0.005$ (right). Each curve is computed from a long free run with $1.6\times 10^7$ data points.  The \emph{Full Model} filter uses \eqref{lor96}.  The \emph{Cubic+AR(1)} curves use \eqref{ArnoldReduced} with the parameters of \cite{amp:13}.  The \emph{Online Fit} curves use \eqref{ArnoldReduced} with $b_1=\alpha$ and $\hat \sigma = \frac{1}{8}\sum_{i=1}^8 \sigma_{ii}$ taken from the estimates produced from the RSFAD and the remaining parameters are $b_0=b_2=b_3=\phi=0$.  The \emph{Offline Fit} uses the procedure for offline parameter estimation from \cite{amp:13} to find $b_1$ and $\hat \sigma$ and sets all the other parameters to zero.}
\end{figure}

In Figure \ref{2layerL96comp} we compare the performance on the filtering problem and we see that the \emph{Offline Fit} parameters give worse performance than the observation in terms of RMSE.  On the other hand, the \emph{Online Fit} parameters produce filtered solutions with a RMSE which is relatively close to that of the full model. 
Notice that the consistency of the \emph{Offline Fit} is very good, which agrees with the results in \cite{amp:13} which compares ensemble spread to ensemble error in the prior model.  However, as shown in Appendix D, a good consistency result is meaningless when the mean estimate is not accurate; so while the actual error and filter error estimate agree, they are both very large.  In contrast, the \emph{Online Fit} is underestimating the covariance slightly (compare the scale of the y-axis to that in Figure \ref{2layerL96}) but the RMSE and consistency are close to those of the \emph{Full Model}.  Moreover, in order to make a fair comparison, the \emph{Online Fit} only uses the diagonal part of the covariance matrix estimated by the RSFAD and the additional covariance of the off-diagonal terms is probably needed to produce a more consistent filter.  In Figure \ref{pdfcorr} we compare the equilibrium marginal density and correlation function of the \emph{Online Fit} and \emph{Offline Fit} to those of the slow variables of the full model.  In this regime the \emph{Online Fit} produces very good agreement with both the equilibrium density and the correlation function over a very long time (note that $4$ model time units corresponds to $800$ integration steps for the reduced model).  In contrast, the \emph{Offline Fit} and even the full \emph{Cubic+AR(1)} models showed some deviations, notably underestimating the variance and overestimating the lag correlations at the later times.

Since the \emph{Online Fit} gives good filter performance and also closely matches the equilibrium statistics of the full model, we conclude that the ansatz \eqref{lor96reduced} is sufficient for this regime of the two-layer Lorenz-96.  We should emphasize that there is no inherent problem with offline parameter estimation.  The problem with the linear regression based estimation scheme of \cite{amp:13} is that the deterministic parameter, $\alpha=b_1$, and diffusion amplitude, $\hat{\sigma}$, in the stochastic parameterization model are estimated separately. So, when a parameter in \eqref{lor96reduced} is independently perturbed, the nonlinear feedback of this perturbation is not appropriately accounted in the filtered estimates.  A successful offline parameter estimation scheme would need to simultaneously account for all of these nonlinear feedback relationships, rather than with two separate least squares estimates.


\section{Summary and discussion}\label{discussion}

In this paper, we studied two simple examples to understand how model error from unresolved scales affects the state estimation and uncertainty quantification of multiscale dynamical systems, given noisy observations of all the resolved scale components alone. From the mathematical analysis of these simple test problems, we learned that for a continuous time linear model with Gaussian noise, there exists a unique choice of parameters in a linear reduced model for the slow variables which gives optimal filtering when only the slow variables are observed.  Moreover, these parameters simultaneously gives the best equilibrium statistical estimates, and as a consequence they can be estimated offline from equilibrium statistics of the true signal.  In particular this shows that in the linear setting the Mean Stochastic Model (MSM) introduced in \cite{mgy:10,mh:12} is the optimal reduced model.  

By examining the continuous-time nonlinear SPEKF problem, we showed that the linear theory extends to this non-Gaussian, nonlinear configuration as long as we know the optimal stochastic parameterization ansatz and there is no error in the observation model. We confirmed this finding by noticing that the stochastic terms in the Kushner equations do not produce additional constraints to determine the reduced model parameters as long as the observation model has no additional error. This implies that one would get the same parameters by matching the first two-moments of the corresponding Fokker-Planck equations (ignoring the stochastic terms in the Kushner equations). Although we only show the Gaussian closure approximate filter in this paper, we found no additional constraints when the Gaussian closure approximation is ignored (the coefficients of the higher moments satisfy the same constraints). Numerically, we show that the additional correction terms that we found in our formal asymptotic expansion produces accurate filtering as well as accurate long term covariance prediction.  Moreover, we reinforce this result numerically on a complex nonlinear system in Section \ref{L96} by numerical estimation of a reduced stochastic model for the two-layer Lorenz-96 model.  Once again, given the ``right" stochastic parameterization (chosen based on our analysis and the earlier work \cite{mtv:01,mfc:09}) we find a single set of parameters which simultaneously produces reasonably accurate filter estimates and equilibrium statistical estimates.  

When the stochastic parameterization ansatz is insufficient, parameters chosen for good filtering may give poor equilibrium statistics and vice versa.  This is shown analytically and numerically in Section \ref{nonlinear theory} for the SPEKF filtering problem when the stochastic parameterization does not include a multiplicative stochastic forcing. This is also reinforced numerically in Section \ref{L96}(\ref{l96params}) for the two-layer Lorenz-96 model, where we show that neither linear damping (RDFD) nor additive stochastic forcing (RSFA) alone are sufficient to give accurate filtering across multiple time-scale separations, and the ``right" parameterization requires both of these terms (RSFAD).  Moreover, in Section \ref{L96}(\ref{l96compare}) we show that the parameters estimated by RSFAD match the equilibrium statistics of the full model and give good filter performance even for sparse observations at long observation times with small time-scale separation.

Finally, even when the correct stochastic ansatz is known, it is imperative to estimate the parameters simultaneously and to account for the nonlinear feedback of the stochastic parameters in the parameter estimation technique.  In particular, in Section \ref{L96}(\ref{l96compare}) we compare an offline parameter estimation technique proposed in \cite{wilks:05} and used by \cite{amp:13}, to an online parameter estimation technique introduced in \cite{bs:13}.  We find that the online parameters give good filter performance and match the equilibrium statistics whereas the offline parameters yield filter estimates which are worse than the observations. In our numerical results, the online stochastic parameter estimation scheme produces extremely accurate filter and equilibrium statistical estimates even when the data assimilation method only takes noisy observations of only half of the slow variables, while the offline technique uses a much longer data set of noiseless observations of all of the slow variables. 

The weakness of the offline technique that we tested \cite{wilks:05,amp:13} is that the deterministic and stochastic parameters are estimated separately based on a linear regression fitting on a training data set. Such a scheme does not account for the feedback of these parameters on the reduced model, which is particularly important when the underlying dynamics are nonlinear.  We emphasize that offline estimation methods are not inherently worse than online estimation, however a successful estimation technique must estimate all parameters simultaneously. As pointed out in \cite{bm:14}, the design of an adequate offline estimation method for accurate filtering with model error may involve solving a constrained optimization problem with the constraints given by the three information measures to account for the statistics of the filter error, including the mean biases, and the correlations between the truth and the filter estimates. It would be interesting to see whether such method produces parameters that give accurate equilibrium statistical estimates, assuming that the optimal stochastic parameterization is known. On the other hand, there are, for example, other offline methods based on information theoretic criteria which were shown to improve long term uncertainty quantification \cite{bm:12,mb:12, mg:10,mg:11a, mg:11b}; it would also be interesting to check whether these methods can also give accurate filtered solutions when the optimal stochastic ansatz is given. In the linear and Gaussian setting, the Mean Stochastic Model in \cite{mgy:10,mh:12} is an example of an offline estimation method that produces optimal filtering as well as equilibrium statistical prediction. In this special setting, the offline technique is sufficient because of two special circumstances, first, the parameters can be completely characterized by two equilibrium statistics and second, the underlying signal is stationary. In a nonlinear setting, the online technique gives a natural way for the parameters to evolve via small perturbations while the feedback from these perturbations on the reduced model are compared to the observations, adaptively. Moreover, it does not require a training data set and can work directly with the noisy observations as they become available. 

The online parameter estimation methods also have some weaknesses, they sometimes are limited by issues of observability of parameters (see Chapter 13 of \cite{mh:12}) and they also required careful design to avoid expensive computational costs \cite{dcdg:85,bs:13,hmm:14}. Extending these techniques to more general stochastic parameterizations and high-dimensional problems is therefore an important task for future research. To make stochastic parameterization more useful in real applications, many open problems remain to be solved, including how to determine the ``right" stochastic parameterization if the one we proposed is inadequate, and how to simultaneously estimate the multiplicative noise coefficients in addition to the linear damping and additive noise considered in Section \ref{L96}. Another important issue is to understand whether the linear theory holds when the observation model also depends on the fast variables. We plan to address these questions in our future research.

\section*{Acknowledgments}
The research of J.H. is partially supported by the Office of Naval Research Grants N00014-11-1-0310, N00014-13-1-0797, MURI N00014-12-1-0912 and the National Science Foundation DMS-1317919. T. B. is supported by the ONR MURI grant N00014-12-1-0912.

\appendix{Expansion of the true filter covariance}\label{appcovarianceExpansion}
 
We consider the two-dimensional linear and Gaussian filtering problem (written in matrix form),
\begin{align}
\label{appfullmodelmatrix} 
d{\vec x} &= A_{\epsilon}\vec x \,dt+ \sqrt{Q_{\epsilon}}dW,\\
dz &= H\vec x dt + \sqrt{R}dV,\nonumber 
\end{align}
where $\vec x = (x,y)^\top$, $H=[1,0]$, $W = (W_x, W_y)^\top$ and $V$ are standard i.i.d. Wiener processes,
\begin{align} 
A _{\epsilon} = \left( \begin{array}{cc} a_{11} & a_{12} \\ a_{21}/\epsilon & a_{22}/\epsilon \end{array}\right), \hspace{15pt}\textup{ }\hspace{15pt} Q_{\epsilon} = \left(\begin{array}{cc} \sigma_x^2 & 0\\ 0 &\sigma_y^2/\epsilon \end{array}\right). \nonumber
\end{align}
Throughout this analysis, we assume that $A_1$ has negative eigenvalues, $a_{22}<0$, and $\tilde{a}\equiv a_{11}-a_{12}a_{22}^{-1}a_{21}<0$. The optimal filter posterior statistics, $\hat{\vec x} = \mathbb{E}\big[\vec x \big]$ and $\hat S = \mathbb{E}\big[(\vec x-\hat{\vec x})(\vec x-\hat{\vec x})^\top \big]$, are given by the Kalman-Bucy equations \cite{kalman:61}.  In this appendix we find the steady state covariance $\hat s_{11} = \mathbb{E}[(x-\hat x)^2]$ for the slow variable $x$ and we expand the solution in terms of $\epsilon$.  We show that up to order-$\epsilon^2$, the covariance $\hat s_{11}$ solves a one-dimensional Riccati equation which will motivate the optimal reduced model in \ref{appoptimalreduced}.
 
The Kalman-Bucy solution implies that $\hat S$ has a steady state solution given by the algebraic Riccati equation,
\begin{align}\label{appriccatiappendix}
0 = A_{\epsilon}\hat S + \hat SA_{\epsilon}^\top - \hat SH^\top R^{-1}H\hat S +Q_{\epsilon}.
\end{align}
 For the case of the two variable system, by setting $\hat S = \left(\begin{array}{cc} \hat s_{11} & \hat s_{12} \\ \hat s_{12} & \hat s_{22} \end{array}\right)$ the steady state Riccati equation yields the following three equations,
\begin{align}
0 &= \sigma_x^2 - \hat s_{11}^2/R + 2a_{11}\hat s_{11} + 2a_{12}\hat s_{12}, \nonumber \\
0 &= a_{11}\hat s_{12} + a_{12}\hat s_{22} - \hat s_{11}\hat s_{12}/R + \hat s_{11}a_{21}/\epsilon + \hat s_{12}a_{22}/\epsilon , \nonumber \\
0 & = \sigma_y^2/\epsilon - \hat s_{12}^2/R + 2\hat s_{12}a_{21}/\epsilon + 2\hat s_{22}a_{22}/\epsilon . \nonumber
\end{align}
Solving the third equation for $\hat s_{22}$ and plugging the result into the second equation yields,
\begin{align}
0 = \left(\frac{\epsilon a_{12}}{2Ra_{22}} \right) \hat s_{12}^2 + \left(a_{11} - \frac{a_{12}a_{21}}{a_{22}} - \frac{\hat s_{11}}{R} + \frac{a_{22}}{\epsilon}\right) \hat s_{12} + \left( \frac{\hat s_{11}a_{21}}{\epsilon} - \frac{a_{12}\sigma_y^2}{2a_{22}} \right). \nonumber
\end{align}
Multiplying this expression by $\epsilon$, for $a_{22}\neq 0$ we obtain,
\begin{align}\label{apps12eq}
\hat s_{12} &= \frac{ -\hat s_{11}a_{21} + \epsilon \left(\frac{a_{12}\sigma_y^2}{2a_{22}} \right) }{ a_{22} + \epsilon ( a_{11} -a_{12}a_{21}/a_{22} - \hat s_{11}/R )} + \mathcal{O}(\epsilon^2), \nonumber \\
&= -\hat s_{11}\frac{a_{21}}{a_{22}} + \epsilon \frac{ a_{12}\sigma_y^2/2 + \hat s_{11}a_{21}( a_{11} -a_{12}a_{21}/a_{22} - \hat s_{11}/R )}{a_{22}^2} + \mathcal{O}(\epsilon^2),
\end{align}
where we have used $\frac{a+b\epsilon}{c+d\epsilon} = \frac{a}{c} + \frac{bc-ad}{c^2} \epsilon + \mathcal{O}(\epsilon^2)$.  Plugging this solution for $\hat s_{12}$ into the first Riccati equation above gives the following equation for $\hat s_{11}$,
\begin{align}
0 &= -\left( \frac{1}{R} + \epsilon \frac{2a_{12}a_{21}}{a_{22}^2 R} \right)\hat s_{11}^2 
+ 2 \left(a_{11}  - \frac{a_{12}a_{21}}{a_{22}}  + \epsilon\left( \frac{a_{21}a_{12}a_{11}}{a_{22}^2} - \frac{a_{21}^2a_{12}^2}{a_{22}^3}\right)  \right) \hat s_{11} 
+ \left(  \sigma_x^2 + \epsilon \frac{a_{12}^2\sigma_y^2}{a_{22}^2}  \right) + \mathcal{O}(\epsilon^2), \nonumber \\
&=-\left(\frac{1+2\epsilon\hat a}{R}\right)\hat s_{11}^2 
+ 2\tilde a \left(1 + \epsilon\hat a \right) \hat s_{11} 
+ \left(  \sigma_x^2 + \epsilon \sigma_y^2 \frac{a_{12}^2}{a_{22}^2}  \right) + \mathcal{O}(\epsilon^2), \nonumber
\end{align}
where $\tilde a = a_{11} - \frac{a_{12}a_{21}}{a_{22}}$ and $\hat a = \frac{a_{12}a_{21}}{a_{22}^2}$.  Dividing both sides by $(1+2\epsilon \hat a)$ and expanding in $\epsilon$ we have,
\begin{align}\label{appfullmodelcov}
0 &=-\frac{\hat s_{11}^2}{R} + 2\tilde a (1-\epsilon\hat a)  \hat s_{11} + \left( \sigma_x^2(1-2\epsilon\hat a)  + \epsilon \sigma_y^2 \frac{a_{12}^2}{a_{22}^2}  \right) + \mathcal{O}(\epsilon^2).
\end{align}
Equation \eqref{appfullmodelcov} yields the variance $\hat s_{11}$ of the state estimate for the observed variable $x$ based on the optimal filter using the full model.  Note that by truncating terms that are order $\epsilon^2$, equation \eqref{appfullmodelcov} for $\hat s_{11}$ has the form of a Riccati equation for a one dimensional dynamical system.  In particular, consider the linear one-dimensional filtering problem,
\begin{align}
dX &=  aX\,dt+\sigma_X \,dW_X, \nonumber\\
dz &= X\,dt + \sqrt{R}\, dV.\label{appreducedmodel}
\end{align}
The steady state covariance $\tilde s$ from the Kalman-Bucy solution to \eqref{appreducedmodel} solves the Riccati equation, 
\begin{align}\label{appriccati1d} -\frac{\tilde s^2}{R} + 2a\tilde s + \sigma_X^2 = 0. \end{align}
Our goal is to find $\{a, \sigma_X\}$ such that the solution of \eqref{appreducedmodel} agrees with the solution for $x$ of \eqref{appfullmodelmatrix}.  In order to make $\tilde s$ agree with $\hat s_{11}$, we establish the following tradeoff between $a$ and $\sigma_X^2$ in the limit as $\epsilon \to 0$.
\begin{customthm}{2.1}\label{appthm1}
Let $\hat s_{11}$ be the first component of the steady state solution to \eqref{appriccatiappendix} and let $\tilde s$ solve \eqref{appriccati1d}.  Then $\lim_{\epsilon\to 0} \frac{\tilde s-\hat s_{11}}{\epsilon} = 0$ if and only if 
\[ \sigma_X^2 = -2(a-\tilde a(1-\epsilon\hat a))\hat s_{11} + \sigma_x^2(1-2\epsilon\hat a) + \epsilon\sigma_y^2\frac{a_{12}^2}{a_{22}^2}  + \mathcal{O}(\epsilon^2). \]
\end{customthm}

\begin{proof}
Subtracting \eqref{appriccati1d} from \eqref{appfullmodelcov}, we obtain
\begin{align}\label{appp1eq1} -\frac{\hat s_{11}^2 - \tilde s^2}{r} + 2\tilde a(1-\epsilon\hat a)\hat s_{11} - 2a\tilde s +  \left(  \sigma_x^2(1-2\epsilon\hat a) + \epsilon \sigma_y^2 \frac{a_{12}^2}{a_{22}^2}  \right) - \sigma_X^2 = \mathcal{O}(\epsilon^2). \end{align}
First, assuming the $\sigma_X^2$ has the form given in the statement, \eqref{appp1eq1} reduces to 
\[  \mathcal{O}(\epsilon^2) = -\frac{\hat s_{11}^2 - \tilde s^2}{R} + 2a(\hat s_{11} - \tilde s) = (\hat s_{11}-\tilde s)\left(- \frac{\hat s_{11}+\tilde s}{R} + a \right),  \]
which shows that $\hat s_{11}-\tilde s = \mathcal{O}(\epsilon^2)$ so $\lim_{\epsilon\to 0} \frac{\hat s_{11}-\tilde s}{\epsilon} = 0$.  Conversely, if we assume that $\lim_{\epsilon\to 0} \frac{\hat s_{11}-\tilde s}{\epsilon} = 0$ then we can rewrite \eqref{appp1eq1} as
\begin{align} 0 &= -(\hat s_{11}-\tilde s)\frac{\hat s_{11} + \tilde s}{R} + 2\tilde a(1-\epsilon\hat a)\hat s_{11} -2a \hat s_{11} +  \left(  \sigma_x^2(1-2\epsilon\hat a) + \epsilon \sigma_y^2 \frac{a_{12}^2}{a_{22}^2}  \right) - \sigma_X^2 + \mathcal{O}(\epsilon^2), \nonumber\\
&= (2\tilde a(1-\epsilon\hat a) - 2a) \hat s_{11} +  \left(  \sigma_x^2(1-2\epsilon\hat a) + \epsilon \sigma_y^2 \frac{a_{12}^2}{a_{22}^2}  \right)  - \sigma_X^2 + \mathcal{O}(\epsilon^2), 
\end{align}
and solving for $\sigma_X^2$ gives the desired identity.
\end{proof}


\appendix{Existence and uniqueness of the optimal reduced model}\label{appoptimalreduced}

In this appendix we consider using the one-dimensional filtering scheme \eqref{appreducedmodel} to filter noisy observations of $x$ that solves the two-dimensional model in \eqref{appfullmodelmatrix}.  We will show that there exists a unique choice of parameters $\{a,\sigma_X\}$ which gives the optimal filtered estimate of the slow variable $x$ from \eqref{appfullmodelmatrix} in the sense that both the mean and covariance estimates match the true filtered solutions and the equilibrium statistics also match those of the full prior model.  This optimal choice is determined by requiring the parameters to lie on the manifold defined by Theorem \ref{appthm1}, and additionally, requiring that the equilibrium covariance implied by the reduced filter model to match the equilibrium covariance of the full model. These two constraints will determine the parameters $\{a,\sigma_X\}$ up to order-$\epsilon^2$.  We also show that as long as the parameters $\{a,\sigma_X\}$ agree with the optimal parameters up to order-$\epsilon$ the resulting filter is consistent consistent in the sense that the \emph{actual error covariance} of the filtered mean estimate must equal to the \emph{error covariance estimate} produced by the filtering scheme.  Moreover, the reduced filter mean estimate will agree with mean estimate for the slow variable from the full filter path-wise up to order-$\epsilon$.

The main goal of this appendix will be to prove the following theorem.
\begin{customthm}{2.2}\label{appthm2} 
There exists a unique choice of parameters given by $a = \tilde a(1-\epsilon \hat a)$ and $\sigma_X^2$ according to Theorem \ref{appthm1}, such that the steady state reduced filter \eqref{appreducedmodel} is both consistent and optimal up to order-$\epsilon^2$.  This means that $\tilde s$, the steady state covariance estimate of the reduced filter, is consistent with the steady state actual error covariance $E_{11} = \lim_{t\to\infty}\mathbb{E}[(x(t)-\tilde x(t))^2]$ so that $\tilde s = E_{11} + \mathcal{O}(\epsilon^2)$, and also $\tilde s$ agrees with the steady state covariance $\hat s_{11}$ from the optimal filter $\tilde s = \hat s_{11} + \mathcal{O}(\epsilon^2)$.
The unique optimal parameters can also be determined by requiring the covariance of the reduced model to match that of the slow variable from the full model up to order-$\epsilon^2$.
\end{customthm}
\begin{proof}
We introduce the following three-dimensional SDE which governs the joint evolution of the full state $\vec x = (x,y)$ and the reduced filter estimate of the state $\tilde x$.  Note that the reduced filter with parameters $\{a,\sigma_X\}$ is given by 
\begin{align} \label{appreducedfilterAnsatz}
d\tilde{x} = a \tilde{x} \,dt + K(dz - \tilde{x}\,dt) = (a-K)\tilde x dt + Kx dt + K\sqrt{R}dV,
\end{align}
where $K = \tilde s/R$ and $\tilde s$ is determined by $\{a,\sigma_X\}$ since $\tilde s$ solves \eqref{appriccati1d} the time-dependent Riccati equation,  
\begin{align}\frac{d\tilde s}{dt} = -\frac{\tilde s^2}{R} + 2a\tilde s + \sigma_X^2. \nonumber\end{align}
Our goal is to determine the parameters $\{a,\sigma_X\}$ which yield a optimal reduced filter.  To find this we combine the full model from \eqref{appfullmodelmatrix} with the reduced filter equation in \eqref{appreducedfilterAnsatz} which yields,
\begin{align}\label{apppriorAndFilter}
\left(\begin{array}{c} dx \\ dy \\ d\tilde x \end{array}\right) =  \left(\begin{array}{ccc} a_{11} & a_{12} & 0 \\ a_{21}/\epsilon & a_{22}/\epsilon & 0 \\ K & 0 & a-K \end{array}\right) \left(\begin{array}{c} x\\ y \\ \tilde x \end{array}\right) + \left(\begin{array}{ccc} \sigma_x & 0 & 0 \\ 0 & \sigma_y/\sqrt{\epsilon} & 0 \\  0 & 0 & K\sqrt{R} \end{array}\right) \left(\begin{array}{c} dW_x \\ dW_y \\ dV \end{array}\right).
\end{align}
Writing the Lyapunov equation for the covariance matrix $E = \mathbb{E}[(x(t)\, y(t) \, \tilde x(t))^\top (x(t)\, y(t)\, \tilde x(t))]$ we find the $e_{11} = c_{11}$, $e_{12}=c_{12}$, and $e_{22} = c_{22}$ as in \eqref{appceq}.  Since $E$ is symmetric there are three remaining variables $e_{13}, e_{23}$ and $e_{33}$ which satisfy the remaining three equations from the Lyapuonov equation which are,
\begin{align}\label{applyapFilter}
0 &= a_{11}e_{13} + a_{12}e_{23} + e_{13}(a-\tilde{s}/R) + e_{11}\tilde{s}/R  \nonumber \\
0 &= a_{21}e_{13}/\epsilon + a_{22}e_{23}/\epsilon + e_{23}(a-\tilde{s}/R) + e_{12} \tilde{s}/R  \nonumber \\
0 &= \tilde{s}^2/R + 2e_{33}(a-\tilde{s}/R) + e_{13} 2\tilde{s}/R,
\end{align}
where we have substituted $K=\tilde s/R$.  Solving the second equation of \eqref{applyapFilter} for $e_{23}$ we find, 
\[ e_{23} = -\frac{a_{21}}{a_{22}}e_{13} + \epsilon \frac{a_{21}}{a_{22}^2} ( e_{11}\tilde s/R + (a-\tilde s/R)e_{13}) + \mathcal{O}(\epsilon^2), \]
and substituting this expression for $e_{23}$ into the first equation of \eqref{applyapFilter} we have,
\[ 0 = \tilde a e_{13} + (a-\tilde s/R)(1+\epsilon\hat a) e_{13} + \tilde s/R(1+\epsilon\hat a) e_{11} + \mathcal{O}(\epsilon^2) \]
and solving for $e_{13}$ we find that,
\begin{align}\label{appe13soln} e_{13} = \frac{-\tilde s e_{11}(1+\epsilon\hat a)/R}{\tilde a + (a-\tilde s/R)(1+\epsilon \hat a)} + \mathcal{O}(\epsilon^2) =  \frac{-\tilde s e_{11}}{\tilde a(1-\epsilon \hat a)R + a R - \tilde s}+ \mathcal{O}(\epsilon^2) 
\end{align}
At this point we must choose one more constraint to obtain a unique choice of parameters $\{a,\sigma_X\}$ up to order-$\epsilon^2$. Unfortunately, the consistency condition $E_{11} = \tilde s$, is too weak; it only specifies the choice of parameters up to order-$\epsilon$. From the general linear theory of Hilbert spaces, the optimal filter mean estimate in the sense of least squares is given by the orthogonal projection onto the subspace spanned by its innovations (see Theorem 6.1.2 and the discussion in Section 6.2 in \cite{oksendal:03}). This condition implies that the actual error, $e=x-\tilde{x}$ is orthogonal to the estimate $\tilde{x}$ under the joint probability distribution for $(W_X,V)$, that is, $\mathbb{E}(e\,\tilde{x})=0$.  Thus we introduce the optimality condition $\mathbb{E}[(x-\tilde x)\tilde x] = 0$, this is the optimality condition in the sense of Hilbert space projections.  We first find the manifold of parameters determined by the optimality condition, note that,
\begin{align} \label{appE11soln} \mathbb{E}[(x-\tilde x)\tilde x] &= e_{13} - e_{33} = \frac{\tilde s^2}{2(aR-\tilde s)} + e_{13}\left(1+\frac{\tilde s}{aR-\tilde s}\right) + \mathcal{O}(\epsilon^2) \nonumber \\
&= \frac{\tilde s^2}{2(aR-\tilde s)} + e_{13} \frac{aR}{aR-\tilde s}  + \mathcal{O}(\epsilon^2) \nonumber \\
&=  \frac{\tilde s^2}{2(aR-\tilde s)} - \frac{\tilde s e_{11} a}{(\tilde a(1-\epsilon\hat a) + a -\tilde s/R)(aR-\tilde s)}  + \mathcal{O}(\epsilon^2) .
\end{align}
Applying the optimality condition $\mathbb{E}[(x-\tilde x)\tilde x] = 0$ we find that,
\begin{align} 0 &= \tilde s^2(\tilde a(1-\epsilon\hat a) + a -\tilde s/R) - 2\tilde s e_{11} a  + \mathcal{O}(\epsilon^2)  \nonumber \\
&= \tilde s\left( \tilde a(1-\epsilon\hat a)\tilde s + a\tilde s -\tilde s^2/R + \frac{a}{\tilde a(1-\epsilon\hat a)} \sigma_s^2  \right)  + \mathcal{O}(\epsilon^2)  \nonumber \\
&=  \tilde s\left( (\tilde a(1-\epsilon\hat a) - a)\tilde s -\sigma_X^2 + \frac{a}{\tilde a(1-\epsilon\hat a)}\sigma_s^2 \right)  + \mathcal{O}(\epsilon^2), 
\end{align}
where $e_{11} = c_{11}$ from \eqref{appc11soln} below.  Thus the optimality condition is defined by the manifold of parameters,
\begin{align} \label{appoptimalityManifold} \sigma_X^2 = (\tilde a(1-\epsilon\hat a) - a)\tilde s + \frac{a}{\tilde a(1-\epsilon\hat a)}\sigma_s^2  + \mathcal{O}(\epsilon^2) 
\end{align}
and combining this with the manifold of Theorem \ref{appthm1} we have
\begin{align} 0 &= 2(a-\tilde a(1-\epsilon\hat a))\tilde s - \sigma_s^2 + (\tilde a(1-\epsilon\hat a) - a)\tilde s + \frac{a}{\tilde a(1-\epsilon\hat a)}\sigma_s^2  + \mathcal{O}(\epsilon^2)  \nonumber \\
&= (a-\tilde a(1-\epsilon\hat a))\tilde s + \sigma_s^2\frac{a-\tilde a(1-\epsilon\hat a)}{\tilde a(1-\epsilon\hat a)}  + \mathcal{O}(\epsilon^2)  \nonumber \\
&= (a-\tilde a(1-\epsilon\hat a))\left(\frac{\tilde a(1-\epsilon\hat a)\tilde s + \sigma_s^2}{\tilde a(1-\epsilon\hat a)}\right)  + \mathcal{O}(\epsilon^2)  \nonumber \\
&= (a-\tilde a(1-\epsilon\hat a))\left(\frac{\tilde a(1-\epsilon\hat a)\tilde s + \sigma_X^2 + 2a\tilde s - 2\tilde a(1-\epsilon\hat a)\tilde s}{\tilde a(1-\epsilon\hat a)}\right)  + \mathcal{O}(\epsilon^2)  \nonumber \\
&= (a-\tilde a(1-\epsilon\hat a))\left(\frac{\tilde s^2/R - \tilde a(1-\epsilon\hat a)\tilde s}{\tilde a(1-\epsilon\hat a)} \right)  + \mathcal{O}(\epsilon^2) 
\end{align}
Since $\tilde s,R > 0$ and $\tilde a<0$, the optimality condition, $\mathbb{E}[(x-\tilde x)\tilde x] = 0$, is satisfied if and only if $a = \tilde a(1-\epsilon\hat a) + \mathcal{O}(\epsilon^2)$ when $\sigma_X$ lies on the manifold of Theorem \ref{appthm1}.

We will now show that the optimal parameters yield a consistent filter in the sense that $\mathbb{E}[(x-\tilde x)^2] = \tilde s + \mathcal{O}(\epsilon^2)$, meaning that the actual error covariance equals the filter covariance estimate. The actual error covariance can be written as, $\mathbb{E}[(x-\tilde x)^2] = e_{11}+e_{33}-2e_{13}$.  Solving the third equation of \eqref{applyapFilter} for $e_{33}$ and substituting the solution into the actual error covariance we find that,
\begin{align} \mathbb{E}[(x-\tilde x)^2] &= e_{11}+e_{33}-2e_{13} = e_{11} - \frac{\tilde s^2}{2aR - 2\tilde s} - \frac{\tilde s e_{13}}{aR-\tilde s} - 2e_{13} \nonumber \\
&= e_{11} - \frac{\tilde s^2}{2aR - 2\tilde s} - e_{13} \frac{2aR-\tilde s}{aR-\tilde s} \nonumber \\
&= e_{11} - \frac{\tilde s^2}{2aR - 2\tilde s} + \frac{\tilde s e_{11}}{\tilde a(1-\epsilon\hat a)R+aR-\tilde s}\left( \frac{2aR-\tilde s}{aR-\tilde s}\right) + \mathcal{O}(\epsilon^2) \nonumber \\
&= - \frac{\tilde s^2}{2aR - 2\tilde s} + \frac{e_{11}}{aR-\tilde s}\left(aR -\tilde s+ \frac{2aR\tilde s - \tilde s^2}{R(a+\tilde a(1-\epsilon\hat a)) - \tilde s}\right) + \mathcal{O}(\epsilon^2).
\end{align} 
Since the Lyapunov equation is describing the steady state covariance, we may assume the $\tilde s$ has also reached steady state and thus solves the Riccati equation $0 = -\tilde s^2 + 2aR\tilde s + \sigma_X^2 R$.  Moreover, we can substitute the steady state covariance $e_{11} = c_{11}$ from \eqref{appc11soln}, abbreviating $\sigma_s^2 = \sigma_x^2(1-2\epsilon\hat{a})+\epsilon\sigma_y^2\frac{a_{12}^2}{a_{22}^2}$, to find that
\begin{align}\label{appactual error}
 \mathbb{E}[(x-\tilde x)^2] &= \frac{1}{2aR - 2\tilde s}\left(-\tilde s^2 - \frac{\sigma_s^2}{\tilde{a}(1-\epsilon\hat{a})}\left(aR -\tilde s+ \frac{-\sigma_X^2 R}{R(a+\tilde a(1-\epsilon\hat a)) - \tilde s}\right)\right) + \mathcal{O}(\epsilon^2) \nonumber \\
 &= \tilde s + \frac{1}{2aR - 2\tilde s}\left(\tilde s^2-2aR\tilde s - \frac{\sigma_s^2}{\tilde{a}(1-\epsilon\hat{a})}\left(\frac{R^2 a(a+\tilde a(1-\epsilon\hat a)) -\tilde a(1-\epsilon\hat a)\tilde sR}{R(a+\tilde a(1-\epsilon\hat a)) - \tilde s}\right)\right) + \mathcal{O}(\epsilon^2) \nonumber \\
 &= \tilde s + \frac{1}{2a - 2\tilde s/R}\left(\sigma_X^2- \sigma_s^2\left(\frac{R\frac{a^2}{\tilde{a}(1-\epsilon\hat{a})}+ Ra -\tilde s}{R\tilde a(1-\epsilon\hat a) + Ra - \tilde s}\right)\right) + \mathcal{O}(\epsilon^2).
 \end{align}
 We find the consistency condition, $\mathbb{E}[(x-\tilde x)^2] = \tilde s$ to be
  \[ \sigma_X^2= \sigma_s^2\left(\frac{\frac{a}{\tilde{a}(1-\epsilon\hat{a})}+ 1 -\tilde s/(aR)}{\frac{\tilde a(1-\epsilon\hat a)}{a} + 1 - \tilde s/(aR)}\right)+ \mathcal{O}(\epsilon^2). \]
 Clearly the optimal parameters $\{a,\sigma_X\} = \{\tilde a(1-\epsilon\hat a),\sigma_s\}$ satisfy this condition and therefore they yield a consistent filter.  
 

From the above proof, we have shown that the unique parameters can be determined by the manifold of Theorem \ref{appthm1} along with the optimality condition $\mathbb{E}[(x-\tilde x)\tilde x] = 0$.  These unique parameters can also be determined by matching the equilibrium covariance of the reduced model to that of the slow variable from the full model in \eqref{appfullmodelmatrix}.  From Lemma \ref{applem1} below, we have the covariance of the slow variable, $c_{11}$.  The reduced filter model in \eqref{appreducedmodel} has an asymptotic variance of $-\sigma_X^2/(2a)$. By requiring the two equilibrium variances to be equal, $c_{11} = -\sigma_X^2/(2a)$, we obtain a second constraint on the parameters $\{a,\sigma_X\}$ given by,
\begin{align}
\frac{\sigma_X^2}{a} =  \frac{\sigma_x^2(1-2\epsilon\hat{a})+\epsilon\sigma_y^2\frac{a_{12}^2}{a_{22}^2}}{\tilde{a}(1-\epsilon\hat{a})} + \mathcal{O}(\epsilon^2).\label{appconst2}
\end{align}
Substituting the relation \eqref{appconst2} into the manifold in Theorem 2.1, we obtain:
\[ \left(\sigma_x^2(1-2\epsilon\hat{a})+\epsilon\sigma_y^2\frac{a_{12}^2}{a_{22}^2}\right)\frac{a}{\tilde a(1-\epsilon\hat a)}= -2\hat s_{11}(a-\tilde a(1-\epsilon \hat a)) + \sigma_x^2(1-2\epsilon\hat{a})+\epsilon\sigma_y^2\frac{a_{12}^2}{a_{22}^2} + \mathcal{O}(\epsilon^2) \]
which we simplify to obtain,
\begin{align}
\left(\frac{a}{\tilde a(1-\epsilon\hat a)}-1 \right) \left(2\tilde a(1-\epsilon\hat a)\hat s_{11} + \sigma_x^2(1-2\epsilon\hat a) + \epsilon\sigma_y^2\frac{a_{12}^2}{a_{22}^2} \right) =\left(\frac{a}{\tilde a(1-\epsilon\hat a)}-1\right) \frac{\hat s_{11}^2}{R} = \mathcal{O}(\epsilon^2),
\end{align}
where the second equality is from the covariance expansion in \eqref{appfullmodelcov}. Since $\hat s_{11}$ and $R$ are both $\mathcal{O}(1)$, we have $a=\tilde{a}(1-\epsilon\hat{a})+\mathcal{O}(\epsilon^2)$ and from the equilibrium variance constraint in \eqref{appconst2}, we find $\sigma_X^2= \sigma_x^2(1-2\epsilon\hat{a})+\epsilon\sigma_y^2\frac{a_{12}^2}{a_{22}^2}+\mathcal{O}(\epsilon^2)$.  This shows that there are unique parameters $\{a,\sigma_X\}$ which match the covariance of the prior model and posterior estimate simultaneously up to order-$\epsilon^2$.
 \end{proof}
 
\noindent {\bf Remark}. A result of \cite{fz:11} shows that for $a=\tilde a + \mathcal{O}(\epsilon)$ and $\sigma_X^2 = \sigma_x^2 +\mathcal{O}(\epsilon)$, the reduced filter mean and covariance estimates are uniformly optimal for all time in the following sense. Given identical initial statistics, $\hat{x}(0)=\tilde{x}(0), \hat{s}_{11}(0)=\tilde{s}(0)>0$, there are time-independent constants $C$, such that, $\mathbb{E}\Big(|\hat{x}(t)-\tilde{x}(t)|^2\Big) \leq C\epsilon^2$. We conjecture that the pathwise convergence should be, \[\mathbb{E}\Big(|\hat{x}(t)-\tilde{x}(t)|^2\Big)\leq \mathcal{O}(\epsilon^4)\] for the unique parameters from Theorem \ref{appthm2}. This conjecture is based on numerical result shown in Figure \ref{appepsilon4} below.  However, the proof of this would require solving the Lyapunov equation of a five-dimensional joint evolution of the full model, full filter, and reduced filter, repeating the similar steps as in \eqref{appreducedfilterAnsatz}-\eqref{appe13soln}, for $\mathbb{E}(\hat x\tilde x), \mathbb{E}(\hat x^2),$ and $\mathbb{E}(\tilde x^2)$ up to $\mathcal{O}(\epsilon^4)$. Since this Lyapunov equation has 15 equations of 15 variables, it is not illuminating to verify our conjecture analytically.
 
 \begin{figure}[h]
\centering
\includegraphics[width=.45\textwidth]{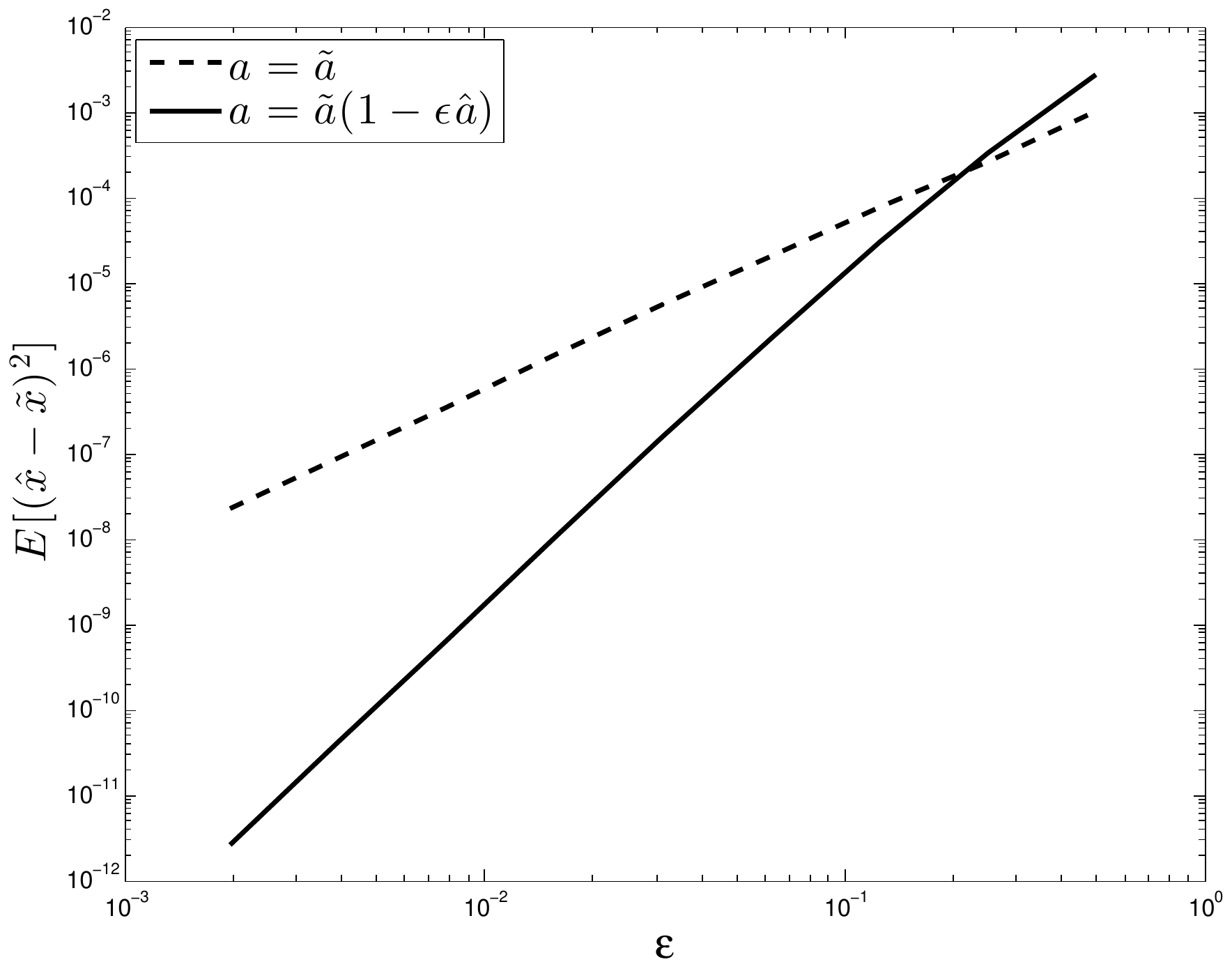}
\caption{\label{appepsilon4} We show the mean squared error between the true filter estimate $\hat x$ and the reduced filter estimate $\tilde x$ for two different parameter values $\{a,\sigma_X\}$ on the manifold of Theorem \ref{appthm1}.  Notice that when $a=\tilde a$ the convergence is order-$\epsilon^2$ whereas with the optimal parameter $a=\tilde a(1-\epsilon\hat a)$ the convergence appears to be order-$\epsilon^4$.  The parameters are $a_{11}=a_{12}=a_{22} =-1$, $a_{21}=R=\sigma_x=\sigma_y=1$}
\end{figure}

The following lemmas will establish the equilibrium covariance and correlation time of the slow variable, $x$, in \eqref{appfullmodelmatrix}.  These facts are used in Theorem \ref{appthm2} above as well as in the proof of a Corollary 2.3 to Theorem \ref{appthm2} which is in the manuscript.  
\begin{lem}\label{applem1} The equilibrium covariance of the slow variable, $x$, from the full model in \eqref{appfullmodelmatrix} is given by,
\begin{align}\label{appc11soln}
c_{11} = -\frac{\sigma_x^2(1-2\epsilon\hat{a})+\epsilon\sigma_y^2\frac{a_{12}^2}{a_{22}^2}}{2\tilde{a}(1-\epsilon\hat{a})} + \mathcal{O}(\epsilon^2).
\end{align}
\end{lem}
\begin{proof}
In order to find the optimal parameters for the reduced model in \eqref{appreducedmodel}, we will expand the Lyapunov equation which defines the covariance of the full model in \eqref{appfullmodelmatrix}.  Let $C=\lim_{t\to\infty}\mathbb{E}\left[\vec x(t) \vec x(t)^\top \right]$ be the equilibrium covariance matrix of the prior model in \eqref{appfullmodelmatrix}, so that $C$ is determined by the Lyapunov equation,
\begin{align}\label{appfullLyapunov}
A_\epsilon C + CA_\epsilon + Q_\epsilon = 0.
\end{align}
The Lyapunov equation \eqref{appfullLyapunov} yields the following three equations for the entries of $C$,
\begin{align}\label{appceq}
0 &= \sigma_x^2 + 2a_{11} c_{11} + 2a_{12} c_{12}, \nonumber \\
0 &= a_{11}c_{12} + a_{12}c_{22}  + c_{11}a_{21}/\epsilon + c_{12}a_{22}/\epsilon , \nonumber \\
0 & = \sigma_y^2/\epsilon + 2c_{12}a_{21}/\epsilon + 2c_{22}a_{22}/\epsilon.
\end{align}
Solving the third equation for $c_{22}$, and substituting to the second equation we can solve for $c_{12}$ as
\begin{align}\label{appc12} c_{12} = \frac{\sigma_y^2 \frac{a_{12}}{2a_{22}} - c_{11} a_{21}/\epsilon}{\tilde{a} + a_{22}/\epsilon}. \end{align}
and plugging this equation for $c_{12}$ into the first equation in \eqref{appceq} we can solve for $c_{11}$ as
\begin{align} c_{11} = \frac{-\sigma_x^2 - \sigma_y^2 \frac{a_{12}^2}{a_{22}\tilde a + a_{22}^2/\epsilon}}{2\left(a_{11} - \frac{a_{12}a_{21}}{a_{22} + \epsilon \tilde{a}}\right)} = \frac{-\sigma_x^2 - \sigma_y^2 \epsilon \frac{a_{12}^2}{a_{22}^2}}{2\tilde a(1+\epsilon \hat a)} + \mathcal{O}(\epsilon^2). \end{align}
Finally, by multiplying the numerator and denominator by $1-2\epsilon\hat a$ we obtain the desired result in \eqref{appc11soln}.
\end{proof}
In the next lemma we compute the correlation time of the full model in \eqref{appfullmodelmatrix} which is defined as, 
\[ T = \lim_{t\to\infty} \int_0^{\infty} \mathbb{E}\left[\vec x(t) \vec x(t+\tau)^\top\right] d\tau. \]
In particular we are interested in the correlation time of the slow variable which is given by, 
\[ T_c = \lim_{t\to\infty} \mathbb{E}[x(t)x(t+\tau)]/\mathbb{E}[x(t)x(t)], \]
however it is necessary to simplify the expression of $T$ as a matrix first, and we find $T_c$ in the following lemma.
\begin{lem}  For the prior model from \eqref{appfullmodelmatrix} the correlation time of the slow variable is,
\[ T_c = \frac{-1}{\tilde a(1-\epsilon\hat a)} + \mathcal{O}(\epsilon^2)\]
\end{lem}
\begin{proof}
Since $\vec x(t)$ is a stable two-dimensional Ohrnstein-Uhlenbeck process, we can write the solution $\vec x(t+\tau)$ as,
\[ \vec x(t+\tau) = e^{A_{\epsilon}\tau}\vec x(t) + \int_t^{t+\tau} e^{A_{\epsilon}(t+\tau-s)}\sqrt{Q_{\epsilon}} d\vec W_s. \]
Note that the stochastic integral is not correlated to $\vec x(t)$ so the expectation of this term will be zero.  Thus we only need to expand,
\[ T = \lim_{t\to\infty}\int_0^{\infty} \mathbb{E}\left[ \vec x(t) (e^{A_{\epsilon}\tau}\vec x(t))^\top \right] d\tau =  \lim_{t\to\infty}\mathbb{E}\left[\vec x(t)\vec x(t)^\top \left(\int_0^{\infty} e^{A_{\epsilon}\tau} d\tau \right)^\top \right] \]
We now write $A_{\epsilon}$ in its eigenvalue decomposition, $A_{\epsilon} = U_{\epsilon}\Lambda_{\epsilon} U_{\epsilon}^{-1}$ and note that 
\[ \int_0^{\infty} e^{A_{\epsilon}\tau}d\tau = U_{\epsilon}\int_0^{\infty} e^{\Lambda_{\epsilon}\tau} d\tau U_{\epsilon}^{-1} = -U_{\epsilon}\Lambda_{\epsilon}^{-1} U_{\epsilon}^{-1} = -A_{\epsilon}^{-1}. \]
This fact allows us to simplify,
\[ T = \lim_{t\to\infty} - E\left[\vec x(t)\vec x(t)^{\top}(A_{\epsilon}^{-1})^\top \right] = -CA_{\epsilon}^{-\top} = -\left(\begin{array}{cc} c_{11} & c_{12} \\ c_{12} & c_{22} \end{array}\right) \left(\begin{array}{cc} 1/\tilde a & -a_{21}/(a_{22}\tilde a)   \\  -\epsilon a_{12}/(a_{22}\tilde a) & a_{11}/(a_{22}\tilde a) \end{array}\right), \]
so we find that $T_c = T_{1,1}/c_{11}$ is given by,
\[ T_c = \frac{1}{c_{11}} \left(-c_{11}/\tilde a + \epsilon c_{12} a_{12}/(a_{22}\tilde a)\right). \]
Finally, substituting the relationship $c_{12} = -a_{21}c_{11}/a_{22} + \mathcal{O}(\epsilon)$ from \eqref{appc12} we have,
\begin{align}\label{appTc} T_c = -\frac{1+\epsilon\hat a}{\tilde a} + \mathcal{O}(\epsilon^2) = \frac{-1}{\tilde a(1-\epsilon\hat a)} + \mathcal{O}(\epsilon^2). \end{align}
as desired.
\end{proof}

\appendix{Asymptotic expansion of the nonlinear SPEKF}\label{appnonlineartheory}

Consider the following continuous-time nonlinear problem,
\begin{align}\label{appSPEKF}
d u &= [-(\gamma + \lambda_u)u + b] \, dt+ \sigma_u dW_u, \nonumber \\
d b &= -\frac{\lambda_b}{\epsilon} b \,dt + \frac{\sigma_b}{\sqrt{\epsilon}}dW_b, \\
d\gamma &= -\frac{\lambda_{\gamma}}{\epsilon} \gamma \,dt + \frac{ \sigma_{\gamma}}{\sqrt{\epsilon}} dW_{\gamma}, \nonumber\\
dz &= h(u,\beta,\gamma)\,dt +\sqrt{R}\,dV = u\,dt +\sqrt{R}\,dV,\nonumber
\end{align}
where $W_u, W_b, W_\gamma, V$ are standard i.i.d. Wiener processes. We will call this filtering problem SPEKF, which stands for ``Stochastic Parameterization Extended Kalman Filter", as introduced in \cite{ghm:10a,ghm:10b}. The posterior statistical solutions of SPEKF for discrete-time observations were obtained in \cite{ghm:10a,ghm:10b} by applying a Kalman update to the analytically solved prior mean and covariance of the stochastic model for $(u,b,\gamma)$ appearing in \eqref{appSPEKF}. To avoid confusion, we refrained from the common practice of calling the dynamical model in \eqref{appSPEKF} the SPEKF model. 

Notice that the SPEKF posterior solutions obtained in \cite{ghm:10a,ghm:10b} are not the optimal filtered solutions. The true posterior solutions for \eqref{appSPEKF}, given noisy observations, $z$, are characterized by the conditional distribution, $p(u,b,\gamma,t \,| \, z(\tau),0\leq \tau\leq t)$.  The evolution of $p$ is described by the Kushner equation \cite{kushner:64},
\[ dp = \mathcal{L}^*p\, dt + p(h-\mathbb{E}[h])^\top R^{-1}(dz-\mathbb{E}[h]dt), \]
where,
\[ \mathcal{L}^*p = -\nabla\cdot \left((-(\gamma+\lambda_u)u+b,\frac{-\lambda_b b}{\epsilon},\frac{-\lambda_{\gamma}\gamma}{\epsilon})^\top p\right) + \frac{1}{2}\left(\sigma_u^2 p_{uu} +\frac{\sigma_b^2 p_{bb}}{\epsilon} +  \frac{\sigma_{\gamma}^2p_{\gamma\gamma}}{\epsilon}\right), \]
is the Fokker-Planck operator.  For convenience we will write the innovation process as $dw_{\hat{u}} = dz - \mathbb{E}[h]dt$ which allows us to write the Kushner equation as,
\begin{align} dp = \mathcal{L}^*p\, dt + p(h-\mathbb{E}[h])^\top R^{-1}dw_{\hat{u}}. \label{appkushner}\end{align}
In order to make a formal asymptotic expansion in terms of the time scale separation $\epsilon$, we write the posterior as $p = p_0 + \epsilon p_1$.  Notice that the Fokker-Planck operator can be written as $\mathcal{L}^* = \frac{1}{\epsilon}\mathcal{L}_0^* + \mathcal{L}_1^*$ where
\begin{align} \mathcal{L}_0^* p &=  \frac{\partial}{\partial b}(\lambda_b b p) + \frac{\partial}{\partial\gamma}(\lambda_{\gamma}\gamma p) +\sigma_b^2 p_{bb} + \sigma_{\gamma}^2p_{\gamma\gamma}, \nonumber \\
 \mathcal{L}_1^*p &= \frac{\partial}{\partial u}\left((\gamma+\lambda_u)up+bp\right) + \sigma_u^2 p_{uu}. \nonumber
 \end{align}
With this expansion the Kushner equation becomes,
\[  dp_0 + \epsilon dp_1 = \frac{1}{\epsilon}\mathcal{L}_0^* p_0 + \mathcal{L}_0^* p_1 + \mathcal{L}_1^* p_0 + \epsilon \mathcal{L}_1^* p_1 + (p_0+ \epsilon p_1)(h-\mathbb{E}[h])^\top R^{-1}dw_{\hat{u}}.   \]
The order-$\epsilon^{-1}$ term requires that $\mathcal{L}_0^* p_0 = 0$ which says that $p_0$ is in the null space of the operator $\mathcal{L}_0^*$.  Since $b, \gamma$ in \eqref{appSPEKF} are ergodic processes, letting $p_{\infty}(b)$ and $p_{\infty}(\gamma)$ be the respective invariant measures, we can write,
\begin{align}
p_0(u,b,\gamma,t)=\tilde p(u,t)p_{\infty}(b)p_{\infty}(\gamma).\label{appp0}
\end{align}
We will use this fact repeatedly to complete the asymptotic expansion of the posterior distribution $p$. 

Note that convergence results of the marginal true filter distribution to the reduced filter characterized by $\tilde{p}$ are on the order of $\sqrt{\epsilon}$ for general nonlinear problems (see \cite{imkeller:13}). Here, we consider a higher order correction and we will show that for this specific example, we obtain convergence of order $\epsilon$ for the first two-moments. From the Kalman filtering perspective, we are only interested in capturing the first two moments of the posterior distribution $p$.  Using Ito's formula, we will compute the governing equations of the first two moments of the conditional distribution, $p$, which solves the Kushner equation in \eqref{appkushner}. Throughout this section we will assume that the posterior distribution, $p$, has fast decay at infinity allowing us to use integration by parts and neglect boundary terms.  To simplify the presentation, we define $\vec{u} = (u,b,\gamma)$, and all the integrals with respect to $d\vec u$ are three-dimensional integrals.

For the first moment we have, 
\[ d{\hat{u}} = \int u \, dp \,d\vec{u}, \] 
and substituting the Kushner equation we note that $up_{uu}$, $up_{bb}$ and $u p_{\gamma\gamma}$ integrate to zero leaving,
\begin{align} 
d{\hat{u}} &= \left(\int \big(u \frac{\partial}{\partial u}((\gamma+\lambda_u)up + bp) -  u \frac{\partial}{\partial b}(\lambda_b b p) + u\frac{\partial}{\partial\gamma}(\frac{\lambda_{\gamma}}{\epsilon}\gamma p)\big) d\vec{u}\right)dt + \int up(u-\hat{u})^\top R^{-1}dw_{\hat{u}} \, d\vec{u},\nonumber \\
\label{appmeanEq} &= \left(-\lambda_{u}\hat{u} - \int \gamma u p \, d\vec{u} + \int b p \, d\vec{u}\right)dt + \int (u^2-\hat{u} u)p R^{-1} dw_{\hat{u}} \, d\vec{u}, \nonumber \\
 &= \left(-\lambda_{u}\hat{u} - \int \gamma u p \, d\vec{u} + \overline{b} \right)dt+ \hat{S}R^{-1}dw_{\hat{u}}, \nonumber \\
&= \left(-\lambda_{u}\hat{u} - \overline{u\gamma} + \overline{b}\, \right) dt+ \hat{S}R^{-1}dw_{\hat{u}},
\end{align}
where we have used the fact that the innovation process $d w_{\hat{u}}$ is Brownian and uncorrelated with $\vec{u}$ \cite{bc:09}.  To estimate $\overline b = \int bp \, d\vec{u}$ we again apply the Kushner equation to compute,
\begin{align}\label{appbevol} d\overline b = \int b dp \,d\vec{u} &= \left(\int b \frac{\partial}{\partial b}\left(\frac{\lambda_b}{\epsilon} b p\right)\, d\vec{u} \right)dt + \int b(u-\overline u)p R^{-1} dw_{\hat{u}} \, d\vec{u}, \nonumber \\
&= \frac{\lambda_b}{\epsilon} \overline b\, dt + \int b(u-\overline u)p R^{-1} dw_{\hat{u}} \, d\vec{u}, \nonumber \\
&= \frac{\lambda_b}{\epsilon} \overline b\, dt + \mathcal{O}(\epsilon),
\end{align}
where the last equality comes from the expansion $p = p_0 + \epsilon p_1$ with $p_0$ satisfies \eqref{appp0}. Equation \eqref{appbevol} implies that, 
\[ \epsilon d \overline{b} = \lambda_b \overline b \, dt + \mathcal{O}(\epsilon^2), \]
which has solution $\overline b(t) = \overline b(0) e^{-\lambda_b t/\epsilon} + \mathcal{O}(\epsilon^2) \to \mathcal{O}(\epsilon^2)$ as $t\to\infty$.  Thus we can rewrite \eqref{appmeanEq} as
\begin{align}\label{appmeanEq2} d{\hat{u}} &= \left(-\lambda_{u}\hat{u} - \overline{u\gamma} \right) dt+ \hat{S}R^{-1}dw_{\hat{u}} + \mathcal{O}(\epsilon^2).
\end{align}
The term $\overline{u\gamma} = \int \gamma u p\,d\vec{u}$ represents an uncentered correlation between the two variables.  To find the evolution of $\overline{u\gamma}$ we again use the Kushner equation to expand
\begin{align} d \, \overline{u\gamma} &= \left(\int \big(u\gamma \frac{\partial}{\partial u}((\gamma+\lambda_u)up + bp) + u\gamma\frac{\partial}{\partial\gamma}(\frac{\lambda_{\gamma}\gamma}{\epsilon}p)\big) \, d\vec{u} \right)dt + \int u\gamma p(u-\hat{u})R^{-1}dw_{\hat{u}} \,d\vec{u}, \nonumber \\
&= \left(-(\lambda_u + \frac{\lambda_{\gamma}}{\epsilon})\overline{u\gamma} - \int \gamma^2up \, d\vec{u} + \int b\gamma p\, d\vec{u} \right)dt + \int u\gamma p(u-\hat{u})R^{-1}dw_{\hat{u}} d\vec{u}, \nonumber
\end{align}
where the second derivative terms $p_{uu}$ and $p_{\gamma\gamma}$ are both zero.  Applying the expansion $p = p_0 + \epsilon p_1$ with \eqref{appp0}, we can write the integral 
\[ \int \gamma^2u p\, d\vec{u} = \int \gamma^2 p_{\infty}(\gamma)d\gamma \int u \tilde p(u,t)du + \mathcal{O}(\epsilon) = \textup{var}_{\infty}[\gamma]\hat{u} + \mathcal{O}(\epsilon)  = \frac{\sigma_{\gamma}^2}{2\lambda_{\gamma}}\hat{u} + \mathcal{O}(\epsilon), \] 
similarly $\int b\gamma p\, d\vec{u} = \mathcal{O}(\epsilon)$ which gives us the expansion,
\begin{align} \frac{d}{dt} \overline{u\gamma} &= -(\lambda_u + \frac{\lambda_{\gamma}}{\epsilon})\overline{u\gamma} - \textup{var}_{\infty}[\gamma]\hat{u} + \mathcal{O}(\epsilon) = -(\lambda_u + \frac{\lambda_{\gamma}}{\epsilon})\overline{u\gamma} - \frac{\sigma_{\gamma}^2}{2\lambda_{\gamma}}\hat{u} + \mathcal{O}(\epsilon). \nonumber
\end{align}
Multiplying by $\epsilon$ we have
\[ \epsilon \frac{d}{dt} \overline{u\gamma} =  -(\lambda_u \epsilon + \lambda_{\gamma})\overline{u\gamma} - \epsilon\frac{\sigma_{\gamma}^2}{2\lambda_{\gamma}}\hat{u} + \mathcal{O}(\epsilon^2), \]
which has solution $\overline{u\gamma} = e^{-(\lambda_u+\lambda_{\gamma}/\epsilon)t}\overline{u\gamma}_0 - \frac{\epsilon\sigma_{\gamma}^2}{2\lambda_{\gamma}(\lambda_u\epsilon +\lambda_{\gamma})}\hat{u}(1-e^{-(\lambda_u+\lambda_{\gamma}/\epsilon)t}) + \mathcal{O}(\epsilon^2)$.  In the limit as $t\to\infty$ the correlation approaches a steady state $\overline{u\gamma}_{\infty} =  - \frac{\epsilon\sigma_{\gamma}^2}{2\lambda_{\gamma}(\lambda_u\epsilon +\lambda_{\gamma})}\hat{u} + \mathcal{O}(\epsilon^2)$.  Applying this result to \eqref{appmeanEq2} gives the following evolution for the mean state estimate
\begin{align} d{\hat{u}} = -\left(\lambda_{u} -  \frac{\epsilon\sigma_{\gamma}^2}{2\lambda_{\gamma}(\lambda_u\epsilon +\lambda_{\gamma})} \right)\hat{u}\, dt + \hat{S}R^{-1}dw_{\hat{u}} + \mathcal{O}(\epsilon^2). \label{appfirstmoment}\end{align}
By Ito's lemma we have
\[ d(\hat{u}^2) = \left(-2\left(\lambda_u -  \frac{\epsilon\sigma_{\gamma}^2}{2\lambda_{\gamma}(\lambda_u\epsilon +\lambda_{\gamma})} \right)\hat{u}^2 + \hat{S}R^{-1}\hat{S} \right) dt+ 2\hat{u}\hat{S}R^{-1}dw_{\hat{u}}. \]
Following the same procedure for the second moment we have
\begin{align}
d \hat{S} &= \int u^2 dp\, d\vec{u} - d(\hat{u}^2), \nonumber \\
&= \int u^2 \mathcal{L}^*p \,d\vec{u}+ \int u^2p(u-\hat{u})R^{-1}d w \,d\vec{u} + 2\left(\lambda_u -  \frac{\epsilon\sigma_{\gamma}^2}{2\lambda_{\gamma}(\lambda_u\epsilon +\lambda_{\gamma})} \right)\hat{u}^2 - \hat{S}R^{-1}\hat{S} - 2\hat{u}\hat{S}R^{-1}dw_{\hat{u}}. \nonumber
\end{align}
A straightforward computation shows that $\int u^2p(u-\hat{u})R^{-1}d w_{\hat{u}} d\vec{u} = \int (u-\hat u)^3 p R^{-1} dw_{\hat{u}} d\vec{u} + 2\hat{u}\hat{S}R^{-1}\,dw_{\hat{u}}$, so {\bf assuming the $p$ has zero skewness}, we have
\begin{align}\label{appSEq} 
\frac{d}{dt} \hat{S} &=  \int u^2 \mathcal{L}^*p d\vec{u} + 2\left(\lambda_u -  \frac{\epsilon\sigma_{\gamma}^2}{2\lambda_{\gamma}(\lambda_u\epsilon +\lambda_{\gamma})} \right)\hat{u}^2 - \hat{S}R^{-1}\hat{S}, \nonumber \\
&= \int u^2 \frac{\partial}{\partial u}((\gamma +\lambda_u)up + bp)\,d\vec{u} + \sigma_u^2 + 2\left(\lambda_u -  \frac{\epsilon\sigma_{\gamma}^2}{2\lambda_{\gamma}(\lambda_u\epsilon +\lambda_{\gamma})} \right)\hat{u}^2 - \hat{S}R^{-1}\hat{S}, \nonumber \\
&= -2 \int\big( (\gamma +\lambda_u)u^2 p - ub\big) \, d\vec{u} + \sigma_u^2 + 2\left(\lambda_u -  \frac{\epsilon\sigma_{\gamma}^2}{2\lambda_{\gamma}(\lambda_u\epsilon +\lambda_{\gamma})} \right)\hat{u}^2 - \hat{S}R^{-1}\hat{S}, \nonumber \\
&= -2\lambda_u \hat{S}   -2 \overline{u^2\gamma} + 2\overline{ub}  -  \frac{\epsilon\sigma_{\gamma}^2}{\lambda_{\gamma}(\lambda_u\epsilon +\lambda_{\gamma})}\hat{u}^2 + \sigma_u^2 - \hat{S}R^{-1}\hat{S}. 
\end{align}
Simplifying this expression requires finding $\overline{ub} = \int bu p\,d\vec{u}$ and $\overline{u^2\gamma} = \int u^2\gamma p \, d\vec{u}$.  First $\overline{ub}$ has evolution
\[ d\,\overline{ub} = -(\lambda_u \epsilon + \lambda_b)\overline{ub} + \epsilon\frac{\sigma_b^2}{2\lambda_b} + \mathcal{O}(\epsilon^2), \]
the solution of which approaches $\overline{ub} \to \overline{ub}_{\infty} = \frac{\epsilon\sigma_b^2}{2\lambda_b(\lambda_b + \lambda_u\epsilon)} + \mathcal{O}(\epsilon^2)$ as $t \to \infty$.  Second $\overline{u^2\gamma}$ has the following evolution
\begin{align} \frac{d}{dt}\overline{u^2\gamma} &= -(2\lambda_u + \frac{\lambda_{\gamma}}{\epsilon})\overline{u^2\gamma} - 2 \int u^2\gamma^2 p \,d\vec{u}, \nonumber \\
&= -(2\lambda_u + \frac{\lambda_{\gamma}}{\epsilon})\overline{u^2\gamma} -  \frac{ \sigma_{\gamma}^2}{\lambda_{\gamma}} \int u^2 \tilde{p} \,du + \mathcal{O}(\epsilon). 
\end{align}
Multiplying this expression by $\epsilon$ we find the steady state solution $\overline{u^2\gamma}_{\infty} = -\frac{\epsilon\sigma_{\gamma}^2}{\lambda_{\gamma}(2\lambda_u \epsilon + \lambda_{\gamma})}\int u^2 \tilde{p} du  + \mathcal{O}(\epsilon^2)$.  Substituting the expressions for $\overline{ub}$, $\overline{u\gamma}_{\infty}$, and $\overline{u^2\gamma}_{\infty}$ into \eqref{appSEq}, we find,
\begin{align}\label{appSEq2} \frac{d}{dt} \hat{S} &= -2\lambda_u \hat{S} +\sigma_u^2 +\frac{\epsilon\sigma_b^2}{\lambda_b(\lambda_b + \lambda_u\epsilon)}   - \hat{S}R^{-1}\hat{S} - \frac{\epsilon\sigma_{\gamma}^2}{\lambda_{\gamma}(\lambda_u\epsilon +\lambda_{\gamma})}\hat{u}^2 + 2\frac{\epsilon\sigma_{\gamma}^2}{ \lambda_{\gamma}(2\lambda_u \epsilon + \lambda_{\gamma})}\int u^2 \tilde{p} du  + \mathcal{O}(\epsilon^2), \nonumber \\
&= -2\left(\lambda_u - \frac{\epsilon\sigma_{\gamma}^2}{\lambda_{\gamma}(\lambda_u\epsilon +\lambda_{\gamma})} \right) \hat{S} + \frac{\epsilon\sigma_{\gamma}^2}{\lambda_{\gamma}(\lambda_u\epsilon +\lambda_{\gamma})}\hat{u}^2 + \sigma_u^2 + \frac{\epsilon\sigma_b^2}{\lambda_b(\lambda_b + \lambda_u\epsilon)}  - \hat{S}R^{-1}\hat{S} + \mathcal{O}(\epsilon^2).
\end{align}

Equations \eqref{appfirstmoment} and \eqref{appSEq2} describe the dynamics of the posterior mean and covariance estimates of the slow variable $u$ from \eqref{appSPEKF} up to order-$\epsilon^2$, assuming that the skewness is zero. We refer to the solutions of \eqref{appfirstmoment} and \eqref{appSEq2} as the {\bf continuous-time SPEKF solutions} for variable $u$. We do not find all of the cross-correlation statistics between the variables $(u,b,\gamma)$ since our goal is to find a one-dimensional reduced stochastic filter model for the $u$ variable. Motivated by the results in \cite{gh:13}, we propose the following reduced filter to approximate the filtering problem in \eqref{appSPEKF},
\begin{align}
\label{appreducedSPEKF} dU &= -\alpha U dt + \beta U \circ dW_\gamma + \sigma_1 dW_u + \sigma_2 dW_b,\nonumber \\ &= -\left(\alpha-\frac{\beta^2}{2}\right)U dt + \beta U dW_\gamma + \sigma_1 dW_u + \sigma_2 dW_b, \\
dz &= U\,dt + \sqrt{R}\,dV.\nonumber
\end{align}
Our goal is to write the evolution for the first two moments of the corresponding conditional distribution $\pi(U,t \,| \, z_\tau,0\leq\tau\leq t)$ for \eqref{appreducedSPEKF} and match coefficients with \eqref{appfirstmoment} and \eqref{appSEq2}.  The Fokker-Planck operator for \eqref{appreducedSPEKF} is $\mathcal{L}^*\pi = -\frac{d}{dU}(-\alpha U \pi) + \frac{1}{2}\frac{d^2}{dU^2}(\beta^2 U^2 \pi) + \frac{1}{2}(\sigma_1^2+\sigma_2^2)\frac{d^2\pi}{dU^2}$.  By differentiating the first moment $\tilde{u} = \int U\pi \, dU$ and substituting the Kushner equation we have,
\begin{align} 
d{\tilde{u}} &= -\left(\alpha-\frac{\beta^2}{2}\right)\tilde{u}\,dt +  \int U\pi(U-\tilde{u}) R^{-1}d w_{\tilde{u}} \, dU,\nonumber\\ &= -\left(\alpha-\frac{\beta^2}{2}\right)\tilde{u}\,dt + \tilde SR^{-1}dw_{\tilde{u}}, \label{appreducedmeanspekf}
\end{align}
where $dw_{\tilde{u}} = dz - \tilde{u}\,dt$ is the innovation process.  By Ito's formula, we have,
\[ \frac{d}{dt}(\tilde{u}^2) = -2\left(\alpha - \frac{\beta^2}{2}\right)\tilde{u}^2 + \tilde SR^{-1}\tilde S + 2\tilde{u} \tilde SR^{-1}dw_{\tilde{u}}. \]
For the second moment, $\tilde S = \int U^2 \pi \, dU - \tilde{u}^2$, we have,
\begin{align}\label{appUcov} \frac{d}{dt}\tilde S &= \int U^2 \left( \frac{d}{dU}\Big(\Big(\alpha-\frac{\beta^2}{2}\Big) U \pi \Big) + \frac{1}{2}\frac{d^2}{dU^2}(\beta^2 U^2 \pi) + \frac{\sigma_1^2+\sigma_2^2}{2}\frac{d^2\pi}{dU^2} \right) dU + 2\Big(\alpha-\frac{\beta^2}{2}\Big)\tilde{u}^2 - \tilde SR^{-1}\tilde S, \nonumber \\
 &= -\int 2U^2\Big(\alpha-\frac{\beta^2}{2}\Big) \pi \, dU + \int (\beta^2 U^2\pi + (\sigma_1^2+\sigma_2^2) \pi) dU + 2\Big(\alpha-\frac{\beta^2}{2}\Big)\tilde{u}^2 - \tilde SR^{-1}\tilde S, \nonumber \\
 &= -2\alpha \int U^2 \pi \, dU + \beta^2 \int U^2\pi \, dU + \beta^2 \int U^2\pi \, dU + \sigma_1^2+\sigma_2^2 + 2\Big(\alpha-\frac{\beta^2}{2}\Big)\tilde{u}^2 - \tilde SR^{-1}\tilde S, \nonumber \\
  &= -2(\alpha-\beta^2) \tilde S + \beta^2  \tilde{u}^2 + \sigma_1^2+\sigma_2^2 - \tilde SR^{-1}\tilde S, 
\end{align}
assuming the third moments are zero.
We can specify the coefficients in \eqref{appreducedSPEKF} by matching the two mean estimates in \eqref{appfirstmoment} and \eqref{appreducedmeanspekf} and the two covariance estimates, \eqref{appUcov} and \eqref{appSEq2}, in particular,
\begin{align}\label{appparams}
\alpha &= \lambda_u, \quad \sigma_1^2 = \sigma_u^2,\nonumber \\
\sigma_2^2 &= \frac{\epsilon\sigma_b^2}{\lambda_b(\lambda_b + \epsilon\lambda_u)}, \quad\beta^2 =  \frac{\epsilon\sigma_{\gamma}^2}{\lambda_{\gamma}(\lambda_u\epsilon +\lambda_{\gamma})}.
\end{align}

We refer to the reduced stochastic filter in \eqref{appreducedSPEKF}, with posterior mean and covariance estimates given by \eqref{appreducedmeanspekf} and \eqref{appUcov} with parameters \eqref{appparams}, as the {\bf continuous-time reduced SPEKF}. Notice that this reduced filter applies a Gaussian closure on the posterior solutions by truncating the third-order moments. This is different than the Gaussian Closure Filter (GCF) introduced in \cite{bgm:12}, which applies a Gaussian closure on the prior dynamics before using a Kalman update to obtain posterior solutions. 

Assume that the parameters in SPEKF and in the reduced SPEKF yield stable filtered solutions, that is, there are constants such that the posterior mean and covariance statistics are uniformly bounded. Then for identical initial statistics, $\hat{u}(0)=\tilde{u(0)}$ and $\hat{S}(0)=\tilde{S}(0)>0$, with the same argument as in the \cite{fz:11}, we can show that,
\begin{align}
|\hat{S}(t)-\tilde{S}(t)| &\leq C_1\epsilon,\nonumber\\
\mathbb{E}\left(|\hat{u}(t)-\tilde{u}(t)|^2\right)&\leq C_2 \epsilon^2.\nonumber
\end{align}
Notice that despite the formal asymptotic expansion of the filtered statistics up to order-$\epsilon^2$, we cannot obtain the same uniform bounds on the errors $|\hat{S}(t)-\tilde{S}(t)|$ and $|\hat x - \tilde x|$. This is because the covariance estimates no longer converge to a constant steady state due to the term $\beta^2\tilde u^2$ in \eqref{appUcov} which results from the multiplicative noise in \eqref{appreducedSPEKF}.  This also implies that we no longer have uniform bounds on the covariances without assuming that both filters are stable.

Finally, we show that the reduced SPEKF is consistent. Consider the actual error $e = u - \tilde u$, with evolution given by,
\[ de = du - d\tilde u = -\lambda_u e + \sigma_u dW_u - \tilde S R^{-1}dw_{\tilde{u}} + \mathcal{O}(\epsilon). \]   
The evolution of the actual error covariance, $E =\mathbb{E}[e^2]$, is given by the Lyapunov equation,
\begin{align} 
\frac{dE}{dt} &=  -2\lambda_uE + \sigma_u^2 - \tilde S R^{-1} \tilde S + \mathcal{O}(\epsilon), \nonumber
\end{align}
which implies that the difference between the actual error covariance and the filtered covariance estimate evolves according to,
\[\frac{d}{dt}(E-\tilde S) = -2 \lambda_u (E- \tilde S) +\mathcal{O}(\epsilon), 
\]
and by Gronwall's lemma, we have $|E-\tilde{S}| \leq \mathcal{O}(\epsilon)$, when $\lambda_u>0$. We summarize these results in the following theorem.

\begin{customthm}{4.1}\label{appthm4}
Let $\lambda_u > 0$, and let $z$ be noisy observations of the state variable $u$ which solves the full model in \eqref{appSPEKF}.  Given identical initial statistics, $\tilde{u}(0)=\hat{u}(0)$ and $\tilde{S}(0) = \hat S(0)>0$, the mean and covariance estimates of a stable continuous-time reduced SPEKF in \eqref{appreducedSPEKF} with parameters \eqref{appparams} agree with mean and covariance of a stable continuous-time SPEKF for variable $u$ in the following sense. There exist time-independent constants, $C_1, C_2$, such that,
\begin{align}
|\hat{S}(t)-\tilde{S}(t)| &\leq C_1\epsilon,\nonumber\\
\mathbb{E}\left[|\hat{u}(t)-\tilde{u}(t)|^2 \right]&\leq C_2 \epsilon^2.\nonumber
\end{align}
Furthermore, the reduced filtered solutions are also consistent, up to order-$\epsilon$.
\end{customthm}
We remark that if we did not assume that the filtered solutions are stable, then the constants $C_1,C_2$ may be time-dependent.


\appendix{A metric to determine the consistency of covariance estimates}\label{appconsistencySection}


In practical application with nonlinear systems, we often measure the accuracy of the filter mean estimate by computing the mean squared error, $\tilde{E} \equiv \langle (x-\tilde x)^2\rangle$, which is a temporal average of the square difference of a single realization of the signal $x(t)$ and the estimate $\tilde x(t)$ produced by a filter. For practical consideration, we would like to have a measure for the covariance estimate, analogous to $\tilde{E}$, that can be computed by temporal averaging over a single realization of the filtered solutions. This is not a trivial problem when the optimal filter mean, $\hat x(t)$, and covariance, $\hat S(t)$, are not available. Notice that for any fixed time $t$, the mean $\hat x(t)$ and covariance $\hat{S}(t)$ estimates of the optimal filter satisfy $\hat{S}(t) = \mathbb{E}[(x(t)-\hat x(t))(x(t)-\hat x(t))^\top]$, where the expectation is with respect to the solutions of the Kushner equation, $p(x(t)\,|\, z(\tau),0\leq \tau\leq t)$. Thus if $x$ is $n$-dimensional, at each $t$ we have
\[ \mathbb{E}[(x(t)-\hat x(t))^\top \hat{S}^{-1}(x(t)-\hat x(t))] = \textup{Trace}\left(\hat S(t)^{-1/2}\mathbb{E}[(x(t)-\hat x(t))(x(t)-\hat x(t))^\top]\hat S(t)^{-1/2}\right) = n. \]
In general, given any filtered statistical estimates, $\tilde x(t)$ and $\tilde S(t)$, we can define a norm on the state estimate error as follows,
\[ \|x(t)-\tilde x(t)\|^2_{\tilde S(t)} \equiv  \frac{1}{n}(x(t)-\tilde x(t))^\top \tilde S(t)^{-1} (x(t)-\tilde x(t)),\]
where $n$ is the dimension of the state vector $x$.  Assuming that the posterior distribution has an ergodic invariant measure, we have
\[ 1 = \mathbb{E}[\|x-\hat x\|^2_{\hat S}] = \langle \|x-\hat x\|^2_{\hat{S}}\rangle,\]
where $\langle \cdot\rangle$ denotes temporal average. Motivated by this property, we propose the following metric to check whether the filter covariance solutions are over- or under-estimating the actual error covariance.

\begin{defin}{Consistency (of Covariance).}
Let $\tilde x(t)$ and $\tilde S(t)$ be a realization of the solution to a filtering problem for which the true signal of the realization is $x(t)$.  The \emph{consistency} of the realization is defined to be, $\mathcal{C}(x, \tilde x, \tilde S)$ where
\begin{align}\label{appconsistency} \mathcal{C}(x,\tilde x,\tilde S) =  \langle\|x-\tilde x||^2_{\tilde S}\rangle.\end{align}
We say that a filter is \emph{consistent} if $\mathcal{C} = 1$ almost surely (independent of the realization). The filter covariance under(over)estimates the actual error covariance when $\mathcal{C}>1$ ($\mathcal{C}<1$).
\end{defin}

This metric is nothing else but the signal part of the relative entropy measure of two Gaussian distributions \cite{bm:14}. With this definition, it is obvious that the true filter is consistent. However it is not the only consistent filter and not every consistent filter is accurate. For example, in the case of fully observed filtering problems, the observation sequence itself is a trivial filtered solution which is consistent if we take the covariance estimate to be the covariance of the observation noise.  Second, assuming the dynamical system is ergodic, we can define a consistent filter based purely on the equilibrium statistics by using a constant prediction $\tilde x = \mathbb{E}[x(t)]$ and covariance $\tilde S = \mathbb{E}[(x(t)-\tilde x(t))(x(t)-\tilde x(t))^\top]$.  These two examples are the trivial extremes of filtering, the first simply takes the observations as solutions, while the second completely ignores the observations.  However, while these examples are both consistent, neither is doing a good job of estimating the state compared to the true filter.
Therefore, this consistency measure is only a necessary condition for the covariance to be meaningful. It should be used together with the mean squared error measure. We should point out that although this measure is much weaker than the pattern correlation measure advocated in \cite{bm:14}, many suboptimal filters are not even consistent in the sense of Definition 3.1 as  shown in many examples in Sections~3 and 4 in the manuscript.

However, this measure has the following nice property: a consistent filter which produces posterior mean estimates close to the true posterior mean estimates also has a covariance close to the true posterior covariance, which we will verify now. As above, let $\hat x(t)$ and $\hat S(t)$ be the posterior statistical estimates of the true filter and $\tilde x(t)$ and $\tilde S(t)$ be the estimates from a suboptimal filter. For convenience we will drop the letter $t$ when no confusion is possible.  Then, assuming $ \mathcal{C}(x,\tilde x,\tilde S)$ exists, we can write
\begin{align} \mathcal{C}(x,\tilde x,\tilde S) &= \frac{1}{n}\langle (x-\hat x+\hat x-\tilde x)^\top \tilde S^{-1} (x-\hat x+\hat x-\tilde x)\rangle, \nonumber \\
&= \langle\|x-\hat x\|^2_{\tilde S}\rangle + \langle\|\tilde x-\hat x\|^2_{\tilde S}\rangle + \frac{2}{n}\langle (x-\hat x)^\top \tilde S^{-1} (\hat x-\tilde x)\rangle, \nonumber \\
&= \mathcal{C}(x,\hat x, \tilde S) +  \langle\|\tilde x-\hat x\|^2_{\tilde S}\rangle + \frac{2}{n}\langle (x-\hat x)^\top \tilde S^{-1} (\hat x-\tilde x)\rangle \label{appcc}. 
\end{align}
Note also that,
\begin{align} \left| \mathcal{C}(x,\hat x,\hat S) - \mathcal{C}(x,\hat x, \tilde S) \right| &\leq \left| \mathcal{C}(x,\hat x,\hat S) - \mathcal{C}(x,\tilde x, \tilde S) \right| + \left| \mathcal{C}(x,\tilde x, \tilde S) - \mathcal{C}(x,\hat x, \tilde S) \right|, \nonumber \\ &= \left|1 - \mathcal{C}(x,\tilde x, \tilde S) \right| + \langle\|\tilde x-\hat x\|^2_{\tilde S}\rangle + \frac{2}{n}\langle (x-\hat x)^\top \tilde S^{-1} (\hat x-\tilde x)\rangle,\nonumber \\
&\leq  \left|1 - \mathcal{C}(x,\tilde x, \tilde S) \right| + c \, \langle\|\tilde x-\hat x\|^2\rangle. \label{apppptyconsistent}
\end{align}
where the first line is due to the standard triangle inequality, the equality in the second line is based on the fact that the true filter is consistent, $\mathcal{C}(x,\hat x,\hat S)=1$ and the algebraic expression in \eqref{appcc}, and the inequality in the third line is due to Cauchy-Schwarz inequality and the constant $c$ depends on the smallest eigenvalue of $\tilde{S}$ and $\langle \|x-\hat{x}\|^2\rangle$. 

The inequality in \eqref{apppptyconsistent} suggests that if the state estimate $\tilde x$ is close to the true posterior estimate $\hat x$ and the consistency measure is close to one, then the covariance estimate $\tilde S$ is close to the true posterior covariance $\hat S$ in the sense that 
\[ \left| \mathcal{C}(x,\hat x,\tilde S) - \mathcal{C}(x,\hat x, \hat S) \right| = \frac{1}{n}\langle (x-\hat x)^T (\tilde S^{-1} - \hat S^{-1} ) (x-\hat x) \rangle, \]
is small. Thus, a consistent filter with a good estimate of the posterior mean has a good estimate of the posterior covariance.  In practice, many approximate filter mean estimates are quite accurate, in the sense that they are close to the true posterior estimate \cite{ls:12}. Therefore the consistency measure on the approximate filter solutions, $\mathcal{C}(x,\tilde x, \tilde S)$, is relevant for quantifying the skill of $\tilde S$, when the true filter covariance estimate $\hat S$ is not available.


\appendix{Online stochastic parameter estimation}\label{appparamestimation}

In this appendix we overview the method of \cite{bs:13} for fitting an additive noise covariance matrix as part of the filtering algorithm. Let $\vec{x}_k = (x_1(t_k),\ldots,x_N(t_k))^\top\in\mathbb{R}^N$ be the slow components of the solutions of the two-layer Lorenz-96 model in (4.1), given initial conditions $\vec{x}_0$. The goal of the online stochastic parameterization method in Section~4 is to determine parameters $\alpha,  Q=\sigma \sigma^{\top},  R$, of the following reduced stochastic filtering problem,
\begin{align}\label{applor96reduced}
\frac{dx_i}{dt} &= x_{i-1}(x_{i+1}-x_{i-2}) - x_i + F + \Big(- \alpha x_i + \sum_{j=1}^N \sigma_{ij} dw_j\Big),\quad i=1,\ldots, N, \\
\vec{z}_k&=h(\vec{x}_k) + \vec{\xi}_k, \quad\vec{\xi}_k \sim \mathcal{N}(0, R),  \nonumber
\end{align}
where the noisy observations $\vec{z}_n\in\mathbb{R}^M$ are defined at discrete time $t_k$ with time interval, $t_{k+1}-t_k = \Delta t$ and $M \leq N$. In our implementation, we set the observation time interval to be a multiple of the integration time, $\delta t=0.005$, and the Runge-Kutta scheme is simply iterated to solve the deterministic part of the model in \eqref{applor96reduced}. We will call $Q_k$ and $R_k$ the estimates of $ Q$ and $ R$ respectively and we initialized $Q_1=0$ and $R_1 =  R$. The covariance matrices $Q_k$ and $R_k$ are used in the ensemble transform Kalman filter procedure to inflate the ensemble and define the Kalman gain matrix respectively.  

The online parameterization algorithm consists of two parts, where we first apply any choice of ensemble Kalman filter algorithm to update the joint state-parameter estimate for $(\vec{x},\alpha)$, assuming persistent model for the parameter, $d\alpha/dt=0$. Subsequently, we use the innovation, $\epsilon_k = \vec{z}_k-h(\bar{x}^f_k)$ given by the difference between the observations and the mean prior estimate, $h(\bar{x}^f_k)$, projected in the observation subspace, to update the estimates $Q_k$ and $R_k$ of the noise covariances $Q$ and $R$ using the procedure below.  

The procedure uses linearizations of the dynamics and the observation function which are computed from the various ensembles used by the EnKF.  Explicitly, let $x_{k-1}^{a,i} \sim \mathcal{N}(\overline{x}_{k-1}^a,P_{k-1}^a)$ be the analysis ensemble at step $k-1$, where the index $i = 1,...,E$ indicates the ensemble member, and let $x_k^{f,i} = \mathcal{F}(x_{k-1}^{a,i})$ be the forecast ensemble which results from integrating the deterministic part of the model \eqref{applor96reduced} from $t_{k-1}$ to $t_{k}$ with initial condition $x_{k-1}^{a,i}$.  Then, letting $\overline{x}_k^f = \frac{1}{E}\sum_{i=1}^E x_{k}^{f,i}$, we define the matrix of analysis perturbations, $X_{k-1}^a$, and the matrix of forecast perturbations, $X_{k}^f$, by,
\begin{align} X_{k-1}^a &= \left[(x_{k-1}^{a,1}-\overline{x}_{k-1}^a)^\top,...,(x_{k-1}^{a,E}-\overline{x}_{k-1}^a)^\top \right] \nonumber \\
X_{k}^f &= \left[(x_{k}^{f,1}-\overline{x}_{k}^f)^\top,...,(x_{k}^{f,E}-\overline{x}_{k}^f)^\top \right].
\end{align}
Then we can define $F_k = X_k^f (X_{k-1}^a)^{\dagger}$, where $\dagger$ denotes the matrix pseudo-inverse, which we can think of as being a local linearization of the deterministic forward operator $\mathcal{F}$.  Similarly, let $\tilde x_{k}^{f,i} \sim \mathcal{N}(\overline{x}_{k}^f,P_{k}^f + Q)$ be the inflated forecast ensemble and let $z_k^{f,i} = h(\tilde x_{k}^{f,i})$ be the projection of this ensemble into the observation space.  Then we can define $H_k = Z_k^f(\tilde X_k^f)^{\dagger}$, which we think of as a local linearization of the observation function $h$, where, 
\begin{align} \tilde X_k^f &=  \left[(\tilde x_{k}^{f,1}-\overline{x}_{k}^f)^\top,...,(\tilde x_{k}^{f,E}-\overline{x}_{k}^f)^\top \right] \nonumber \\
Z_k^f &=  \left[(z_{k}^{f,1}-\overline{z}_{k}^f)^\top,...,(z_{k}^{f,E}-\overline{z}_{k}^f)^\top \right]
\end{align}
are the matrix of inflated forecast perturbations the matrix of observed forecast perturbations respectively, and where $\overline{z}_k^f = \frac{1}{E}\sum_{i=1}^E z_k^{f,i}$.

\begin{enumerate} 
\item Produce empirical estimates $Q^e_{k-1}$ and $R^e_{k-1}$ of $Q$ and $R$ based on the innovations at time $k$ and $k-1$ using the formula of \cite{bs:13},
\begin{align}\label{appqrestimates}
P_{k-1}^e &= F_{k-1}^{-1}H_k^{-1}\epsilon_k\epsilon_{k-1}^\top H_{k-1}^{-\top} + K_{k-1}\epsilon_{k-1}\epsilon_{k-1}^\top H_{k-1}^{-\top}  \nonumber \\
Q_{k-1}^e &=  P_{k-1}^e - F_{k-2}P_{k-2}^aF_{k-2}^\top \nonumber \\
R_{k-1}^e &= \epsilon_{k-1}\epsilon_{k-1}^\top - H_{k-1}P_{k-1}^fH_{k-1}^\top
\end{align}
where $F_k$ and $H_{k-1},H_k$ are linearizations of the dynamics and observation function estimated from the ensembles as described above.  It was shown in \cite{bs:13} that $P^e_{k-1}$ is an empirical estimate of the background covariance at the previous step which motivates the index $k-1$.  Notice that this procedure requires us to save the linearizations $F_{k-2},F_{k-1}, H_{k-1}, H_k$, innovations $\epsilon_{k-1}, \epsilon_k$, and the analysis $P_{k-2}^a$ and Kalman gain matrix, $K_{k-1}$, from the $k-1$ and $k-2$ steps of the EnKF.
\item The estimates $Q^e_{k-1}$ and $R^e_{k-1}$ are low-rank, noisy estimates of the parameters $ Q$ and $ R$ which will make the posterior estimate statistics from the filter consistent with the empirical statistics in the sense of Equation \eqref{appqrestimates}.  In order to form stable full-rank estimates of $ Q$ and $ R$ we assimilate these estimates using an exponential moving average with window $\tau$,
\begin{align}\label{appqrupdate}
Q_{k} &= Q_{k-1}+ (Q^e_{k-1}- Q_{k-1})/\tau \nonumber \\
R_{k} &= R_{k-1} + (R^e_{k-1} - R_{k-1})/\tau. 
\end{align}
\end{enumerate}

We interpret the moving average in \eqref{appqrupdate} as a simplistic filter which gives stable estimates $Q_k$ and $R_k$ of $Q$ and $R$ from the noisy empirical estimate.  The stochastic nature of the estimate of $Q_k$ can lead to temporary excursions which are not symmetric and/or positive definite, which can lead to instability in the EnKF.  Thus, while the matrix $Q_k$ is not changed, the matrix used in the $k$-th step of the filter is a modified version of $Q_k$ which is forced to be symmetric and positive definite by taking $\tilde Q_k = (Q_k + Q_k^\top)/2$ and then taking the max of the eigenvalues of $\tilde Q_k$ with zero.  Again, we emphasize that $\tilde Q_k$ is only used in the $k$-th filter step and no ad-hoc corrections are made to the matrix $Q_k$ which eventually stabilizes at a symmetric and positive definite matrix naturally via the moving average in \eqref{appqrupdate}.  These ad-hoc corrections are only needed during the transient period of the estimation of $Q$, especially when we initialize $Q_1=0$.  Note that the work of \cite{hmm:14} provides an attractive alternative where a true secondary filter is applied to the estimates $Q^e$ and $R^e$, however this technique has not yet been developed for high-dimensional systems.

Notice that the number of parameters in $ Q$  is large ($N^2=81$ parameters in Section 4(a)) and accurate recovery requires long simulations.  Moreover, for sparse observations, $M< N$, there are observability problems when trying to estimate the full matrix $ Q$, in particular the matrices $H_k$ and $H_{k-1}$ are not invertible as required in \eqref{appqrestimates}. In order to reduce the required simulation time and the required dimension of the observations, we can use the parameterization scheme introduced in \cite{bs:13} to parameterize the $ Q$ matrix. Since the Lorenz96 system is spatially homogeneous, we introduce a cyclic parameterization of $ Q$.  If $ Q$ is $N\times N$ then set $\tilde{M}=\textup{ceil}(N/2)$ and write $ Q = \sum_{r=1}^{\tilde{M}} q_r \hat Q_r$ where $q_r$ are scalar parameters and 
\[ (\hat Q_r)_{ij} = \left\{ \begin{array}{ll} 1 & \rm{if} \hspace{5pt} i = j+r \hspace{5pt}\rm{or}\hspace{5pt} j = i+r \\ 0 & \rm{else} \end{array}\right. \]
where the sums $j+r$ and $i+r$ are cyclic so we implicitly subtract $N$ from any sum greater than the maximum index $N$.  Now following \cite{bs:13} we combine the first two equations of \eqref{appqrestimates} in order to remove the matrix inversions (which may be undefined for sparse observations) which gives us,
\[ H_kF_{k-1}Q_{k-1}^e H_{k-1}^\top = \epsilon_k\epsilon_{k-1}^\top + H_k F_{k-1}K_{k-1}\epsilon_{k-1}\epsilon_{k-1}^\top + H_kF_{k-1}F_{k-2}P_{k-2}^a F_{k-2}^\top H_{k-1}^\top \]
and we set $C_k$ equal to the right hand side, which is a $M \times M$ matrix that we can compute after the $k$-th filter step.  Now we introduce our parameterization $Q_{k-1}^e = \sum_{r=1}^{\tilde{M}} q_r \hat Q_r$ into the left hand side which gives us,
\begin{align}\label{apptovectorize} C_k = \sum_{r=1}^{\tilde{M}} q_r H_k F_{k-1} \hat Q_r H_{k-1}^\top. \end{align}
Since the terms $H_k F_{k-1} \hat Q_r H_{k-1}^\top$ can each be computed from the assumed form of $\hat Q_r$, we can vectorize \eqref{apptovectorize} by letting $\textup{vec}(C_k)$ be the $M^2 \times 1$ vector formed by concatenating the columns of $C_k$ and letting $A_k$ be the $M^2 \times \tilde{M}$ matrix where the $r$-th column is given by $\textup{vec}(H_k F_{k-1} \hat Q_r H_{k-1}^\top)$.  Letting $\vec{q} = (q_1,q_2,...,q_{\tilde{M}})^\top$ we can rewrite \eqref{apptovectorize} as $\textup{vec}(C_k) = A_k \vec q_k$ and we can now solve for $\vec q$ by least squares and then set $Q_{k-1}^e = \sum_{r=1}^{\tilde{M}} q_r \hat Q_r$.  Finally, we use \eqref{appqrupdate} to update $Q_k$ as usual.  Notice that we essentially need to invert the matrix $A_k$ which is $M^2 \times \tilde{M}$ so we typically want the number of parameters $\tilde{M}$ to be less than the square of the observation dimension, $M^2$.  For the cyclic parameterization, the number of parameters is $\tilde{M} = \textup{ceil}(N/2)$ so we need at least $\sqrt{N/2}$ observations for this parameterization to work well.  

In Section 4(b) we assume all $N=M=9$ of the slow variables are observed which allows us to estimate the full $9\times 9$ matrix $ Q$ which corresponds to estimating $N^2=M^2=81$ parameters using equations \eqref{appqrestimates}.  Because of the large number of parameters, we used a long averaging window of $\tau=5000$ and the filter was run for $80000$ observations with the RMSE and consistency averages only using the last $60000$ observations which gave the secondary filter $20000$ filter steps to converge to a stable estimate of $ Q$. 

In Section 4(c) we observe $M=4$ of the $N=8$ slow variables so we cannot estimate the full matrix $ Q$.  Thus we use the cyclic parameterization and since $M^2=16$ and we are only estimating $\tilde M = 4$ parameters this produced stable results.  Since the number of parameters $\tilde M$ was small we were able to use a smaller $\tau=1500$ and each filter was run for $10000$ observations and only the last $7000$ were used for RMSE averages, thereby allowing $3000$ filter steps for the estimate of $ Q$ to converge.



\bibliography{ref}
\bibliographystyle{plain}

\end{document}